\documentclass[12pt]{amsart}
\usepackage{amsthm,amsfonts,amssymb,amsmath,amscd,enumitem,
}
\usepackage{textcomp}
\usepackage{comment} 
\usepackage{xcolor}
\usepackage{longtable}
\usepackage{url}
\usepackage[all,2cell, cmtip]{xy}
\xyoption{2cell}
\includecomment{comment}
\specialcomment{AndreNotes}{\begingroup\sffamily\color{blue}}{\endgroup}
\specialcomment{AmitNotes}{\begingroup\sffamily\color{red}}{\endgroup}
 \begin{document}

\title[Symmetric monoidal categories and $\Gamma$-categories]{Symmetric monoidal categories and $\Gamma$-categories}
\author[A. Sharma]{Amit Sharma}

\email{asharm24@kent.edu}
\address {Department of mathematical sciences\\ Kent State university\\
  Kent, OH}

\date{November 15, 2018}
%

\newcommand{\CONT}{\noindent}
\newcommand{\FIG}{Fig.\ }
\newcommand{\FIGS}{Figs.\ }
\newcommand{\SEC}{Sec.\ }
\newcommand{\SECS}{Secs.\ }
\newcommand{\TAB}{Table }
\newcommand{\TABS}{Tables }
\newcommand{\EQ}{Eq.\ }
\newcommand{\EQS}{Eqs.\ }
\newcommand{\APP}{Appendix }
\newcommand{\APPS}{Appendices }
\newcommand{\CHP}{Chapter }
\newcommand{\CHPS}{Chapters }

\newcommand{\OFF}{\emph{G2off}~}
\newcommand{\TOO}{\emph{G2Too}~}
\newcommand{\CatS}{Cat_{\bigS}}
\newcommand{\PCat}{\mathbf{Perm}}
\newcommand{\PCatOL}{\mathbf{Perm}_{\textit{OL}}}
\newcommand{\PCatSM}{\mathbf{Perm}_{\otimes}}
\newcommand{\mdlPCat}{\left( \PCat, \mathbf{E} \right)}
\newcommand{\PicS}{{\underline{\pic}}^{\oplus}}
\newcommand{\HPicS}{{Hom^{\oplus}_{\pic}}}
\newcommand{\HPerm}{{Hom^{\oplus}_{\PCat}}}
\newcommand{\HopL}{{Hom^{\text{un-opLax}}_{\PCat}}}

\newtheorem{thm}{Theorem}[section]
\newtheorem{lem}[thm]{Lemma}
\newtheorem{conj}[thm]{Conjecture}
\newtheorem{coro}[thm]{Corollary}
\newtheorem{prop}[thm]{Proposition}

\theoremstyle{definition}
\newtheorem{df}[thm]{Definition}
\newtheorem{nota}[thm]{Notation}

\newtheorem{ex}[thm]{Example}
\newtheorem{exs}[thm]{Examples}

\theoremstyle{remark}
\newtheorem{rem}{Remark}
\newtheorem*{note}{Note}

\newtheorem{ack}{Acknowledgments}

\newcommand{\ChI}{{\textit{\v C}}\textit{ech}}
\newcommand{\Ch}{{\v C}ech}

\newcommand{\ChZG}{hermitian line $0$-gerbe}
\renewcommand{\theack}{$\! \! \!$}

\newcommand{\Leins}{\mathfrak{L}}
\newcommand{\ChG}{flat hermitian line $1$-gerbe}
\newcommand{\ChC}{hermitian line $1$-cocycle}
\newcommand{\ChGG}{flat hermitian line $2$-gerbe}
\newcommand{\ChCC}{hermitian line $2$-cocycle}
\newcommand{\id}{id}
\newcommand{\SigA}{\Sigma^\A}
\newcommand{\totSigInf}[1]{\mathbb{L}\Sigma^{\infty}{#1}}
\newcommand{\SigAvec}[1]{\Sigma^\A_{\length{#1}}}
\newcommand{\length}[1]{\mid \vec{#1} \mid}
\newcommand{\underMap}[1]{\mid #1 \mid}
\newcommand{\bProd}{\text{box product}}
\newcommand{\EPs}{\E^{\textit{Ps}}}
\newcommand{\ePs}{\epsilon^{\textit{Ps}}}
\newcommand{\EStr}{\E_{\textit{str}}}
\newcommand{\Pnor}{\Lbb}
\newcommand{\elR}{\textit{el}^{R}(\gn{n}){\mid}_{\N}}
\newcommand{\elRGen}[1]{\textit{el}^{R}(\gn{#1}){\mid}_{\N}}
\newcommand{\elRGeneral}[1]{\textit{el}^{R}(#1){\mid}_{\N}}
\newcommand{\elGen}[2]{\textit{el}^{#1}{#2}{\mid}_{\N}}
\newcommand{\elKbar}{\textit{el}^{\Kbar}(\gn{n}){\mid}_{\N}}
\newcommand{\unit}[1]{\mathrm{1}_{#1}}
\newcommand{\bike}{\text{bicycle}}
\newcommand{\bikes}{\text{bicycles}}
\newcommand{\psbike}{\text{pseudo bicycle}}
\newcommand{\psbikes}{\text{pseudo bicycles}}
\newcommand{\PsBikes}[2]{\textbf{Bikes}^{\textit{Ps}}(#1,#2)}
\newcommand{\strbike}{\text{strict bicycle}}
\newcommand{\strbikes}{\text{strict bicycles}}
\newcommand{\StrBikes}[2]{\textbf{Bikes}^{\textit{Str}}(#1,#2)}
\newcommand{\Bikes}[2]{\textbf{Bikes}^{\textit{Ps}}(#1,#2)}
\newcommand{\NLC}[2]{\textbf{NorLaxCones}(#1,#2)}
\newcommand{\LCns}[2]{\textbf{LaxCones}(#1,#2)}
\newcommand{\PsCns}[2]{\textbf{PsCones}(#1,#2)}
\newcommand{\LC}{\mathfrak{C}}
\newcommand{\Coker}{Coker}
\newcommand{\Com}{Com}
\newcommand{\Hom}{Hom}
\newcommand{\Mor}{Mor}
\newcommand{\Map}{Map}
\newcommand{\alg}{alg}
\newcommand{\an}{an}
\newcommand{\Ker}{Ker}
\newcommand{\Ob}{Ob}
\newcommand{\SymMon}{\mathbf{SymMon}}
\newcommand{\Kbar}{\overline{\K}}
\newcommand{\KSeg}{\K}
\newcommand{\KSegLax}{{\Kbb}}
\newcommand{\Proj}{\mathbf{Proj}}
\newcommand{\topo}{\mathbf{Top}}
\newcommand{\inrt}{\mathbf{Inrt}}
\newcommand{\act}{\mathbf{Act}}
\newcommand{\n}{\underline{n}}
\newcommand{\pn}{\underline{n}^+}
\newcommand{\kan}{\mathcal{K}}
\newcommand{\pkan}{\mathcal{K}_\bullet}
\newcommand{\Kan}{\mathbf{Kan}}
\newcommand{\pKan}{\mathbf{Kan}_\bullet}
\newcommand{\gp}{\mathcal{A}_\infty}
\newcommand{\mdl}{\mathcal{M}\textit{odel}}
\newcommand{\sSets}{\mathbf{sSets}}
\newcommand{\sSetsQ}{(\mathbf{sSets, Q})}
\newcommand{\sSetsK}{(\mathbf{sSets, \Kan})}
\newcommand{\pSSets}{\mathbf{sSets}_\bullet}
\newcommand{\pSSetsK}{(\mathbf{sSets}_\bullet, \Kan)}
\newcommand{\pSSetsQ}{(\mathbf{sSets_\bullet, Q})}
\newcommand{\Sets}{\mathbf{Sets}}
\newcommand{\cyl}{\mathbf{Cyl}}
\newcommand{\lin}{\mathcal{L}_\infty}
\newcommand{\Vect}{\mathbf{Vect}}
\newcommand{\Aut}{Aut}
\newcommand{\Ein}{E_\infty}
\newcommand{\EinS}{E_\infty{\text{- space}}}
\newcommand{\EinSs}{E_\infty{\text{- spaces}}}
\newcommand{\EinC}{\text{coherently commutative monoidal category}}
\newcommand{\EinCs}{\text{coherently commutative monoidal categories}}
\newcommand{\EinLO}{E_\infty{\text{- local object}}}
\newcommand{\EinSLO}{\E_\infty\S{\text{- local object}}}
\newcommand{\pic}{\mathcal{P}\textit{ic}}
\newcommand{\Dlin}{\pic}
\newcommand{\Gmn}{\Gamma \left(m_n \right)}
\newcommand{\bigS}{\mathbf{S}}
\newcommand{\bigA}{\mathbf{A}}
\newcommand{\bhom}{\mathbf{hom}}
\newcommand{\bHom}[3]{\mathbf{hom}_{#3}(#1, #2)}
\newcommand{\bhomK}{\mathbf{hom}({\textit{K}}^+,\textit{-})}
\newcommand{\Bhom}{\mathbf{Hom}}
\newcommand{\bhomk}{\mathbf{hom}^{{\textit{k}}^+}}
\newcommand{\CatHom}[3]{[#1,#2]^{#3}}
\newcommand{\pCatHom}[3]{[#1,#2]_\bullet^{#3}}
\newcommand{\pHomCat}[2]{[#1,#2]_{\bullet}}
\newcommand{\Dlino}{\pic^{\textit{op}}}
\newcommand{\lino}{\mathcal{L}^{\textit{op}}_\infty}
\newcommand{\lind}{\mathcal{L}^\delta_\infty}
\newcommand{\linK}{\mathcal{L}_\infty(\kan)}
\newcommand{\linC}{\mathcal{L}_\infty\text{-category}}
\newcommand{\linCs}{\mathcal{L}_\infty\text{-categories}}
\newcommand{\ainCs}{\text{additive} \ \infty-\text{categories}}
\newcommand{\ainC}{\text{additive} \ \infty-\text{category}}
\newcommand{\inC}{\infty\text{-category}}
\newcommand{\inCs}{\infty\text{-categories}}
\newcommand{\gS}{{\Gamma}\text{-space}}
\newcommand{\ggS}{\Gamma \times \Gamma\text{-space}}
\newcommand{\gSs}{\Gamma\text{-spaces}}
\newcommand{\ggSs}{\Gamma \times \Gamma\text{-spaces}}
\newcommand{\gO}{\Gamma-\text{object}}
\newcommand{\gSCat}{{\Gamma}\text{-space category}}
\newcommand{\gCat}{{\Gamma}\text{- category}}
\newcommand{\gCats}{{\Gamma}\text{- categories}}
\newcommand{\gCAT}{{\Gamma}\Cat}
\newcommand{\gCATop}{\gCAT^{\textit{op}}}
\newcommand{\OplaxExp}[2]{\left({#2}^{#1}\right)^{\textit{Ps}}}
\newcommand{\OplaxNatExp}[2]{\left({#2}^{#1}\right)^{\textit{Ps}}}
\newcommand{\SMExp}[2]{\left(#2^{\Leins(#1)}\right)^{\textit{Ps}}}
\newcommand{\SMNatExp}[2]{\left(#2^{\Leins(#1)}\right)^{\textit{Ps}}}
\newcommand{\OplaxSec}[2]{{\Gamma}^{\textit{OL}} \left(\N, \OplaxExp{#1}{#2}\right)}
\newcommand{\PsExp}[2]{\left({#2}^{#1}\right)^{\textit{Ps}}}
\newcommand{\SMSec}[2]{{\Gamma}_{\otimes}^{\textit{str}}\left(\Leins, \OplaxExp{\Leins(#1)}{#2}\right)}
\newcommand{\OLsmSec}[2]{{\Gamma}_{\otimes}^{\textit{OL}}\left(\N, \PsExp{#1}{#2}\right)}
\newcommand{\SMHom}[2]{[#1,#2]_\otimes}
\newcommand{\SMHomNor}[2]{[#1,#2]_\bullet^\otimes}
\newcommand{\OLHom}[2]{[#1,#2]^{\textit{OL}}}
\newcommand{\OLSMHom}[2]{[#1,#2]^{\textit{OL}}_\otimes}
\newcommand{\OLHomNor}[2]{[#1,#2]^{\textit{OL}}}
\newcommand{\StrSMHom}[2]{[#1,#2]_\otimes^{\textit{str}}}
\newcommand{\parMHom}[2]{\mathbf{SBikes}(#1,#2)}
\newcommand{\StrSec}[2]{{\Gamma}^{Str}(#1,#2)}
\newcommand{\StrExp}[2]{\left(#2^{A{#1}}\right)^{\textit{Str}}}
\newcommand{\Odot}[2]{\underset{#1=1}{\overset{#2} \odot}}
\newcommand{\Prod}[2]{\underset{#1=1}{\overset{#2} \prod}}
\newcommand{\Otimes}[2]{\underset{#1=1}{\overset{#2} \otimes}}
\newcommand{\OtimesC}[3]{\underset{#1=1}{\overset{#2} \underset{#3} \otimes}}
\newcommand{\pss}{\mathbf{S}_\bullet}
\newcommand{\gSC}{{{{\Gamma}}\mathcal{S}}}
\newcommand{\ggSC}{{\Gamma\Gamma\mathcal{S}}}
\newcommand{\gSD}{\mathbf{D}(\gSC^{\textit{f}})}
\newcommand{\sCat}{\mathbf{sCat}}
\newcommand{\pSCat}{\mathbf{sCat}_\bullet}
\newcommand{\Dhom}{\mathbf{R}Hom_{\pic}}
\newcommand{\gop}{\Gamma^{\textit{op}}}
\newcommand{\gn}[1]{\Gamma^{#1}}
\newcommand{\gnk}[2]{\Gamma^{#1}({#2}^+)}
\newcommand{\gnf}[2]{\Gamma^{#1}({#2})}
\newcommand{\ggn}[1]{\Gamma\Gamma^{#1}}
\newcommand{\fU}{\mathbf{U}}
\newcommand{\cDN}{\underset{\mathbf{D}[\textit{n}^+]}{\circ}}
\newcommand{\cDK}{\underset{\mathbf{D}[\textit{k}^+]}{\circ}}
\newcommand{\cDL}{\underset{\mathbf{D}[\textit{l}^+]}{\circ}}
\newcommand{\cD}{\underset{\gSD}{\circ}}
\newcommand{\cDT}{\underset{\gSD}{\widetilde{\circ}}}
\newcommand{\ppsSets}{\sSets_{\bullet, \bullet}}
\newcommand{\gdHom}{\underline{Hom}_{\gSD}}
\newcommand{\HomU}{\underline{Hom}}
\newcommand{\ominf}{\Omega^\infty}
\newcommand{\ev}{ev}
\newcommand{\PNat}{\overline{\L}}
\newcommand{\PStr}{\L}
\newcommand{\PStrA}{\P^{\textit{str}}_\A}
\newcommand{\PNatnor}{\P^{\textit{Nat}}_{\textit{nor}}}
\newcommand{\PStrnor}{\P^{\textit{str}}_{\textit{nor}}}
\newcommand{\cu}{C(X;\mathfrak{U}_I)}
\newcommand{\Sing}{Sing}
\newcommand{\gSR}{{\Gamma}\Gamma \mathcal{S}}
\newcommand{\gSRP}{{\Gamma}\Gamma \mathcal{S}^{\textit{proj}}}
\newcommand{\tensPGSR}[2]{#1 \underset{\gSR}\wedge #2}
\newcommand{\pTensP}[3]{#1 \underset{#3}\wedge #2}
\newcommand{\MGCat}[2]{\underline{\map}_{\gCAT}({#1},{ #2})}
\newcommand{\MGBoxCat}[2]{\underline{\map}_{\gCAT}^{\Box}({#1},{ #2})}
\newcommand{\TensPFunc}[1]{- \underset{#1} \otimes -}
\newcommand{\TensP}[3]{#1 \underset{#3}\otimes #2}
\newcommand{\odotPFunc}[1]{- \underset{#1} \odot -}
\newcommand{\odotP}[3]{#1 \underset{#3}\odot #2}
\newcommand{\EgSRP}{\Ein\gSRP}
\newcommand{\Catop}{\Cat^{\textit{op}}}
\newcommand{\Catl}{\Cat_{\textit{l}}}
\newcommand{\SQCup}[2]{\underset{#1 = 1}{\overset{#2} \sqcup}}
\newcommand{\Ind}{\textit{Ind}}
\newcommand{\SFunc}[2]{\mathbf{SFunc}({#1} ; {#2})}
\newcommand{\CatFunc}[2]{\mathbf{CatFunc}({#1} ; {#2})}

\newcommand{\liminj}{\varinjlim}
\newcommand{\limproj}{\varprojlim}
\newcommand{\Lbb}{\overline{\overline{\L}}}
\newcommand{\Kbb}{\overline{\overline{\K}}}
\newcommand{\Nat}{\mathbb{N}}
\newcommand{\PLbb}{\P\overline{\overline{\L}}}
\newcommand{\ud}[1]{\underline{#1}}
\newcommand{\Gpd}{\mathbf{Gpd}}
\newcommand{\NatNum}{\mathbb{N}}

\def\ol#1{\overline{#1}}
\def\Pic{\mathbf{2}\mathcal P\textit{ic}}
\def\nc{\mathbb C}

\def\Z{\mathbb Z}
\def\P{\mathcal P}
\def\J{\mathcal J}
\def\I{\mathcal I}
\def\nC{\mathbb C}
\def\H{\mathcal H}
\def\A{\mathcal A}
\def\C{\mathcal C}
\def\D{\mathcal D}
\def\E{\mathcal E}
\def\G{\mathcal G}
\def\B{\mathcal B}
\def\L{\mathcal L}
\def\U{\mathcal U}
\def\K{\mathcal K}
\def\N{\mathcal N}
\def\M{\mathcal M}
\def\O{\mathcal O}
\def\R{\mathcal R}
\def\S{\mathcal S}

\newcommand{\undertilde}[1]{\underset{\sim}{#1}}
\newcommand{\abs}[1]{{\lvert#1\rvert}}
\newcommand{\mC}[1]{\mathfrak{C}(#1)}
\newcommand{\sigInf}[1]{\Sigma^{\infty}{#1}}
\newcommand{\x}[4]{\underset{#1, #2}{ \overset{#3, #4} \prod }}
\newcommand{\mA}[2]{\textit{Add}^n_{#1, #2}}
\newcommand{\mAK}[2]{\textit{Add}^k_{#1, #2}}
\newcommand{\mAL}[2]{\textit{Add}^l_{#1, #2}}
\newcommand{\Mdl}[2]{\L_\infty}
\newcommand{\inv}[1]{#1^{-1}}
\newcommand{\Lan}[2]{\mathbf{Lan}_{#1}(#2)}
\newcommand{\Bike}[3]{{#1}:{#2} \leadsto {#3}}
\newcommand{\POb}[1]{\mathbf{P}(#1)}
\newcommand{\SuppBike}[1]{{#1}_{\textit{Supp}}}
\newcommand{\Supp}[1]{{\textit{Supp}(#1)}}
\newcommand{\nBoxProd}[2]{\underset{i=1}{\overset{#1} \boxtimes} #2}
\newcommand{\BoxProd}[1]{\overset{#1} \boxtimes}
\newcommand{\nOdotProd}[2]{\underset{i=1}{\overset{#1} \odot} #2}
\newcommand{\OdotBike}[3]{\SuppBike{#1}(#2) \odot \SuppBike{#1}(#3)}
\newcommand{\partition}[2]{\delta^{#1}_{#2}}
\newcommand{\inclusion}[2]{\iota^{#1}_{#2}}

\newcommand{\del}{\partial}
\newcommand{\sCatO}{\mathcal{S}Cat_\O}
\newcommand{\FCgop}{\mathbf{F}\mC{N(\gop)}}
\newcommand{\hProd}{{\overset{h} \oplus}}
\newcommand{\hProdn}{\underset{n}{\overset{h} \oplus}}
\newcommand{\hProdk}[1]{\underset{#1}{\overset{h} \oplus}}
\newcommand{\map}{\mathcal{M}\textit{ap}}
\newcommand{\MapC}[3]{\mathcal{M}\textit{ap}_{#3}(#1, #2)}
\newcommand{\HMapC}[3]{\mathcal{M}\textit{ap}^{\textit{h}}_{#3}(#1, #2)}
\newcommand{\MGS}[2]{\underline{\map}_{\gSC}({#1},{ #2})}
\newcommand{\MGSR}[2]{\underline{\map}_{\gSR}({#1},{ #2})}
\newcommand{\MGSBox}[2]{\underline{\map}^{\Box}_{\gSC}({#1},{ #2})}
\newcommand{\Aqcat}[1]{\underline{#1}^\oplus}
\newcommand{\Cat}{\mathbf{Cat}}
\newcommand{\pCat}{\mathbf{CAT}_\bullet}
\newcommand{\SMCat}{\mathbf{SMCAT}}
\newcommand{\Sp}{\mathbf{Sp}}
\newcommand{\SpStb}{\mathbf{Sp}^{\textit{stable}}}
\newcommand{\SpStr}{\mathbf{Sp}^{\textit{strict}}}
\newcommand{\Sspec}{\mathbb{S}}
\newcommand{\pC}[1]{{#1}_\bullet}

\newcommand{\gM}{\Gamma \mathcal{M}}
\newcommand{\gMC}[1]{\Gamma \mathcal{#1}}
\newcommand{\gMP}{\gM^{\textit{proj}}}
\newcommand{\tensPGSC}[2]{#1 \underset{\gSC}\wedge #2}
\newcommand{\tensPGM}[2]{#1 \underset{\gM}\wedge #2}
\newcommand{\EgM}{\Ein\gM}
\newcommand{\EgSC}{\Ein\gSC}
\newcommand{\EgMP}{\Ein\gMP}

\begin{abstract}
In this paper we construct a symmetric monoidal closed model category of coherently commutative monoidal categories.
The main aim of this paper is to establish a Quillen equivalence
between a model category of coherently commutative monoidal categories
and a natural model category of Permutative (or strict symmetric monoidal) categories, $\PCat$, which is not a symmetric monoidal closed model category. 
 The right adjoint of this Quillen equivalence is the classical Segal's Nerve functor.

\end{abstract}

\maketitle

\tableofcontents

\section{Introduction}
In the paper \cite{BF78} Bousfield and Friedlander constructed a model category of $\gSs$ and proved that its homotopy category is equivalent to a homotopy category of \emph{connective spectra}. Their research was taken further by Schwede \cite{schwede} who constructed a Quillen equivalent model structure on $\gSs$ whose fibrant objects can be described as (pointed) \emph{spaces} having a \emph{coherently commutative group} structure. Schwede's model category is a symmetric monoidal closed model category under the \emph{smash product} defined by Lydakis in \cite{Lydakis} which is just a version of the \emph{Day convolution product} \cite{Day2} for normalized functors.
This paper is the first in a series of papers in which we study \emph{coherently commutative monoidal} objects in cartesian closed model categories. Our long term objective is to understand coherently commutative monoidal objects in suitable model categories of $(\infty, n)$-categories such as \cite{rezk2010}, \cite{ara2014}. The current paper deals with the case of ordinary categories which is an intermediate step towards achieving the aforementioned goal.
A $\gCat$ is a functor from the (skeletal) category of finite based sets $\gop$ into the category of all (small) categories $\Cat$. We denote the category of all $\gCats$ and natural transformations between them by $\gCAT$. Along the lines of the construction of the stable Q-model category in \cite{schwede} we construct a symmetric monoidal closed model category structure on $\gCAT$ which we refer to as the model category structure of \emph{coherently commutative monoidal categories}. A $\gCat$ is called a \emph{coherently commutative monoidal category} if it satifies the Segal condition, see \cite{segal} or equivalently it is a \emph{homotopy monoid} in $\Cat$ in the sense of Leinster \cite{Leinster}. These $\gCats$ are fibrant objects in our model category of coherently commutative monoidal categories. The main objective of this paper is to compare the category of all (small) symmetric monoidal categories with our model category of coherently commutative monoidal categories. There are many variants of the category of symmetric monoidal categories all of which have equivalent homotopy categories, see \cite[Theorem 3.9]{mandell}. All of these variant categories are \emph{fibration categories} but they do not have a model category structure. Due to this  shortcoming, in this paper we will work in a subcategory $\PCat$ which inherits a model category structure from $\Cat$. The objects of $\PCat$ are \emph{permutative categories} (also called strict symmetric monoidal categories) and maps are strict symmetric monoidal functors. We recall that a \emph{permutative category} is a symmetric monoidal category whose
tensor product is strictly associative and unital. It was shown by May \cite{May72} that permutative categories are algebras over the categorical Barrat-Eccles operad, in $\Cat$. We will construct a model category structure on  $\PCat$ by transferring along the functor $F$ which assigns to each category, the free permutative category generated by it. This functor is a right adjoint of an adjunction 
$U:\PCat \rightleftharpoons \Cat:F$ where $U$ is the forgetful functor. This model category structure also follows from results in \cite{BM1} and \cite{Lack}. We will refer to this model structure on $\PCat$ as the \emph{natural model category stucture} of permutative categories. The weak equivalences and fibrations in this model category structure are inherited from the natural model category structure on $\Cat$, namely they are equivalence of categories and isofibrations respectively.
The homotopy category of $\PCat$ is equivalent to the homotopy categories of all the variant categories of symmetric monoidal categories mentioned above. The model category of all (small)
 permutative categories is a $\Cat$-model category.
 However the shortcoming of the natural model category structure is that it is not a symmetric monoidal closed model category structure.
In the paper \cite{Schmitt} a tensor product of symmetric monoidal categories has been defined but this tensor product does not endow the category of symmetric monoidal categories with a symmetric monoidal closed structure.
However it would only endow a suitably defined homotopy category with a symmetric monoidal closed structure.
%

The model category structure of coherently commutative monoidal categories on $\gCAT$ is obtained by localizing the projective (or strict) model category structure on $\gCAT$. The guiding principle of this construction is to introduce a \emph{semi-additive} structure on the homotopy category. We achieve this by inverting all canonical maps
\begin{equation*}
X \sqcup Y \to X \times Y
\end{equation*}
in the homotopy category of the projective model category structure on $\gCAT$.
 The fibrant objects in this model category structure are coherently commutative monoidal categories. 
 We show that $\gCAT$ is a symmetric monoidal closed model category with respect to the Day convolution product. 
In the paper \cite{KS} the authors construct a model category of \emph{$E_\infty$-quasicategories} whose underlying category is the category of (honest) commutative monoids in a functor category. The authors go on further to describe a chain of Quillen equivalences between their model category and the model category of algebras over an $E_\infty$-operad in the Joyal model category of simplicial sets.
However they do not get a symmetric monoidal closed model category structure. Moreover in this paper we want to explicitly describe a pair of functors which give rise to a Quillen equivalence (in the case of ordinary categories).

 In the paper \cite{segal}, Segal described a functor from (small) symmetric monoidal categories to the category of infinite loop spaces, or equivalently, the category of connective spectra. This functor is often called Segal's $K$-theory functor because when applied to the symmetric monoidal category of finite rank projective modules over a ring $R$, the resulting (connective) spectrum is Quillen's algebraic K-theory of $R$.
  This functor factors into a composite of two functors, first of which takes values in the category of (small) $\gCats$ $\gCAT$, followed by a group completion  functor. In this paper we will refer to this first factor as \emph{Segal's Nerve} functor. We will construct an unnormalized version of the Segal's nerve functor and will denote it by $\KSeg$. The main result of this paper
 is that the \emph{unnormalized Segal's nerve functor} $\KSeg$ is the right Quillen functor of a Quillen equivalence between the natural model category of permutative categories and the model category of coherently commutative monoidal categories. Unfortunately, the left adjoint to $\KSeg$ does not have any simple description therefore in order to
 to prove our main result we will construct another Quillen equivalence, between the same two model categories, whose right adjoint is obtained by a \emph{thickening} of $\KSeg$.
 We will denote this by $\Kbar$ and refer to it as the \emph{thickened Segal's nerve} functor. The (skeletal) category of finite (unbased) sets whose objects are ordinal numbers is
 an \emph{enveloping category} of the commutative operad, see \cite{sharma3}.
In order to define a left adjoint to $\Kbar$ we will construct a \emph{symmetric monoidal completion} of  an oplax symmetric monoidal functor along the lines of Mandell \cite[Prop 4.2]{mandell}.
In order to do so we define a permutative category $\Leins$ equipped with an oplax symmetric monoidal inclusion functor $i:\N \to \Leins$, having the universal property that each oplax symmetric monoidal functor $X:\N \to \Cat$ extends uniquely to a symmetric monoidal functor $\Leins{X}:\Leins \to \Cat$ along the inclusion $i$. The category of oplax symmetric monoidal functors $\OLHom{\N}{\Cat}$
is isomorphic to $\gCAT$ therefore this symmetric monoidal extension defines a functor $\Leins:\gCAT \to \SMHom{\Leins}{\Cat}$.
Now the left adjoint to $\Kbar$, $\PNat$, can be described as the following composite
  \begin{equation*}
  \gCAT \overset{\Leins(-)} \to \SMHom{\Leins}{\Cat} \overset{hocolim} \to \PCat,
   \end{equation*}
 where $hocolim$ is a homotopy colimit functor.
 The relation between permutative categories and connective spectra has been well explored in
 \cite{Thomason}, \cite{mandell}. Thomason was the first one to show that every connective spectra is, up to equivalence, a K-theory of
 a permutative category. Mandell \cite{mandell} used a different approach to establish a similar result based on the equivalence between $\gSs$ and connective spectra established in \cite{BF78}. In the same paper Mandell proves a \emph{non-group completed} version of Thomason's theorem \cite[Theorem 1.4]{mandell} by constructing an oplax version of Segal's nerve functor. This theorem states that the oplax version of Segal's nerve functor induces an equivalence of homotopy theories between a homotopy theory of permutative categories and a homotopy theory of coherently commutative monoidal categories where the weak equivalences of both homotopy theories are based on weak equivalences in the Thomason model category structure on $\Cat$ \cite{Thomason2}.
 We have based our theory on the natural model category structure on $\Cat$ wherein the notion of weak equivalence is much stronger.
In a subsequent paper we plan to show that our main result implies the non-group completed version of Thomason's theorem \cite[Theorem 1.4]{mandell}.

 \begin{ack}
  The author gratefully acknowledges support from CIRGET which is based at the Université du Québec à Montréal, at which
  a significant part of the research for this paper was performed. The author is most thankful to Andre Joyal for his mentorship during the author's postdoctoral fellowship in Montreal. Without his vision, guidance and motivation this research would not have been possible. The author is also thankful to Nick Gurski for his useful comments on the paper.
\end{ack}

  \section{The Setup}
In this section we will review the machinery needed for various constructions in this paper.
We will begin with a review of symmetric monoidal categories and different types of functors between them functors between them.
We will also review $\gCats$ and collect some useful results about them. Most importantlt we will be reviewing the notion of \emph{Grothendieck construction} of $\Cat$ values functors and use those to construct \emph{Leinster's category} which will play a pivotal role in our theory. We will also review the \emph{natural} model category structure on $\Cat$ and $\pCat$.
\subsection[Preliminaries]{Preliminaries}
\label{Preliminaries}
 In this subsection we will briefly review the theory of permutative categories and monoidal and oplax functors between them. The definitions reviewed here and the notation specified here will be used throughout this paper.
\begin{df}
\label{SM-Cat}
A symmetric monoidal category is consists of a 7-tuple
\[
(C, -\otimes-, \unit{C}, \alpha, \beta_l, \beta_r, \gamma)
\]
 where $C$ is a category, $-\otimes-:C \times C \to C$ is a bifunctor, $\unit{C}$ is a distinguished object of $C$, 
\[
\alpha:(-\otimes-)\otimes- \Rightarrow -\otimes(-\otimes-)
\]
is a natural isomorphism called the associativity natural transformation, $\beta_l:\unit{C} \otimes - \Rightarrow id_C$ and $\beta_r:-\otimes \unit{C} \Rightarrow id_C$ are called the left and right unit natural isomorphisms and finally
\[
\gamma:(-\otimes-) \Rightarrow -\otimes- \circ \tau
\]
is the symmetry natural isomorphism. This data is subject to some conditions which are well documented in \cite[Sec. VII.1, VII.7]{MacL}
\end{df}
\begin{df}
A symmetric monoidal category $C$ is called either a \emph{permutative category} or a \emph{strict} symmetric monoidal category if the natural isomorphisms $\alpha$, $\beta_l$ and $\beta_r$ are the identity natural transformations.
\end{df}

 \begin{df}
 \label{oplax-sym-mon-functor}
 An  \emph{oplax symmetric monoidal} functor $F$ is a triple $(F, \lambda_F, \epsilon_F)$, where $F:C \to D$ is a functor  between symmetric monoidal categories $C$ and $D$, 
  \begin{equation*}
  \label{oplax-nat-trans}
  \lambda_F:F \circ (- \underset{C} \otimes -) \Rightarrow (- \underset{D} \otimes -) \circ (F \times F)
  \end{equation*}
  is a natural transformation 
  and $\epsilon_F:F(\unit{C}) \to \unit{D}$ is a morphism in $D$, such that the following
  three conditions are satisfied
  \begin{enumerate}[label = {OL.\arabic*}, ref={OL.\arabic*}]
 \item \label{OpL-unit} For each objects $c \in Ob(C)$, the following diagram commutes
 \[
  \xymatrix@C=12mm{
 F(\unit{C} \underset{C} \otimes c) \ar[r]^{\lambda_F(\unit{C}, c) \ \ \ }  \ar[d]_{F(\beta^C_l(c))}  
 &F(\unit{C}) \underset{D} \otimes F(c) \ar[d]^{\epsilon_F \underset{D} \otimes id_{F(c)}} \\
 F(c) \ar[r]_{\inv{\beta^D_l(F(c))} \ \ \ \ } & \unit{D} \underset{D} \otimes F(c)
 }
 \]
 \item \label{OpL-symmetry} For each pair of objects $c_1, c_2 \in Ob(C)$, the following diagram commutes
 \[
  \xymatrix@C=22mm{
 F(c_1\underset{C} \otimes c_2) \ar[r]^{\lambda_F(c_1, c_2) \ \ \ }  \ar[d]_{F(\gamma_C(c_1,c_2))}  
 &F(c_1) \underset{D} \otimes F(c_2) \ar[d]^{\gamma_D(F(c_1), F(c_2)) \ } \\
 F(c_2 \underset{C} \otimes c_1) \ar[r]_{\lambda_F(c_2, c_1) } & F(c_2) \underset{D} \otimes F(c_1)
 }
 \] 
\item \label{OpL-associativity} For each triple of objects $c_1, c_2, c_3 \in Ob(C)$, the following diagram commutes
 \[
  \xymatrix@C=10mm{
  & F(c_1\underset{C} \otimes c_2) \underset{D} \otimes F(c_3) \ar[rd]^{\lambda_F(c_1, c_2) \underset{D} \otimes id_{F(c_3)}} \\
 F((c_1 \underset{C} \otimes c_2) \underset{C} \otimes c_3) \ar[ru]^{\lambda_F(c_1\underset{C} \otimes c_2, c_3) \ \ \ }  \ar[d]_{F(\alpha_C(c_1,c_2, c_3))}  
 &&(F(c_1) \underset{D} \otimes F(c_2)) \underset{D} \otimes F(c_3) \ar[d]^{\alpha_D(F(c_1), F(c_2),  F(c_3)) \ } \\
 F(c_1 \underset{C} \otimes (c_2 \underset{C} \otimes c_3)) \ar[rd]_{\lambda_F(c_1, c_2 \underset{C} \otimes c_3) \ \ \ \ } && F(c_1) \underset{D} \otimes (F(c_2) \underset{D} \otimes F(c_3)) \\
 & F(c_1) \underset{D} \otimes (F(c_2 \underset{C} \otimes c_3)) \ar[ru]_{ \ \  \ \ \ \ \ id_{F(c_1)} \underset{D} \otimes  \lambda_F(c_2, c_3)}
 }
 \]

\end{enumerate}

 \end{df}
 \begin{df}
 An \emph{oplax natural transformation} $\eta$ between two oplax symmetric monoidal functors $F:C \to D$ and $G:C \to D$ is a natural transformation $\eta:F \Rightarrow G$ such that for each pair of objects $c_1, c_2$ of the symmetric monoidal category $C$, the following two diagrams commute:
 \begin{equation*}
 \xymatrix@C=14mm{
 F(c_1 \underset{C} {\otimes} c_2) \ar[r]^{\eta(c_1 \underset{C} {\otimes} c_2)} \ar[d]_{\lambda_F(c_1, c_2)} & G(c_1 \underset{C} {\otimes} c_2) \ar[d]^{\lambda_G(c_1, c_2)}  & F(\unit{C}) \ar[rd]_{\epsilon_F} \ar[rr]^{\eta(\unit{C})} && G(\unit{C}) \ar[ld]^{\epsilon_G} \\ 
 F(c_1) \underset{D} {\otimes} F(c_2) \ar[r]_{\eta(c_1) \underset{D} \otimes \eta(c_2)}  & G(c_1) \underset{D} {\otimes} G(c_2) & & \unit{D}
 }
 \end{equation*}
 \end{df}
 \begin{nota}
 \label{unital-SM-Func}
 We will say that a functor $F:C \to D$ between two symmetric monoidal categories is \emph{unital} or \emph{normalized} if it preserves the unit of the symmetric monoidal structure \emph{i.e.} $F(\unit{C}) = \unit{D}$. In particular, 
 we will say that an oplax symmetric monoidal functor is a unital (or normalized) oplax symmetric monoidal functor if the morphism $\epsilon_F$ is the identity.
\end{nota}

\begin{prop}
\label{Ext-Fun-DegWise-iso}
Let $F:C \to D$ be a functor and 
\[
\phi = \lbrace \phi(c):F(c) \overset{\cong} \to G(c) \rbrace_{c \in Ob(C)}
\]
is a family of isomorphisms in $D$ indexed by the object set of $C$. Then there exists a unique functor $G:C \to D$ such that the family $\phi$ glues together into a natural isomorphism $\phi:F \Rightarrow G$.
\end{prop}
The following lemma is a useful property of unital symmetric monoidal functors:
%

\begin{lem}
\label{SM-Func-closed-isom}
Given a unital oplax symmetric monoidal functor $(F, \lambda_F)$ between two symmetric monoidal categories $C$ and $D$, a functor $G:C \to D$, and a unital natural isomorphism
$\alpha:F \cong G$, there is a unique natural isomorphism $\lambda_G$ which enhances $G$ to a unital oplax symmetric monoidal functor $(G, \lambda_G)$ such that $\alpha$ is an oplax symmetric monoidal natural isomorphism. If $(F, \lambda_F)$ is unital symmetric monoidal then so is $(G, \lambda_G)$.
\end{lem}
\begin{proof}
We consider the following diagram:
 \begin{equation*}
 \label{unit-counit-SMM}
 \xymatrix@C=2mm{
 && C \times C \ar@/^2pc/[dd]^{F \times F} \ar@/_2pc/[dd]_{G \times G}\ar[rrrrrrrr]^{\TensPFunc{C}} &&&&&&&&C \ar@/^2pc/[dd]^{G} \ar@/_2pc/[dd]_{F } &&  \\
 &&& \ar@{=>}[ll]_{\ \alpha \times \alpha} &&&&&& \ar@{=>}[rr]^{\ \alpha} &&&& \\
 &&D \times D \ar[rrrrrrrr]_{\TensPFunc{D}} &&&&&&&& D &&
 }
 \end{equation*}
 This diagram helps us define a composite natural isomorphism $\lambda_{G}:G \circ (\TensPFunc{C}) \Rightarrow (\TensPFunc{D}) \circ G \times G$ as follows:
 \begin{equation}
 \label{def-OL-str-G}
 \lambda_{G} := (id_{\TensPFunc{D}} \circ \alpha \times \alpha) \cdot \lambda_F \cdot (\inv{\alpha} \circ id_{\TensPFunc{C}}).
 \end{equation}
 This composite natural isomorphism is the unique natural isomorphism which makes $\alpha$ a unital monoidal natural isomorphism.
 Now we have to check that $\lambda_G$ is a unital monoidal natural isomorphism with respect to the above definition. Clearly, $\lambda_G$ is unital because both $\alpha$ and $\lambda_F$ are unital natural isomorphisms. We first check
the symmetry condition \ref{OpL-symmetry}. This condition is satisfied because the following composite diagram commutes
\begin{equation*}
 \label{symm-cond-alpha}
 \xymatrix@C=12mm{
 G(\TensP{c_1}{c_2}{C}) \ar[r]^{\inv{\alpha}(\TensP{c_1}{c_2}{C})} \ar[d]_{G(\gamma^C(c_1, c_2))} & F(\TensP{c_1}{c_2}{C}) \ar[d]^{F(\gamma^C(c_1, c_2))} \ar[r]^{\lambda_F(c_1, c_2) \ \ \ }  &\TensP{F(c_1)}{F(c_2)}{D} \ar[d]^{\gamma^D(F(c_1), F(c_2))} \ar[r]^{\TensP{\alpha(c_1)}{\alpha(c_2)}{D}} & \TensP{G(c_1)}{G(c_2)}{D} \ar[d]^{\gamma^D(G(c_1), G(c_2))}  \\
 G(\TensP{c_2}{c_1}{C}) \ar[r]_{\inv{\alpha}(\TensP{c_2}{c_1}{C})} & F(\TensP{c_2}{c_1}{C}) \ar[r]_{\lambda_F(c_2, c_1) \ \ \ }  &\TensP{F(c_2)}{F(c_1)}{D} \ar[r]_{\TensP{\alpha(c_2)}{\alpha(c_1)}{D}} & \TensP{G(c_2)}{G(c_1)}{D}
 }
 \end{equation*}
The condition \ref{OpL-associativity} follows from the following equalities
\begin{multline*}
\alpha_D(G(c_1), G(c_2), G(c_3)) \circ  \TensP{\lambda_G(c_1, c_2)}{id_{G(c_3)}}{D} \circ \lambda_G(\TensP{c_1}{c_2}{C}, c_3) = \\
\TensP{(\TensP{\alpha(c_1)}{\alpha(c_2)}{D})}{\alpha(c_3)}{D}
 \circ
\alpha_D(F(c_1), F(c_2), F(c_3)) \circ 
  \TensP{\lambda_F(c_1, c_2)}{id_{F(c_3)}}{D} \circ \\\lambda_F(\TensP{c_1}{c_2}{C}, c_3) \circ
   \inv{\alpha}(\TensP{(\TensP{c_1}{c_2}{C})}{c_3}{C}) = \\
 \TensP{(\TensP{\alpha(c_1)}{\alpha(c_2)}{D})}{\alpha(c_3)}{D}
 \circ \TensP{id_{F(c_1)}}{\lambda_F(c_1, c_2)}{D} \circ \lambda_F(c_1, \TensP{c_2}{c_3}{C}) \circ F(\alpha_C(c_1, c_2, c_3)) \\
 \circ  \inv{\alpha}(\TensP{(\TensP{c_1}{c_2}{C})}{c_3}{C}) = \\
 \TensP{id_{G(c_1)}}{\lambda_G(c_1, c_2)}{D} \circ \lambda_G(c_1, \TensP{c_2}{c_3}{C}) \circ G(\alpha_C(c_1, c_2, c_3)).
\end{multline*}

If $F= (F, \lambda_F)$ is a symmetric monoidal functor then so is $G= (G, \lambda_G)$ because \eqref{def-OL-str-G} is a natural isomorphism.
 
\end{proof}

In this paper we will frequently encounter oplax (and lax) symmetric monoidal functors.
In particular we will be dealing with such functors taking values in $\Cat$.
Let $\ast$ denote the terminal category.
\begin{df}
 We define a category $\Catl$ whose objects
are pairs $(C, c)$, where $C$ is a category and $c:\ast \to C$ is a functor whose value is
$c \in C$. A morphism from $(C, c)$ to $(D, d)$ in $\Catl$ is a pair $(F, \alpha)$, where
$F:C \to D$ is a functor and $\alpha:F(c) \to d$ is a map in $D$. The category $\Catl$ is
equipped with an obvious projection functor
\begin{equation}
\label{univ-proj-cat}
p_l:\Catl \to \Cat.
\end{equation}
We will refer to the functor $p_l$ as the \emph{universal left fibration over} $\Cat$.
\end{df}
Let $(F, \alpha):(C, c) \to (D, d)$ and $(G, \beta):(D, d) \to (E, e)$ be  A pair of composable arrows in $\Catl$. Then their composite is defined as follows:
\[
(G, \beta) \circ (F, \alpha) := (G \circ F, \beta \cdot (id_G \circ \alpha)),
\]
where $\cdot$ represents vertical composition and $\circ$ represents horizontal
composition of 2-arrows in $\Cat$.
\begin{df}
The \emph{category of elements} of a $\Cat$ valued functor $F:C \to \Cat$, denoted by
$\int^{c \in C} F(c)$ or $\textit{el} F$, is a category which is defined by the following pullback square in $\Cat$:
\begin{equation*}
\xymatrix{
\int^{c \in C} F(c) \ar[r]^{p_2} \ar[d]_{p_1}& \Catl \ar[d]^{p_l} \\
C \ar[r]_{F \ \ \  } & \Cat
 }
\end{equation*}
\end{df}
The category $\int^{c \in C} F(c)$ has the following description:

The object set of $\int^{c \in C} F(c)$ consists of all pairs $(c, d)$, where $c \in Ob(C)$ and $d :\ast \to F(c)$ is a functor.
A map $\phi:(c_1, d_1) \to (c_2, d_2)$ is a pair $(f, \alpha)$, where $f:c_1 \to c_2$ is a
map in $C$ and $\alpha:F(f) \circ d_1 \Rightarrow d_2$ is a natural transformation. The category of elements of $F$
is equipped with an obvious projection functor $p:\int^{c \in C} F \to C$.
\begin{rem}
\label{simp-des-cat-el}
We observe that a functor $d:\ast \to F(c)$ is the same as an object $d \in F(c)$.
Similarly a natural transformation $\alpha:F(f) \circ d \Rightarrow b$ is the same as an
arrow $\alpha:F(f)(d) \to b$ in $F(a)$, where $f:c \to a$ is an arrow in $C$. This observation leads to
a simpler equivalent description of $\int^{c \in C} F(c)$. The objects of $\int^{c \in C} F(c)$ are
pairs $(c, d)$, where $c \in C$ and $d \in F(c)$. A map from $(c, d)$ to $(a, b)$ in $\int^{c \in C} F(c)$
is a pair $(f, \alpha)$, where $f:c \to a$ is an arrow in $C$ and $\alpha:F(f)(d) \to b$ is an arrow in $F(a)$.
\end{rem}
Next we want to define a symmetric monoidal structure on the category $\int^{c \in C} F(c)$.
In order to do so we will use two functors which we now define.
The first is the following composite
\begin{equation*}
p^{\otimes}_1:\int^{c \in C} F(c) \times \int^{c \in C} F(c) \overset{p_1 \times p_1} \to C \times C
\overset{- \underset{C} \otimes -} \to C.
\end{equation*}
The second functor 
\begin{equation*}
p^{\otimes}_2:\int^{c \in C} F(c) \times \int^{c \in C} F(c) \to \Catl
\end{equation*}
is defined on objects as follows:
\[
p^{\otimes}_2((c_1, d_1), (c_2, d_2)) := d_1 \otimes d_2,
\]
where the map on the right is defined by the following composite
\[
\ast \overset{((d_1, d_2))} \to F(c_1) \times F(c_2) \overset{\lambda_F((c_1, c_2))}
\to F(c_1 \underset{C} \otimes c_2 ).
\]
 Let $(f_1, \alpha_1): (c_1, d_1) \to (a_1, b_1)$ and $(f_2, \alpha_2):(c_2, d_2) \to (a_2, b_2)$
 be two maps in $\int^{c \in C} F(c)$. The functor is defined on arrows as follows:
\[
p^{\otimes}_2((f_1, \alpha_1), (f_2, \alpha_2)) := (F(f_1 \underset{C} \otimes f_2), \alpha_1 \otimes \alpha_2),
\]
where the second component $\alpha_1 \otimes \alpha_2$ is a natural transformation
\[
\alpha_1 \otimes \alpha_2: F(f_1 \underset{C} \otimes f_2) \circ \lambda_F((c_1, c_2)) \circ (d_1, d_2)
\Rightarrow \lambda_F((a_1, a_2)) \circ (b_1, b_2).
\]
In order to define this natural transformation, consider the following diagram:
\begin{equation*}
\xymatrix@C=16mm@R=12mm{
&& F(c_1 \underset{C} \otimes c_2 ) \ar[rdd]^{F(f_1 \underset{C} \otimes f_2)}  \\
& F(c_1) \times F(c_2) \ar[ru]^{\lambda_F((c_1, c_2)) \ \ \ \ } \ar[rd]^{ \ \ \ \ F(f_1) \times F(f_2)} 
\ar@{=>}[d]_{(\alpha_1, \alpha_2)} \\
\ast \ar[ru]^{((d_1, d_2))  \ \ \ \ }  \ar[rr]_{((b_1, b_2)) \ \ \ \ \ \ } & & F(a_1) \times F(a_2) \ar[r]_{\ \ \lambda_F((a_1, a_2))} & F(a_1 \underset{C} \otimes a_2 )
}
\end{equation*}
Now we define
\[
\alpha_1 \otimes \alpha_2 := id_{\lambda_F((a_1, a_2))} \circ (\alpha_1, \alpha_2).
\]
The arrow $\alpha_1 \otimes \alpha_2(\ast)$ has domain $\lambda_F((a_1, a_2))(F(f_1)(d_1(\ast)), F(f_2)(d_2(\ast))) \in F(a_1 \underset{C} \otimes a_2 )$. The following diagram
 \[
 \xymatrix@C=16mm{ 
 & \ast \ar[ld]_{((d_1, d_2)) \ \ } \ar[rd]^{ \ \ ((b_1, b_2))} \\
 F(c_1) \times F(c_2) \ar[d]_{\lambda_F((c_1, c_2))} \ar[rr]^{F(f_1) \times F(f_2)}
 && F(a_1) \times F(a_2) \ar[d]^{\lambda_F((a_1, a_2))}  \\
 F(c_1 \underset{C} \otimes c_2)  \ar[rr]_{F(f_1 \underset{C} \otimes f_2)}  && F(a_1 \underset{C} \otimes a_2) 
 }
 \]
 shows that 
 \[
 F(f_1 \underset{C} \otimes f_2)(\lambda_F((c_1, c_2))(d_1(\ast), d_2(\ast))) = \lambda_F(a_1, a_2)(F(f_1)(d_1(\ast)), F(f_2)(d_2(\ast))).
 \]
Now we have to verify that $p^{\otimes}_2$ is a bifunctor.
 Let $(g_1, \beta_1): (a_1, b_1) \to (x_1, z_1)$ and $(g_2, \beta_2): (a_2, b_2) \to (x_2, z_2)$ be another pair of maps in $\int^{c \in C} F(c)$. The following diagram will be useful in establishing the desired bifunctorality:
 \begin{equation*}
\xymatrix@C=14mm@R=12mm{
&& F(c_1 \underset{C} \otimes c_2 ) \ar[rdd]^{F(f_1 \underset{C} \otimes f_2)}  \\
& F(c_1) \times F(c_2)  \ar[ru]^{\lambda_F((c_1, c_2)) \ \ \ \ } \ar[rd]^{ \ \ \ \ F(f_1) \times F(f_2)} 
\ar@{=>}[d]_{(\alpha_1, \alpha_2)} \\
\ast \ar@/_2pc/[rrrd]_{(z_1, z_2)} \ar[ru]^{(d_1, d_2)  \ \ \ \ }  \ar[rr]_{(b_1, b_2) \ \ \ \ \ \ } & & F(a_1) \times F(a_2) \ar@{=>}[d]_{(\beta_1, \beta_2)} \ar[rd]^{ \ \ \ \ F(g_1) \times F(g_2)} \ar[r]_{\ \ \lambda_F((a_1, a_2))} & F(a_1 \underset{C} \otimes a_2 ) \ar[rd]^{F(g_1 \underset{C} \otimes g_2)} \\
&&& F(x_1) \times F(x_2) \ar[r]_{\ \ \lambda_F((x_1, x_2))}& F(x_1 \underset{C} \otimes x_2 )
}
\end{equation*}
  Now consider the following chain of equalities:
 \begin{multline*}
 p^{\otimes}_2((g_1, \beta_1), (g_2, \beta_2)) \circ p^{\otimes}_2((f_1, \alpha_1), (f_2, \alpha_2)) = \\
  ((F(g_1 \underset{C} \otimes g_2), id_{\lambda_F((x_1, x_2))} \circ (\beta_1, \beta_2)) \circ ((F(f_1 \underset{C} \otimes f_2), id_{\lambda_F((a_1, a_2))} \circ (\alpha_1, \alpha_2)) = \\
  (F((g_1 \underset{C} \otimes g_2) \circ (f_1 \underset{C} \otimes f_2)),
 (id_{\lambda_F((x_1, x_2))} \circ (\beta_1, \beta_2)) \cdot ( id_{F(g_1 \underset{C} \otimes g_2)} \circ (\alpha_1, \alpha_2))) =  \\
 (F((g_1 \underset{C} \otimes g_2) \circ (f_1 \underset{C} \otimes f_2)),
 (id_{\lambda_F((x_1, x_2))} \circ (\beta_1, \beta_2)) \cdot (id_{\lambda_F((x_1, x_2))} \circ (id_{F(g_1) \times F(g_2))}  \circ (\alpha_1, \alpha_2))) =  \\
 (F(g_1f_1 \underset{C} \otimes g_2f_2), id_{\lambda_F((x_1, x_2))} \circ ((\beta_1, \beta_2) \cdot (id_{F(g_1) \times F(g_2))}  \circ (\alpha_1, \alpha_2))) = \\
(F(g_1f_1 \underset{C} \otimes g_2f_2), id_{\lambda_F((x_1, x_2))} \circ ((\beta_1 \cdot (id_{F(g_1)}\circ \alpha_1), \beta_2 \cdot (id_{F(g_2)} \circ \alpha_1))) = \\
p^{\otimes}_2((g_1, \beta_1) \circ (f_1, \alpha_1)),((g_2, \beta_2) \circ (f_2, \alpha_2)).
 \end{multline*}
 The above chain of equalities prove that $p^{\otimes}_2$ is a bifunctor.
 The definitions of the functors $p^{\otimes}_1$ and $p^{\otimes}_2$ imply that
 the outer rectangle in the following diagram is commutative:
 \begin{equation*}
 \label{tens-P-def-el}
 \xymatrix{
    \int^{c \in C} F(c) \times \int^{c \in C} F(c) \ar[rr]^{p^{\otimes}_2} \ar[dd]_{p^{\otimes}_1} \ar@{-->}[rd]_{- \boxtimes -} &&  \Catl \ar[dd]^{p_l}  \\
 & \int^{c \in C} F(c) \ar[ru] \ar[ld] \\
 C \ar[rr]_F & & \Cat
 }
 \end{equation*}
 Since $\int^{c \in C} F(c)$ is apullback of $p_l$ along $F$, therefore there exists
 a bifunctor
 \begin{equation}
 \label{tensP-cat-el-laxF}
  - \boxtimes -:\int^{c \in C} F(c) \times \int^{c \in C} F(c) \to \int^{c \in C} F(c)
 \end{equation}
 which makes the entire diagram commutative. We describe this bifunctor next.
 Let $((c_1, d_1),(c_2, d_2))$ be an object in $\int^{c \in C} F \times \int^{c \in C} F(c)$. 
 \[
 (c_1, d_1) \boxtimes (c_2, d_2) := (c_1 \underset{C} \otimes c_2, \lambda_F(c_1, c_2) \circ ((d_1, d_2))).
 \]
 Let $(f_1, \alpha_1): (c_1, d_1) \to (a_1, b_1)$ and $(f_2, \alpha_2):(c_2, d_2) \to (a_2, b_2)$
 be two maps in $\int^{c \in C} F(c)$. 
 \[
 (f_1, \alpha_1) \boxtimes(f_2, \alpha_2) := (f_1 \underset{C} \otimes f_2, id_{\lambda_F(a_1, a_2)} \circ (\alpha_1, \alpha_2)).
 \]
 
\begin{thm}
\label{perm-cat-elm}
The category of elements of a $\Cat$ valued lax symmetric monoidal functor whose domain is a permutative category is a permutative category.
\end{thm}
\begin{proof}
Let $(F, \lambda_F):C \to \Cat$ be a lax symmetric monoidal functor. 
We begin by defining the symmetry natural isomorphism $\gamma_{\int^{c \in C} F(c)}$.
Let $(c_1, d_1), (c_2, d_2)$ be a pair of objects in $\int^{c \in C} F$. We define
 \begin{equation*}
 \gamma_{\int^{c \in C} F}(((c_1, d_1), (c_2, d_2))) :=
 (\gamma_C(c_1 , c_2), id).
 \end{equation*}
 The second component is identity because the lax symmetric monoidal structure of $F$
 implies that the following diagram commutes:
 \begin{equation*}
 \label{sym-nat-iso-el}
 \xymatrix@C=16mm@R=12mm{
 \ast \ar[rd]^{ \ \ \ \ ((d_1, d_2))} \ar@/^1pc/[rrd]^{ \ \ \ \ ((d_2, d_1))} \ar@/_1pc/[rdd] \\
 & F(c_1) \times F(c_2) \ar[r]^{\tau} \ar[d]^{\lambda_F((c_1, c_2))} & F(c_2) \times F(c_1) \ar[d]^{\lambda_F((c_2, c_1))} \\
 & F(c_1 \underset{C} \otimes c_2) \ar[r]_{\gamma_{C}((c_1, c_2))} & F(c_2 \underset{C} \otimes c_1)
 }
 \end{equation*} 
  It is easy to see that this defines a natural isomorphism. We claim that the proposed symmetric monoidal structure on $\int^{c \in C} F(c)$ is strictly associative. Given a third object $(c_3, d_3)$
  in $\int^{c \in C} F(c)$, we observe that
  \begin{multline*}
  ((c_1, d_1) \boxtimes (c_2, d_2)) \boxtimes (c_3, d_3) = \\
  (c_1 \underset{C} \otimes c_2 \underset{C} \otimes c_3, (\lambda_F((c_1 \underset{C} \otimes c_2, c_3)) \circ (\lambda_{F((c_1, c_2))} \times id)  \circ((d_1, d_2), d_3)).
  \end{multline*}
  The following diagram, which is the lax version of \eqref{OpL-associativity} for $F$,
  \begin{equation*}
  \label{associat-cat-el}
  \xymatrix@C=18mm@R=12mm{
  \ast   \ar[rd]^{ \ \ \ \ ((d_1, d_2), d_3)}  \ar@/_1pc/[rddd]_{ (d_1, (d_2, d_3)) \ \ \ \ } \\
 & (F(c_1) \times F(c_2)) \times F(c_3) \ar[dd]^{\alpha_{F(c_1), F(c_2), F(c_3)}} \ar[r]^{ \ \ \ \ \lambda_{F((c_1, c_2))} \times id}  & F(c_1 \underset{C} \otimes c_2) \times F(c_3) \ar[d]^{\lambda_F((c_1 \underset{C} \otimes c_2, c_3))} \\
 && F(c_1 \underset{C} \otimes c_2 \underset{C} \otimes c_3) \\
 & F(c_1) \times (F(c_2) \times F(c_3)) \ar[r]_{ \ \ \ \ id \times \lambda_{F((c_2, c_3))}} & F(c_1) \times F(c_2 \underset{C} \otimes c_3) \ar[u]_{\lambda_F{(c_1, c_2 \underset{C} \otimes c_3)}}
  }
 \end{equation*}
tells us that
\begin{multline*}
((c_1 \underset{C} \otimes c_2 \underset{C} \otimes c_3, \lambda_F((c_1 \underset{C} \otimes c_2, c_3) \circ (\lambda_{F((c_1, c_2))} \times id)  \circ((d_1, d_2), d_3)) = \\
(c_1 \underset{C} \otimes c_2 \underset{C} \otimes c_3, \lambda_F((c_1, c_2 \underset{C} \otimes c_3) \circ (id \times \lambda_{F((c_2, c_3))} )  \circ(d_1, (d_2, d_3)) = \\
(c_1, d_1) \boxtimes ((c_2, d_2) \boxtimes (c_3, d_3).
\end{multline*}
Thus we have proved that the symmetric monoidal functor is strictly associative. It is easy to see that the symmetry isomorphism $\gamma_{\int^{c \in C} F(c)}$ satisfies the \emph{hexagon diagram} because $C$ is a permutative category by assumption. Thus we have proved that$\int^{c \in C} F(c)$ is a permutative category.
 
\end{proof}

\subsection[Review of $\gCats$]{Review of $\gCats$}
\label{Rev-gamma-cats}
In this subsection we will briefly review the theory of $\gCats$. We begin by introducing some notations which will be used throughout the paper.
\begin{nota}
We will denote by $\ud{n}$ the finite set $\lbrace 1, 2, \dots, n \rbrace$ and by $n^+$ the based set $\lbrace 0, 1, 2, \dots, n \rbrace$ whose basepoint is the element $0$.
\end{nota}
\begin{nota}
 We will denote by $\N$ the skeletal category of finite unbased sets whose objects are $\ud{n}$ for all $n \ge 0$ and maps are functions of unbased sets. The category $\N$ is a (strict) symmetric monoidal category whose symmetric monoidal structure will be denoted by $+$. For to objects $\ud{k}, \ud{l} \in \N$ their \emph{tensor product} is defined as follows:
 \[
 \ud{k} + \ud{l} := \ud{k + l}.
 \]
\end{nota}
\begin{nota}
 We will denote by $\gop$ the skeletal category of finite based sets whose objects are $n^+$ for all $n \ge 0$ and maps are functions of based sets.
\end{nota}

\begin{nota}
 We denote by $\inrt$ the subcategory of $\gop$ having the same set of objects as $\gop$
 and intert morphisms.
\end{nota}
\begin{nota}
 We denote by $\act$ the subcategory of $\gop$ having the same set of objects as $\gop$
 and active morphisms.
\end{nota}
\begin{nota}
A map $f:\ud{n} \to \ud{m}$ in the category $\N$ uniquely determines an active map in $\gop$ which we will denote by $f^+:n^+ \to m^+$.
This map agrees with $f$ on non-zero elements of $n^+$.
\end{nota}
\begin{nota}
 Given a morphism $f:n^+ \to m^+$ in $\gop$, we denote by $\text{Supp}(f)$ the largest
 subset of $\n$ whose image under $f$ does not caontain the basepoint of $m^+$.
 The set $\text{Supp}(f)$ inherits an order from $\n$ and therefore could be regarded as
 an object of $\N$. We denote by $\text{Supp}(f)^+$ the based set $\text{Supp}(f) \sqcup \lbrace 0 \rbrace$
 regarded as an object of $\gop$ with order inherited from $\n$.
\end{nota}

\begin{prop}
 Each morphism in $\gop$ can be uniquely factored into a composite of an inert map followed
 by an active map in $\gop$.
\end{prop}

\begin{proof}
Any map $f:n^+ \to m^+$ in the category $\gop$ can be factored as follows:
\begin{equation}
\label{fact-in-gamma-op}
  \xymatrix@C=11mm{
 n^+  \ar[rd]_{f_{\textit{inrt}}}   \ar[rr]^{f}  &&m^+   \\
  &\text{Supp}(f)^+  \ar[ru]_{f_{\textit{act}}}
 }
\end{equation}
where $\text{Supp}(f) \subseteq n$ is the \emph{support} of the function $f$
\emph{i.e.}  $\text{Supp}(f)$ is the largest subset of $n$ whose elements are mapped
by $f$ to a non zero element of $m^+$. The map $f_{inrt}$ is the projection of
$n^+$ onto the support of $f$ and therefore $f_{inrt}$ is an inert map. The map
$f_{act}$ is the restriction of $f$ to $\text{Supp}(f) \subset \n$,
therefore it is an \emph{active} map in $\gop$.

\end{proof}

\begin{lem}
\label{oplax-rest-gcat}
 The restriction of a $\gCat$ $X$ to $\N$, namely the composite functor
 \[
  X|_{\N}:\N \to \gop \overset{X} \to \Cat 
 \]
 is an oplax symmetric monoidal functor.
 \end{lem}
 \begin{proof}
 	We begin by recalling that for any two categories $C$ and $D$ their product $C \times D$ is defined to be a category whose object set is $Ob(C) \times Ob(D)$ and the morphism set is $Mor(C) \times Mor(D)$. The symmetry isomorphism $\gamma^{\Cat}_{C, D}:C \times D \to D \times C$ is the obvious isomorphism which on objects is the bijection
 	\[
 	 \tau:Ob(C) \times Ob(D) \cong Ob(D) \times Ob(C).
 	 \]
 	We will define a triple $(X|_\N, \lambda_X, \epsilon_X)$ where $X|_{\N}$ is the restriction functor as defined above.
 	$\epsilon_X:X(0^+) = X|_\N(\ud{0}) \to \ast$ is the terminal map (recall that the terminal category $\ast$ the unit of the cartesian structure on $\Cat$). Now we define a natural transformation $\lambda_X$:
 	
  For any pair of objects $k, l \in Ob(\N)$ we have the following functor
  \begin{equation*}
  \label{partition-map}
  (X(\partition{k+l}{k}), X(\partition{k+l}{l})):X((k+l)^+) \to X(k^+) \times X(l^+).
  \end{equation*}
  Using this functor we will define a natural transformation
 \begin{equation*}
  \label{oplax-nat-trans-rest-gcat}
  \lambda_X:X|_{\N} \circ (- + -) \Rightarrow (- \times -) \circ (X|_{\N} \times X|_{\N}),
  \end{equation*}
 where $- + -$ and $- \times -$ are the bifunctors providing symmetric monoidal structures on
 $\N$ and $\Cat$ respectively. This natural transformation is defined as follows:
 \[
  \lambda_X(k, l) := (X(\partition{k+l}{k}), X(\partition{k+l}{l})).
 \]
 For each pair of maps $(f_1, f_2)$ in $\N$, where $f_1:k \to m$ and $f_2:l \to n$, we have
 the following two commutative diagrams
 \begin{equation*}
\label{nat-oplax-maps}
  \xymatrix@C=11mm{
 X((k+l)^+)  \ar[r]^{X(f_1 + f_2)}   \ar[d]_{X(\partition{k+l}{k})}  &X((m+n)^+) \ar[d]^{X(\partition{m+n}{m})}  & X((k+l)^+)  \ar[r]^{X(f_1 + f_2)}   \ar[d]_{X(\partition{k+l}{l})}  &X((m+n)^+) \ar[d]^{X(\partition{m+n}{n})} \\
  X(k^+)   \ar[r]_{X(f_1)} & X(m^+) & X(l^+)   \ar[r]_{X(f_2)} & X(n^+)
 }
\end{equation*}
 These diagrams together imply that $\lambda_X$ is a natural transformation.
 We observe that in the category $\gop$ we have the following equalities:
 \begin{equation}
 \partition{k+l}{k} \circ \tau(k, l) = \partition{l+k}{k} \hspace{1cm} \text{and} \hspace{1cm} \partition{k+l}{l} \circ \tau(k, l) = \partition{l+k}{l}
 \end{equation}
 where $\tau(k, l):(k+l)^+ \to (l+k)^+$ is the map $\gamma_\N(k, l)^+$ where $\gamma_\N(k, l):\ud{k+l} \to \ud{l+k}$ is the symmetry map in $\N$.
 These two equations give us the following two commutative diagrams
 \begin{equation*}
\label{nat-oplax-maps-2}
  \xymatrix@C=11mm{
 X((k+l)^+)  \ar[r]^{X(\tau(k,l))}   \ar[d]_{X(\partition{k+l}{k})}  &X((l+k)^+) \ar[d]^{X(\partition{l+k}{k})}  & X((k+l)^+)  \ar[r]^{X(\tau(k,l))}   \ar[d]_{X(\partition{k+l}{l})}  &X((l+k)^+) \ar[d]^{X(\partition{l+k}{l})} \\
  X(k^+)   \ar@{=}[r] & X(k^+) & X(l^+)   \ar@{=}[r] & X(l^+)
 }
\end{equation*}
The above two diagrams imply that the following diagram is commutative which implies that the functor $X|_\N$ satisfies the symmetry condition \ref{OpL-symmetry}.
 \begin{equation*}
\label{symmetry-oplax-cond}
  \xymatrix@C=14mm{
 X((k+l)^+)  \ar[r]^{X(\gamma^\N_{k,l})}   \ar[d]_{(X(\partition{k+l}{k}), X(\partition{k+l}{l}))}  &X((l+k)^+) \ar[d]^{(X(\partition{k+l}{l}), X(\partition{k+l}{k}))}   \\
  X(k^+) \times X(l^+)  \ar[r]_{\gamma^{\Cat}_{X(k^+), X(l^+)}} & X(l^+) \times X(k^+) &
 }
\end{equation*}
   The condition \ref{OpL-unit} follows from the commutativity of the following diagram for each $k^+ \in Ob(\gop)$:
   \begin{equation*}
   \label{unit-oplax-cond}
   \xymatrix@C=14mm{
   	& X(0^+) \times X(k^+) \ar[dd]^{(\epsilon_X, id)} \\
   	X(k^+) \ar[ru]^{(X(0_k), id)} \ar[rd]_{(t, id)} \\
   	& \ast \times X(k^+)
   }
\end{equation*}
where $t:X(k^+) \to \ast$ is the unique terminal map in $\Cat$ and $0_k:k^+ \to 0^+$ is the terminal map in $\gop$.
The condition \ref{OpL-associativity} follows from an obvious (large) commutative diagram.
 \end{proof}

\subsection[Natural model category structure on $\Cat$]{Natural model category structure on $\Cat$}
\label{Nat-Mdl-Str-CAT}
In this subsection we will review the \emph{natural model category
structure} on the category of all small categories $\Cat$.
 The weak equivalences in this model structure are \emph{equivalences of categories}. We begin by reviewing this notion:
 \begin{df}
 	\label{eq-of-cat}
 	A functor $F:C \to D$ is called an equivalence of categories if there exists another functor $G:D \to C$ and two natural isomorphisms
 	\begin{equation*}
 	\epsilon:FG \cong id_D \ \ \ \ \ \ \textit{and} \ \ \ \ \ \ \eta:id_C \cong GF
 	\end{equation*}
 	\end{df}
 The inclusion of the category of all (small) groupoids $\Gpd$ into $\Cat$ has a right adjoint which we denote by $J:\Cat \to \Gpd$. For any category $C$, $J(C)$ is the largest groupoid contained in $C$.
 The following charaterization of an equivalence of categories will be useful throughout the paper:
 \begin{lem}
 	\label{char-eq-cat}
 	A functor $F:C \to D$ is an equivalence of categories if and only if the following two induced functor are equivalences of groupoids:
 	\[
 	J(F):J(C) \to J(D) \ \ \ \ \textit{and} \ \ \ \ J([I, F]):J([I, C]) \to J([I, D])
 	\]
 \end{lem}
 \begin{proof}
 	$(\Rightarrow)$
 	It is easy to see that $J$ preserves equivalences of categories \emph{i.e.} if $F:C \to D$ is an equivalence of categories then $J(F):J(C) \to J(D)$ is also an equivalence.
 	An equivalence of categories induces an equivalence on its category of arrows, thus the functor $J([I; F])$ is also an equivalence of categories.
 	
 	$(\Leftarrow)$ Let us assume that the two conditions hold. Let $f:d \to e$ be an arrow in $D$ such that the domain and codomain objects $d$ and $e$ respectively are in the image of $F$. By assumption the functor $J([I, F])$ is essentially surjective therefore there exists an arrow $g:a \to b$ in $C$ such that the following diagram commutes:
 	\[
 	\xymatrix{
 		F(a) \ar[d]_{F(g)} \ar[r]^{\epsilon(d)} &d \ar[d]^f \\
 		F(b) \ar[r]_{\epsilon(e)} & e
 	}
 	\]
 	where the pair $(\epsilon(d), \epsilon(e))$ is an isomorphism in the arrow category $[I; D]$.
 	By the assumption that $J(F)$ is an equivalence of categories, there exist two unique (invertible) arrows $h$ and $k$ such that $F(h) = \epsilon(d)$ and $F(k) = \epsilon(e)$. This implies that there exists a unique arrow $\inv{k} \circ g \circ h$ such that
 	$f = F(k \circ g \circ \inv{h})$. Thus we have proved that $F$ is fully-faithful. Let $p$ be an object of $D$ which is NOT in the image of $F$ then the assumption of equivalence of $J([I, F])$ garuntees the existence of an isomorphism $m:x \to y$ in $C$ such that the following diagram commutes:
 	\[
 	\xymatrix{
 		F(x) \ar[d]_{F(m)} \ar[r]^{\epsilon(d)} &d \ar@{=}[d]^{id_d} \\
 		F(y) \ar[r]_{\epsilon(e)} & d
 	}
 	\]
 	Thus we have shown that $F$ is essentially surjective.
 \end{proof}
 A significant part of this section will be devouted to review properties
 of fibrations in this model structure, namely \emph{isofibrations}, which
 we now define:

\begin{df}
\label{isofibration}
 If $C$ and $D$ are categories, we shall say that a functor $F:C \to D$ is an \emph{isofibration}
 if for every object $c \in C$ and every isomorphism $v \in Mor(\D)$ with source $F(c)$,
 there exists an isomorphism $u \in C $ with source $c$ such that $F(u)=v$.
\end{df}
 \begin{nota}
 Let $J$ be the groupoid generated by one isomorphism $0 \cong 1$.
 We shall denote the inclusion $\lbrace 0 \rbrace \subset J$ as a
 map $d_1:0 \to J$ and the inclusion $\lbrace 1 \rbrace \subset J$
 by the map $d_0:1 \to J$.
 \end{nota}
 \begin{nota}
 \label{int-hom}
 Let $A$ and $B$ be two small categories, we will denote by $[A, B]$,
 the category of all functors from $A$ to $B$ and natural transformations
 between them. 
 \end{nota}
 The next proposition gives a characterization of isofibrations.
 \begin{prop}
 \label{char-of-isofib}
 A functor $F:C \to D$ is an isofibration if and only if it has the
 right lifting property with respect to the inclusion $i_0:{0}\hookrightarrow J$
 and therefore also with the inclusion $i_1:{1} \hookrightarrow J$.
 \end{prop}
 \begin{proof}
 Let us assume that $F:C \to D$ is an isofibration, then whenever
 we have a (outer) commutative square
 \[
\label{lift-isofibration}
  \xymatrix@C=11mm{
 0  \ar@{^{(}->}[d]_{i_0} \ar[r]  &C \ar[d]^{F}   \\
  J \ar@{-->}[ru] \ar[r] &D
 }
\]
it is easy to see that there exists a dotted arrow which makes the
entire diagram commute. Conversely, let us assume that the functor
$F$ has the right lifting property with respect to the inclusion functor
$i_0$. Let $f:F(c) \to d$ be an isomorphism in $D$, where $c \in Ob(C)$.
Now there exists a (unique) functor $A:J \to D$ such that $A(0 \cong 1) = f$
and we have the following (outer) commutative square
\[
  \xymatrix@C=11mm{
 0  \ar@{^{(}->}[d]_{i_0} \ar[r]^c  &C \ar[d]^{F}   \\
  J \ar@{-->}[ru]^L \ar[r]_A &D
 }
\]
By assumption there exists a dotted arrow $L$ making the
entire diagram commutative. This implies that
\[
F(L(0 \cong 1) ) = A(0 \cong 1) = f.
\]
Thus we have an isomorphism $L(0 \cong 1):c \to e$,
in $C$, such that $F(L(0 \cong 1) ) = f$. This proves that
$F$ is an isofibration. Thus we have proved that a functor
is an isofibration if and only if it has the right lifting property
with respect to $i_0$. Finally we will show that a functor has
the right lifting property with respect to $i_0$ if and only if it
has the right lifting property with respect to $i_1$. It would be
sufficient to observe that the right commutative square, in the
following diagram
\[
  \xymatrix{
 1 \ar@{^{(}->}[d]_{i_1}  \ar@{}[r]|*=0[@]{\cong} &0  \ar@{^{(}->}[d]_{i_0} \ar[r]^c  &C \ar[d]^{F}   \\
  J \ar[r]_{\sigma} \ar@{-->}[rru] &J \ar@{-->}[ru] \ar[r]_A &D
 }
\]

 has a lift if and only if the outer commutative diagram has
 a lift. The automorphism $\sigma:J \to J$ in the diagram above
 permutes the two objects of the groupoid $J$.
 \end{proof}
 We want to present a charaterization of \emph{acyclic isofibrations},
 \emph{i.e.} those functors of categories which are both an isofibration
 and an equivalence of categories, similar to the characterization of
 isofibrations given by proposition \ref{char-of-isofib}. The following
 property of acyclic isofibrations will be useful in achieving this goal:
 \begin{lem}
 \label{acyc-fib-surj-eq}
 An equivalence of categories is an isofibration iff it is surjective on objects.
 \end{lem}
 \begin{proof}
 Let $F:C \to D$ be an equivalence which is an isofibration.
 Then for every object $d \in Ob(D)$, there exists an object $c \in Ob(C)$
 together with an isomorphism $v: F(c)\to d$ because an equivalence is essentially
 surjective. There is then an isomorphism $u:c \to c'$ in $C$ such that $F(u)=v$ because $F$
 is an isofibration. We then have $F(c')= d$, and this shows that $F$ is surjective
 on objects. Conversely, let us show that an equivalence $F:C \to D$ surjective on
 objects is an isofibration. If $c$ is an object of $C$ and $v:F(c) \to d$ is an
 isomorphism in $D$, then there exists an object $c'\in C$ such that $F(c')=d$
 because $F$ is surjective on object by assumption. The map
 $F_{c,c'}Hom_{C}(c,c')\to Hom_D(F(c),F(c'))$ specified by the functor $F$
 is bijective because $F$ is an equivalence. Hence there exists
 a morphism $u: c \to c'$ such that $F(u)=v$. The morphism $u$ is invertible because
 $v$ is invertible and $F$ is an equivalence. This shows that $F$ is an isofibration.
 \end{proof}

 \begin{nota} 
 We will denote the category $0 \to 1$ either by $I$ or by $[1]$. We will
 denote the discrete category $\lbrace 0, 1 \rbrace$ either by $\partial I$
 or $\partial [1]$. We will denote the category $0 \overset{f_{01}} \to 1
 \overset{f_{12}} \to 2$ by $[2]$.
 \end{nota}
 Now we define a category $\partial [2]$ which has the same object set
 as the category $[2]$, namely $\lbrace 0, 1, 2 \rbrace$. The Hom sets
 of this category are defined as follows:
 \begin{equation*}
Hom_{\partial [2]}(i, j) = 
\begin{cases}
\lbrace f_{01} \rbrace, & \text{if} \ i=0 \ \ \text{and} \ \ j=1 \\
\lbrace f_{12} \rbrace, & \text{if} \ i=1 \ \ \text{and} \ \ j=2 \\
\lbrace f_{02}, f_{12} \circ f_{01} \rbrace, & \text{if} \ i=0 \ \ \text{and} \ \ j=2 \\
\lbrace id \rbrace, & \text{otherwise}.
\end{cases}
\end{equation*}
 We have the following functor
 \[
  \partial_2:\partial[2] \hookrightarrow [2]
 \]
 which is identity on objects. This functor sends the morphism
 $f_{01}(\text{resp.} \ f_{12})$ to the morphism $0 \to 1(\text{resp.} \ 1 \to 2)$ in the category $[2]$.
 Both morphisms $f_{02}, f_{12} \circ f_{01}$ are mapped to the composite morphism
 $0 \overset{f_{01}} \to 1 \overset{f_{12}} \to 2$. Similarly we have
 the map $\partial_1: \partial[1] \to [1]$ which is identity on objects.
 We have a third functor $\partial_0:\emptyset \to [0]$ which is obtained
 by the unique function $\emptyset \to \lbrace 0 \rbrace$. We will refer
 to these three functors as the \emph{boundary maps}.

 \begin{prop}
 \label{char-of-acyc-fib}
 A functor $F:C \to D$ is both isofibration and an equivalence of categories
 if and only if it has the right lifting property with respect to the three
 boundary maps $\partial_0, \partial_1$ and $\partial_2$.
 \end{prop}
 \begin{proof}
 Let us first assume that $F$ is an isofibration as well as an
 equivalence of categories. Now Lemma \ref{acyc-fib-surj-eq} says
 that $F$ is surjective on objects which is equivalent to $F$ having
 the right lifting property with respect to the boundary map $\partial_0$.
 Now we observe that for any pair of objects $d, d' \in Ob(D)$, there
 exists a pair of objects $c, c' \in Ob(C)$ such that $F(c) = d$
 and $F(c') = d'$ and the morphism
 \[
  F_{c,c'}:Hom_C(c, c') \to Hom_D(d, d')
 \]
 is a bijection. This implies that $F$ has the right lifting
 property with respect to the morphism $\partial_1$. Whenever we
 have the following (outer) commutative diagram
 \[
  \xymatrix@C=11mm{
 \partial[2]  \ar[d]_{\partial_2} \ar[r]^K  &C \ar[d]^{F}   \\
  [2] \ar@{-->}[ru]^L \ar[r] &D
 }
\]
 we have the following equality
 \[
  F(K(f_{02})) = F(K(f_{12})) \circ F(K(f_{01})),
 \]
 where the maps $f_{02}$, $f_{12}$ and $f_{01}$ are defined above.
 The morphism
 \[
  F_{K(0), K(2)}:Hom_C(K(0), K(2)) \to Hom_D(F(K(0)), F(K(2)))
 \]
 is a bijection, this implies that the morphism $K(f_{02}):K(0) \to K(2)$
 is the same as the composite morphism $K(f_{12}) \circ K(f_{01}):K(0) \to K(2)$
 Now we are ready to define the lifting (dotted) arrow $L$. We define
 the object function of the functor $L$ to be the same as that of the functor $K$,
 \emph{i.e.} $L_{Ob} = K_{Ob}$. We define $L(f_{01}) = K(f_{01})$ and
 $L(f_{12}) = K(f_{12})$. Now the discussion above implies that this
 definition makes the entire diagram commute.
 
 Conversely, let us assume that the morphism $F$ has the right lifting
 property with respect to the three boundary maps. The morphism $F$
 having the right lifting property with respect to
 $\partial_0$ is equivalent to $F$ being surjective on objects.
 Now the right lifting property with respect to $\partial_1$ implies
 that for any map $g:d \to d'$ in the category $D$, there exists a
 map $w:c \to c'$ in $C$, such that $F(w) = g$, for each pair of objects $c, c' \in Ob(C)$
 such that $F(c) = d$ and $F(c') = d'$. Let $c \in Ob(C)$ and
 $v:F(c) \to d$ be an isomorphism in $D$. Now we can define a functor
 $A:[2] \to D$, on objects by $A(0) = A(2) = F(c)$, $A(1) = d$ and on morphisms by
 $A(f_{01}) = v$ and $A(f_{12}) = \inv{v}$. As mentioned earlier, the
 right lifting property with respect to $\partial_1$ implies that
 there exist two maps $u:c \to c'$ and $r:c' \to c$ such that $F(u) = v$
 and $F(r) = \inv{v}$. This allows us to define a functor
 $K:\partial[2] \to C$, on objects by $K(0)=K(2)=c$ and $K(1)=c'$
 and on morphisms by $K(f_{01}) = u$, $K(f_{12}) = r$ and $K(f_{02}) = r \circ u$.
 This definition gives us the following (outer) commutative diagram
 \[
  \xymatrix@C=11mm{
 \partial[2]  \ar[d]_{\partial_2} \ar[r]^K  &C \ar[d]^{F}   \\
  [2] \ar@{-->}[ru]^L \ar[r]_A &D
 }
\]
 Our aassumption of right lifting property with respect to $\partial_2$
 gives us a lift (dotted arrow) $L$ which makes the entire diagram commute.
 This implies the $r \circ u = id_c$. A similar argument will show
 that $u \circ r = id_{c'}$. Thus we have shown that $F$ is an isofibration
 which is surjective on objects. Lemma \ref{acyc-fib-surj-eq} says that $F$
 is both an equivalence of categories and an isofibration.

 \end{proof}

 \begin{df}
 We shall say that a functor $F:C \to D $ is \emph{monic (resp. surjective, bijective) on objects}
 if the object function of $F$, $F_{Ob}:Ob(C) \to Ob(D)$, is injective (resp. surjective, bijective).
 \end{df}
 \begin{thm} [\cite{AJ1}]
 \label{nat-model-cat-str}
 There is a combinatorial model category structure on the category of all small
 categories $\Cat$ in which
 \begin{enumerate}
 \item A cofibration is a functor which is monic on objects.
 \item A fibration is an isofibration and
 \item A weak-equivalence is an equivalence of categories.
  \end{enumerate}
  Further, this model category structure is cartesian closed and proper.
 We will call this model category structure as the \emph{natural model category structure on $\Cat$}.
 \end{thm}
\begin{nota}
We will denote by $0$ the terminal category having one object 
$0$ and just the identity map. 
\end{nota}
\begin{df}
\label{pointed-cat}
A small \emph{pointed} category is a pair $(C, \phi)$ consisting of a small category $C$ and a functor $0 \to C$. A \emph{basepoint preserving}
functor between two pointed categories $(C, \phi)$ and $(D, \psi)$ is a
functor $F:C \to D$ such that the following diagram commutes
\[
\label{pointed-functor}
  \xymatrix@C=11mm{
 & 0  \ar[rd]^{\psi} \ar[ld]_{\phi}   \\
  C \ar[rr]_{F} &&D
 }
\]
\end{df}
Every model category uniquely determines a model
category structure on that category of its pointed objects, see
\cite[Proposition 4.1.1]{JT1}. Thus we have the following theorem:
 \begin{thm} 
 \label{nat-model-str-pointed-cat}
 There is a model category structure on the category of all pointed small
 categories and basepoint preserving functors $\pCat$ in which
 \begin{enumerate}
 \item A cofibration is a basepoint preserving functor which is monic on objects.
 \item A fibration is a basepoint preserving functor which is also an isofibration of (unbased) categories and
 \item A weak-equivalence is a basepoint preserving functor which is also equivalence of (unbased) categories.
  \end{enumerate}
 We will call this model category structure as the \emph{natural model
 category structure on $\pCat$}.
 \end{thm}
 Let $J^+$ denote the category $J \coprod \ast$ \emph{i.e.} the category
 having two connected components $J$ and the terminal $\ast$.
 We will consider $J^+$ as a pointed category having basepoint $\ast$.
 Let $0^+$ and $1^+$ denote the discrete pointed categories $0 \coprod \ast$
 and $1 \coprod \ast$ respectively, both having basepoints $\ast$.
 Let $I$ denote the category $0 \to 1$. As above we denote by $I^+$
 the category $I \coprod \ast$. We will use the following result later in this paper
\begin{thm}
 \label{comb-mdl-cat-pointed-cat}
 The natural model category structure on $\pCat$ is a
 combinatorial model category structure.
\end{thm}
\begin{proof}
 The category $\pCat$ is locally presentable because $\Cat$ is
 a locally presentable category and the category of pointed objects
 of a locally presentable category is also locally presentable \cite{AR94}.
 Now it remains to show that
 the natural model category $\pCat$ is cofibrantly generated.
 Proposition \ref{char-of-acyc-fib} implies that a morphism in $\pCat$
 is an acyclic fibration if and only if it has the right lifting
 property with respect to the three boundary maps $\partial_0^+, \partial_1^+$
 and $\partial_2^+$.
 
 The set of
 generating acyclic cofibrations is $Q = \lbrace i_0^+, i_1^+ \rbrace$ where
 $i_0^+:0^+ \hookrightarrow J^+$ and $i_1^+:1^+ \hookrightarrow J^+$ are
 the two basepoint preserving inclusion functors. This follows from proposition \ref{char-of-isofib}.
 The set of generating cofibrations is $R = \lbrace in_0^+, in_1^+ \rbrace$ where
 $in_0^+:0^+ \hookrightarrow I^+$ and $in_1^+:1^+ \hookrightarrow I^+$ are
 the two basepoint preserving inclusion functors.

\end{proof}

 Let $(C, \phi)$ and $(D, \psi)$ be two pointed categories, we define
 another category, which is denoted by $C \vee D$, by the following
 pushout square:
 \[
\label{sum-pointed-cat}
  \xymatrix@C=15mm{
 0 \ar[r]^{\phi} \ar[d]_{\psi} & C   \ar[d]   \\
 D \ar[r] &C \vee D
 }
\]
$C \vee D$ is a pointed category with the obvious basepoint and
 we will refer to it as the \emph{sum of $(C, \phi)$ and $(D, \psi)$}. 
 We define another (small) pointed category $C \wedge D$ by the
 following pushout square:
 \[
\label{tens-pointed-cat}
  \xymatrix@C=15mm{
 C \vee D \ar@{^{(}->}[r] \ar[d] & C \times D  \ar[d]   \\
 \ast \ar[r] &C \wedge D
 }
\]
We will refer to the pointed category $C \wedge D$ as the
\emph{tensor product} of $C$ and $D$.
It is easy to check that this tensor product construction is
functorial \emph{i.e.} there is a bifunctor $-\wedge-:\Cat \times \Cat \to \Cat$
which is defined on objects by $(C,D) \mapsto C \wedge D$.

\subsection[Leinster construction]{Leinster construction}
\label{real-funct}
 In this section we will construct a permutative category which would help us in constructing the
 desired left adjoint to the Segal's Nerve functor. We will refer to this category
 as the \emph{Leinster category} and we will denote it by $\Leins$. The defining property of this permutative category is that for each permutative category $P$ we get a following bijection of mapping sets:
 \[
 \PCat(\Leins, P) \cong \mathbf{OLSM}(\N, P)
 \]
 where the mapping set $\mathbf{OLSM}(\N, P)$ is the set of all oplax symmetric monoidal functors from $\N$ to $P$. The existence of this category is predicted in the paper \cite[Theorem 2.8]{GJO}
 An object in $\Leins$ is an order preserving morphism of the category $\N$
 namely an order preserving map of (finite) unbased sets $\vec{k}:\underline{k} \to \underline{r}$. For another object $\vec{m}:\underline{m} \to \underline{s}$ in $\Leins$, a morphism between $\vec{k}$ and $\vec{m}$ is a pair $(h, \phi)$, where $h:\underline{s} \to \underline{r}$ and $\phi:\underline{k} \to \underline{m}$ are morphisms in $\N$
 such that the following diagram commutes:
 \begin{equation*}
 \xymatrix{
 \underline{k} \ar[d]_{\vec{k}} \ar[r]^\phi & \underline{m} \ar[d]^{\vec{m}} \\
 \underline{r} & \underline{s} \ar[l]^{h}
 }
 \end{equation*}
 
 \begin{nota}
 For an object $\vec{m}:\underline{m} \to \underline{s}$ in $\Leins$ we will
 refer to the natural number $s$ as the \emph{length of $\vec{m}$}.
 \end{nota}
 \begin{rem}
An object of $\Leins$, $\vec{m}:\underline{m} \to \underline{s}$, should be viewed as a finite sequence of objects of $\N$ namely $(m_1, m_2, \dots, m_s)$ for $s > 0$, with $s = 0$ corresponding to the empty sequence $()$, where $m_i = \inv{\vec{m}}(i)$, for $1 \le i \le s$.
 \end{rem}
 \begin{rem}
 An object $\vec{m}:\underline{m} \to \underline{s}$ does not have to be a surjective map. In other words the corresponding sequence $\vec{m} = (m_1, \dots, m_s)$ can have components which are empty sets.
 \end{rem}
 \begin{rem}
 Let $\vec{n}$ and $\vec{m}$ be two objects in $\Leins$. A morphism $(h, \phi):\vec{n} \to \vec{m}$, in $\Leins$, should be viewed as a family of morphisms
 \[
 \phi(i) = n_i \to \underset{h(j) = i} + m_j
 \]
 for $1 \le i \le s$,
 where $+$ represents the symmetric monoidal structure on $\N$.
 \end{rem}
%

We want to recall from \cite{sharma3} or appendix \ref{OLtoSMFunc} how each $\gCat$ $X$ can be extended to a symmetric monoidal functor
$\Leins(X):\Leins \to \Cat$. This functor is defined on objects as follows:
\begin{equation*}
\Leins(X)(\vec{m}) := X(m_1^+) \times X(m_2^+) \times \cdots \times X(m_r^+)
\end{equation*}
where $() \ne \vec{m} = (m_1, m_2, \dots, m_r)$ is an object of $\Leins$. $\Leins(X)(()) = \ast$.
For each map $F=(f, \phi): \vec{m} \to \vec{n}$ in $\Leins$ we want to define a functor
\[
\Leins(X)(F):\Leins(X)(\vec{m}) \to \Leins(X)(\vec{n}).
\]
Each map $\phi(i)$ in the family $\phi$ provides us with a composite functor
\[
X(m_i^+) \overset{X(\phi(i))} \to X(\underset{f(j) = i} + n_j) \overset{K_i} \to \underset{f(j) = i} \prod X(n_j),
\]
where $K_i = (X(\partition{\underset{f(j) = i} + n_j}{n_{j_1}}), \dots, X(\partition{\underset{f(j) = i} + n_j}{n_{j_r}}))$. For each pair $n$-fold product functor in $\Cat$,
there is a canonical natural isomorphism between them which we denote by $\textit{can}$.
This gives us the following composite functor
\begin{equation*}
\underset{i=1}{\overset{\length{m}} \prod} X(m_i^+) \overset{\underset{i=1}{\overset{\length{m}} \prod} X(\phi(i))} \to X(\underset{f(j) = i} + n_j) \overset{\underset{i=1}{\overset{\length{m}} \prod} K_i} \to \underset{i=1}{\overset{\length{m}} \prod} \underset{f(j) = i} \prod X(n_j) \overset{\textit{can}}\to \underset{k=1}{\overset{\length{n}} \prod} n_k
\end{equation*}
which is the definition of $\Leins(X)(F)$. In other words
\begin{equation*}
\label{SMExt-def-mor}
\Leins(X)(F) := can \circ \underset{i=1}{\overset{\length{m}} \prod} K_i \circ \underset{i=1}{\overset{\length{m}} \prod} X(\phi(i))
\end{equation*}
\begin{prop}
\label{Ext-gCat}
Let $X$ be a $\gCat$, there exists an extension of $X$ to $\Leins$, $\Leins(X):\Leins \to \Cat$ which is a symmetric monoidal functor.
\end{prop}
\begin{rem}
The symmetric monoidal extension described above is functorial in $X$. In other words we get a functor
\begin{equation}
\label{SM-Ext-OL}
\Leins(-):\gCAT \to \SMHom{\Leins}{\Cat}
\end{equation}
\end{rem}
 \begin{df}
 \label{Groth-Cons-gCat}
 For a $\gCat$ $X$ we define
 \begin{equation*}
  \Lbb(X) := \int^{ \vec{n} \in \Leins} \Leins(X)(\vec{n}).
 \end{equation*}
 \emph{i.e.} the \emph{Grothendieck construction} of $\Leins(X)$.
 \end{df}
 More concretely, an object in the category $\Lbb(X)$ is a pair
 $(\vec{m}, \vec{x})$ where $\vec{m}:\underline{m} \to \underline{s} \in Ob(\Leins)$ and
 \[
  \vec{x} = (x_1, x_2, \dots, x_s) \in Ob(X(m_1^+) \times X(m_2^+) \times \dots \times X(m_s^+) ).
  \]
  A morphism from $(\vec{m}, \vec{x})$ to $(\vec{n}, \vec{y})$ in $\Lbb(X)$
  is a pair $((h, \phi), F)$ where $(h, \phi):\vec{m} \to \vec{n}$ is a map in $\Leins$
  and $F:\Leins(X)((h, \phi))(\vec{x}) \to \vec{y}$ is a map in the product category
  $X(n_1^+) \times X(n_2^+) \times \dots \times X(n_r^+)$.
  
  Now we define a \emph{tensor product}  on the category $\Lbb(X)$. Let $(\vec{n}, \vec{x})$ and $(\vec{m}, \vec{y})$ be two objects of $\Lbb(X)$, we define another object $(\vec{n}, \vec{x}) \underset{\Lbb(X)} \otimes (\vec{m}, \vec{y})$ as follows:
  \begin{equation}
  \label{tens-prod-LbbX-ob}
  (\vec{n}, \vec{x}) \underset{\Lbb(X)} \otimes (\vec{m}, \vec{y}) :=
  (\vec{n} \Box \vec{m}, \inv{\lambda_{\Leins(X)}(\vec{n}, \vec{m})}((\vec{x}, \vec{y}))).
  \end{equation}
  For a pair of morphisms $((h_1, \alpha), a):(\vec{n}, \vec{x}) \to (\vec{k}, \vec{s})$ and $((h_2, \beta), b):(\vec{m}, \vec{y}) \to (\vec{l}, \vec{t})$, in $\Lbb(X)$, we define
 another morphism in $\Lbb(X)$ as follows:
 \begin{equation}
  \label{tens-prod-PX-mor}
  ((h_1, \alpha), a) \underset{\Lbb(X)} \otimes ((h_2, \beta), b) :=
  ((h_1, \alpha) \Box (h_2, \beta), \inv{\lambda_{AX}((h_1, \alpha), (h_2, \beta))}((a, b))),
  \end{equation}
 where $\lambda_{\Leins(X)}((h_1, \alpha), (h_2, \beta))$ is the composite functor
 \[
 (\Leins(X)((h_1, \alpha)) \times \Leins(X)((h_2, \beta))) \circ \lambda_{\Leins(X)}(\vec{n}, \vec{m}) =
 \lambda_{\Leins(X)}(\vec{k}, \vec{l}) \circ \Leins(X)((h_1, \alpha) \Box (h_2, \beta)).
 \]
 \begin{prop}
 \label{LbbX-perm}
 The category $\Lbb(X)$ is a permutative category with respect to the tensor
 product defined above.
 \end{prop}
 \begin{proof}
 The category $\Leins$ is a permutative category. Now the proposition follows from theorem \ref{perm-cat-elm}.
 \end{proof}

\subsection[Gabriel Factorization]{Gabriel Factorization}
\label{Gab-Fact}
In analogy with the way a functor can be factored as a fully faithful
functor followed by an essentially surjective one, every strict symmetric monoidal $\Phi:E \to F$ admits a
factorization of the form
\begin{equation*}
 \xymatrix{
 E \ar[rr]^\Phi \ar[rd]_\Gamma && F \\
 & G \ar[ru]_\Delta
 } 
\end{equation*}
where $\Gamma$ is essentially surjective and $\Delta$ is fully faithful. In fact we may suppose that
$\Gamma$ is identity on objects in which case we get the \emph{Gabriel factorization} of $\Phi$.
In order to obtain a Gabriel factorization we define the symmetric monoidal category $G$ as having
the same objects as $E$ and letting, for $c, d \in Ob(G)$,
\[
G(c, d) := F(\Phi(c), \Phi(d)). 
\]
The composition in $G$ is defined via the composition in $F$ in the obvious way. The symmetric monoidal structure on $G$ is defined on objects as follows:
\begin{equation*}
e_1 \underset{G} \otimes e_2 := e_1 \underset{E} \otimes e_2
\end{equation*}
where $e_1, e_2 \in Ob(G) = Ob(E)$. For a pair of morphisms $f_1:e_1 \to h_1$ and $f_2:e_2 \to h_2$
we define
\begin{equation*}
f_1 \underset{G} \otimes f_2 := f_1 \underset{F} \otimes f_2.
\end{equation*}

We recall that $\Pi_1:\Cat \to \Gpd$ is the functor which assigns to each category $C$ the groupoid obtained by inverting all maps in $C$.
\begin{lem}
\label{adj-on-pi1}
Let $F:C \to D$ be a functor which is either a left or a right adjoint, then the induced functor
$\Pi_1(F):\Pi_1(C) \to \Pi_1(D)$ is an equivalence of categories.
\end{lem}
\begin{proof}
We begin the proof by observing that the category of all (small) groupoids $\Gpd$ is enriched over itself. We claim the pair $(\Pi_1 (F), \Pi_1(H))$ is an adjoint equivalence. For each $c \in Ob(C)$, the unit $\eta$ of the adjunction $(F, H)$ provides a map $\eta(c):c \to GF(c)$ in $C$. We observe that this map has an inverse in $\Pi_1(C)$. We define the unit of the adjunction $\eta_{\Pi_1}:id_{\Pi_1(C)} \Rightarrow \Pi_1(H)\Pi(F)$ as follows:
\[
\eta_{\Pi_1}(c) := \eta(c)
\]
for all $c \in Ob(C) = Ob(\Pi_1(C))$. We recall that maps in $\Pi_1(C)$ are composites of maps in $C$ and their formal inverses. In order to check that the family $\lbrace \eta(c) \rbrace_{c \in C}$ defines a natural isomorphism it would be enough to show that it does so on the generating morphisms of $\Pi_1(C)$.  This is obvious for maps in $C$. Let $\inv{f}$ be a (formal) inverse of a map $f:d \to c$ in $C$. By definition of the functor $\Pi_1$, $\Pi_1(F)(\inv{f}) = \inv{F(f)}$. Now we observe that the following diagram commutes in the category $\Pi_1(C)$
\begin{equation*}
\xymatrix{
c \ar[r]^{\eta(c) \ \ \ \ \ \ \  } \ar[d]_{\inv{f}} & \Pi_1(HF)(c) \ar[d]^{\inv{((HF) (f))}} \\
d \ar[r]_{\eta(d) \ \ \ \ \ \ \  } & \Pi_1(HF)(d)
}
\end{equation*}
because the following diagram commutes in $C$
\begin{equation*}
\xymatrix{
c \ar[r]^{\eta(c) \ \ \ }  & HF(c)  \\
d \ar[u]^{f} \ar[r]_{\eta(d) \ \ \ } & HF(d) \ar[u]_{HF(f) }
}
\end{equation*}
and we have the following equalities
\[
\Pi_1(H)\Pi_1(F)(\inv{f})= \Pi_1(HF)(\inv{f}) = \inv{(\Pi(HF)(f))} =  \inv{((HF) (f))} = \inv{(HF(f))}.
 \]
 Thus we have shown that the family of isomorphisms $\eta_{\Pi_1} = \lbrace \eta(c) \rbrace_{c \in C}$ glue together to define a natural isomorphism $\eta_{\Pi_1}:id_{\Pi_1(C)} \Rightarrow \Pi_1(H)\Pi(F)$. A similar argument gives us a counit natural isomorphism $\epsilon_{\Pi_1}:\Pi_1(F)\Pi_(H) \Rightarrow id_{\Pi_1(D)}$.

\end{proof}

\begin{prop}
 \label{Gab-Fact-adjoints}
 Let $E$ be a symmetric monoidal category, $F$ be a symmetric monoidal groupoid and $\Phi:E \to F$ be a strict symmetric monoidal which is a composite of $n$ strict
 symmetric monoidal functors \emph{i.e.} $\Phi = \phi_n \circ \cdots \circ \phi_1$ such that each $\phi_i$
 has either a left or a right adjoint for $1 \le i \le n$. Then the Gabriel category of $\Phi$, $G$,
 is isomorphisc to $\Pi_1(E)$. 
 \begin{proof}
  The functor $\Phi$ has a Gabriel factorization
  \begin{equation*}
 \xymatrix{
 E \ar[rr]^\Phi \ar[rd]_\Gamma && F \\
 & G \ar[ru]_\Delta
 } 
\end{equation*}
see \cite[Sec. 1.1]{GA1}. The above lemma \ref{adj-on-pi1} tells us that the functor $\Pi_1(\Phi):\Pi_1(E) \to \Pi_1(F) = F$ is an equivalence of groupoids.
In the above situation the Gabriel category $G$ is a groupoid therefore $\Pi_1(G) = G$. We recall that $Ob(G) = Ob(E)$ and since $\Pi_1(\Phi)$ is an equivalence of categories therefore for each pair of objects $e_1, e_2 \in E$ we have the following 
\begin{equation*}
E(e_1, e_2) \cong F(\Phi(e_1),\Phi( e_2)) = G(e_1, e_2).
\end{equation*}
Thus the functor $\Gamma$ is an isomorphism of categories.

 \end{proof}

\end{prop}

\section[The model category of Permutative categories]{The model category of Permutative categories}
\label{Nat-Mdl-Str-Perm}
In this section we will describe a model category structure on
the category of all (small) \emph{permutative categories} $\PCat$
and two model category structures on the category of all $\gCats$ $\gCAT$.
The three desired Quillen adjunctions will be amongst the model categories described
here. We begin with the category $\PCat$. The desired model category structure on $\PCat$
is a restriction of the natural model category structure on $\pCat$ which leads us to call it
the \emph{natural model category structure on} $\PCat$.
\subsection[The natural model category $\PCat$]{The natural model category $\PCat$}
We begin this subsection by reviewing permutative categories. A \emph{permutative category}
is a symmetric monoidal category in which the associativity and unit natural isomorphisms
are the identity natural transformations.  A map in $\PCat$ is a \emph{strict monoidal functor} \emph{i.e.} a functor which strictly preserves the
tensor product, the unit object and also the associativity, unit and symmetry isomorphisms.
A permutative category can be equivalently described as an
algebra over a categorical version of the Barratt-Eccles operad, see \cite[Proposition 2.8]{Dunn}. The objective of this subsection is to define a model category structure on $\PCat$. We will obtain this model category structure, in appendix \ref{Mdl-CAT-Perm}, by regarding $\PCat$ as a reflective subcategory of $\Cat$ and transferring the natural model category structure to $\PCat$.

 \begin{thm} 
 \label{nat-model-str-Perm}
 There is a model category structure on the category of all small
 permutative categories and strict symmetric monoidal functors $\PCat$ in which
 \begin{enumerate}
 \item A fibration is a strict symmetric monoidal functor which is also an isofibration of (unbased) categories and
 \item A weak-equivalence is a strict symmetric monoidal functor which is also an equivalence of (unbased) categories.
 \item A cofibration is a strict symmetric monoidal functor having the left lifting property with respect to all maps which are both fibrations and weak equivalences.
  \end{enumerate}
 Further, this model category structure is combinatorial and proper.
 \end{thm}
 A functor $F:C \to D$ is an equivalence of categories if and only if there exists another
 functor $G:D \to C$ and two natural isomorphisms $FG \cong id_C$ and $id_D \cong GF$. We would
 like to have a similar characterization for a weak equivalence in $\PCat$ but unfortunately this is only possible by relaxing the strictness condition on the functor $G$.
 \begin{thm}
 \label{eq-enhanc-perm}
 Let $F:C \to D$ be a strict symmetric monoidal functor in $\PCat$. Any adjunction $(F, G, \eta, \epsilon)$ consisting of a unital right adjoint functor $G:D \to C$ and a pair of unital natural isomorphisms $\epsilon:FG \cong id_D$ and $\eta:id_C \cong GF$, enhances uniquely to a unital symmetric monoidal adjunction i.e.
 there exists a unique natural isomorphism
 \begin{equation}
 \label{SM-enhanc-R-adj}
  \lambda_G := id \circ (\epsilon \times \epsilon) \cdot (\lambda_{GF} \circ id_{G \times G}) \cdot id \circ (\inv{\epsilon} \times \inv{\epsilon}).
 \end{equation}
   enhancing $G = (G, \lambda_G)$, into a unital symmetric monoidal functor such that $\eta$ and $\epsilon$ are unital monoidal natural isomorphims.
 \end{thm}
 \begin{proof}
  We will show that $G$ is a symmetric monoidal functor and $\eta$ and $\epsilon$ are symmetric monoidal natural isomorphisms.
 We will first show the later part. In order to do so we first have to show that the composite functors $GF$ and $FG$ are symmetric monoidal. We consider the following diagram:
 \begin{equation*}
 \label{unit-counit-SM}
 \xymatrix@C=1mm{
 && C \times C \ar@/^2pc/[dd]^{id} \ar@/_2pc/[dd]_{GF \times GF}\ar[rrrrrrrr]^{\TensPFunc{C}} &&&&&&&&C \ar@/^2pc/[dd]^{GF} \ar@/_2pc/[dd]_{id} &&  \\
 &&& \ar@{=>}[ll]_{\ \eta \times \eta} &&&&&& \ar@{=>}[rr]^{\ \eta} &&&& \\
 &&C \times C \ar[rrrrrrrr]_{\TensPFunc{C}} &&&&&&&& C &&
 }
 \end{equation*}
 Even though the definition follows from lemma \ref{SM-Func-closed-isom}, still the above
 diagram is helpful in defining a composite natural isomorphism $\lambda_{GF}:GF \circ (\TensPFunc{C}) \Rightarrow (\TensPFunc{C}) \circ GF \times GF$ as follows:
 \begin{equation*}
 \lambda_{GF} := (id_{\TensPFunc{C}} \circ \eta \times \eta) \cdot (\inv{\eta} \circ id_{\TensPFunc{C}})
 \end{equation*}
 It follows from lemma \ref{SM-Func-closed-isom} that $GF = (GF, \lambda_{GF})$ is a symmetric monoidal functor and $\eta$ is a unital monoidal natural isomorphism. A similar argument shows that the composite functor $FG$ is a symmetric monoidal functor and $\epsilon$ is a unital monoidal natural isomorphism.

 Now it remains to show that $G$ is a symmetric monoidal functor.
 The given adjunction is a unital symmetric monoidal adjunction if and only
 if there exists a unital monoidal isomorphism $\lambda_G$ such that the following digram commutes
\begin{equation*}
 \xymatrix@C=20mm{
 (\TensPFunc{C}) \circ (GF \times GF) \circ (G \times G) \ar@{=>}[r]^{ \ \ \ \ \ \ \ \ id \circ (\inv{\epsilon} \times \inv{\epsilon})} \ar@{=>}[d]_{\lambda_{GF} \circ id_{G \times G}}  & (\TensPFunc{C})  \circ (G \times G) \ar@{=>}[d]^{\lambda_G}   \\
GF \circ (\TensPFunc{C})  \circ (G \times G) \ar@{=>}[r] & G \circ (\TensPFunc{D})
 }
 \end{equation*}
 There is only one way to define $\lambda_G$ which is the following composite natural isomorphism:
  \begin{multline*}
 (\TensPFunc{C}) \circ (G \times G) \overset{id \circ (\inv{\epsilon} \times \inv{\epsilon})} \Rightarrow (\TensPFunc{C}) \circ (G \times G) \circ (FG \times FG) = (\TensPFunc{C}) \circ (GF \times GF) \circ (G \times G) \\
 \overset{\lambda_{GF} \circ id} \Rightarrow  GF \circ (\TensPFunc{C}) \circ (G \times G) \overset{id \circ \lambda_{F} \circ id }\Rightarrow G \circ (\TensPFunc{D}) \circ (FG \times FG) \overset{id \circ (\epsilon \times \epsilon)} \Rightarrow G \circ (\TensPFunc{D}).
 \end{multline*}
 $\lambda_G$ satisfies the symmetry and associativity condition \ref{OpL-symmetry} and \ref{OpL-associativity} because it is a composite of a symmetric monoidal natural isomorphism $\epsilon$, some identity natural isomorphisms and $\lambda_F$ and all of these satisfy \ref{OpL-symmetry} and \ref{OpL-associativity}. Thus we have proved that $(G, \lambda_G)$ is a symmetric monoidal functor.
 \end{proof}
 \begin{coro}
 \label{char-eq-perm}
 A strict symmetric monoidal functor $F:C \to D$ is a weak equivalence in $\PCat$
 if and only if there exists a symmetric monoidal functor $G:D \to C$ and a pair of symmetric monoidal natural isomorphisms $\epsilon:FG \cong id_D$ and $\eta:id_C \cong GF$.
 \end{coro}
 \begin{proof}
 The only if part of the statement of the corollary is obvious. Let us assume that $F$ is a weak equivalence in $\PCat$. By regarding the unit objects of $C$ and $D$ as basepoints, we may view $F$ as a (pointed) functor in $\pCat$ which is a weak equivalence in the natural model category of pointed categories $\pCat$. Then, by definition, there exists a unital functor $G:D \to C$ and two unital natural isomorphisms $\eta:id_C \Rightarrow GF$ and $\epsilon:FG \Rightarrow id_D$. Now the result follows from the theorem.
 \end{proof}

 Next we want to give a characterization of acyclic fibrations in $\PCat$.
 Recall that a functor is an acyclic fibration in $\Cat$ if and only of it is
 an equivalence which is surjective on objects.
 The following corollary provides equivalent characterizations of acyclic fibrations in $\PCat$
 \begin{coro}
 \label{char-acy-fib}
 Given a strict symmetric monoidal functor $F:C \to D$ between permutative categories, the following statements about $F$ are equivalent:
 \begin{enumerate}
 \item $F$ is an acyclic fibration in $\PCat$.
 \item $F$ is an equivalence of categories and surjective on objects.
 \item There exist a unital symmetric monoidal
 functor $G:D \to C$ such that $FG = id_D$ and a unital monoidal natural isomorphism $\eta:id_C \cong GF$.
 
 \end{enumerate}
  \end{coro}
 \begin{proof}
 $(1) \Rightarrow (2)$
 An acyclic fibration $F:C \to D$ in $\PCat$ is an acyclic fibration in $\pCat$, when the unit objects are regarded as basepoints of $C$ and $D$. Every acyclic fibration in $\pCat$ is an equivalence of categories and surjective on objects, see \ref{acyc-fib-surj-eq}.
 
 $(2) \Rightarrow (3)$
  Since every object is cofibrant in the natural model category structure on $\pCat$,
 therefore there exist a unital functor $G:D \to C$ such that $FG = id_D$. There also exists a unital natural isomorphism $\eta:id_C \cong GF$, see \ref{nat-model-str-pointed-cat}. Now the theorem tells us that there is a unique enhancement of $G$ to a unital symmetric monoidal functor such that $\eta$ 
 is a unital monoidal natural isomorphisms.
 
 $(3) \Rightarrow (4)$
 Conversely, if there exists a unital symmetric monoidal functor $(G, \lambda_G)$ and a unital monoidal natural isomorphisms $\eta:id_C \cong GF$ such that $FG = id_D$ then $F$ is an acyclic fibration in $\Cat$ and therefore it is an acyclic fibration in $\PCat$.
 \end{proof}
  
 Every object in $\Cat$ is cofibrant in the natural model structure but this is not the case in $\PCat$. The cofibrant objects satisfy a freeness condition, for example every free permutative category generated by a category is a cofibrant object in $\PCat$. In general,
the notion of cofibrations in $\PCat$ is stronger than that in $\Cat$ as the following lemma suggests:
\begin{lem}
\label{cof-cat-perm}
A cofibration in $\PCat$ is monic on objects.
\end{lem}
\begin{proof}
We begin the proof by defining a permutative category $EC$ whose set of objects is the same as that of $C$.
The category $EC$ has exactly one arrow between any pair of objects. This category gets a
unique permutative category structure which agrees with the permutative category structure of $C$
on objects. The category $EC$ is equipped with a unique strict symmetric monoidal functor
$\iota_C:C \to EC$ which is identity on objects. It is easy to see that the category $EC$
is a groupoid and the terminal map $EC \to \ast$ is an acyclic fibration in $\PCat$.

Let $i:C \to D$ be a cofibration in $\PCat$. Let us assume that the object function
$Ob(i):Ob(C) \to Ob(D)$ is NOT a monomorphism. Now we have the following commutative diagram
\begin{equation*}
 \xymatrix{
 C \ar[r]^{\iota_C} \ar[d]_i & EC \ar[d] \\
 D \ar[r] & \ast
 }
 \end{equation*}
 The above diagram has NO lift because $Ob(i)$ is NOT a monomorphism. Since the
 terminal map $EC \to \ast$ is an acyclic fibration, we have a contradiction to our
 assumption that $i$ is a cofibration in $\PCat$.
 Thus a cofibration in $\PCat$ is always monic on objects.
\end{proof}

Frequently in this paper we would require a characterization of cofibrations in $\PCat$. The object function of a strict symmetric monoidal functor, which is a homomorphisms of monoids, determines whether the functor is a cofibration in $\PCat$. We now recall that the category of monoids has a (weak) factorization system: 
\begin{lem}
There is a weak factorization system $(L, R)$ on the category of monoids, where $R$ is the class of surjective homomorphisms of monoids.
\end{lem}
\begin{proof}
We have to show that each homomorphism of monoids $f:X \to Y$ admits a 
factorisation $f:X  \overset{u} \to E \overset{p} \to Y$ with $u$ lies in the class $L$ and $p$ lies in the class $R$. 
For this, let $q:F(Y) \to Y$ be the homomorphism adjunct to the identity map on $Y$ in the category of sets, where $F(Y)$ is the free monoid generated by the underlying set of $Y$. The homomorphism $q$ is surjective. Let $E = X \coprod F(Y)$ be the coproduct of $X$ and $F(Y)$ in the category of monoids,
and let $u:X \to E$ and $v:F(Y) \to E$ be the inclusions.  
Then there is  a unique map $p:E \to Y$ such that $pu=f$ and $pv=q$.
We claim that $u$ is in $L$ and $p$ is in $R$. The homomorphism $p$ is surjective because $q$ is surjective. In order to show that $u$ is in $L$ we have to show that whenever we have a (outer) commutative diagram in the category of monoids, where $s$ is in $R$
 \begin{equation*}
 \xymatrix{
 X \ar[d]_{u} \ar[r]^g & C \ar@{->>}[d]^{s} \\
E = X \coprod F(Y) \ar[r]_{ \ \ \ \ \ \ \ f} \ar@{-->}[ru]^{L}  & D
 }
  \end{equation*}
  there exists a diagonal filler $L$ which makes the entire diagram commutative. The lower horizontal map $f$ can be viewed as a pair of homomorphisms $f_1:X \to D$ and $f_2:F(Y) \to D$. Now, it would be sufficient to show that there exists a homomorphism $L_2:F(Y) \to C$ such that the following diagram commutes
  \begin{equation*}
 \xymatrix{
 & C \ar@{->>}[d]^{s} \\
 F(Y) \ar[r]_{ \ \ f_2} \ar@{-->}[ru]^{L_2}  & D
 }
  \end{equation*}
  By adjointness, the existence of the homomorphism $L_2$ is equivalent to the existence of a morphism of sets, $T:Y \to U(C)$, such that the following diagram commutes in the category of sets
   \begin{equation*}
 \xymatrix{
 & U(C) \ar@{->>}[d]^{U(s)} \\
 Y \ar[r]_{U(f_2) \ \ } \ar@{-->}[ru]^{T}  & U(D)
 }
  \end{equation*}
where $U$ is the forgetful functor from the category of monoids to the category of sets which is right adjoint to the free monoid functor $F$. Such a map $T$ exists because $s$ is a surjective map of sets. Thus we have shown that the homomorphism $u$ lies in the class $L$.
\end{proof}

 The next lemma provides the desired characterization of cofibrations. 
 \begin{lem}
 A strict symmetric monoidal functor $F:C \to D$ in $\PCat$ is a cofibration if and only if the object function of $F$ lies in the class $L$ \emph{i.e.} it has the left lifting property with respect to surjective homomorphisms of monoids. 
\end{lem}
\begin{proof}
  Let us assume that $Ob(F)$ lies in the class $L$.
  Let $p:X \to Y$ be an acyclic fibration in $\PCat$ then the object function $Ob(p):Ob(X) \to Ob(Y)$
 is a surjective homomorphism. The assumption that $Ob(F)$ lies in $L$ implies that whenever we have the following (outer) commutative diagram there exists a (dotted) diagonal filler $Ob(L)$  which makes the entire diagram commutative in the category of monoids
   
    \begin{equation}
 \label{free-mon-ob-ext}
 \xymatrix{
 Ob(C) \ar[d]_{Ob(F)} \ar[r] & Ob(X) \ar@{->>}[d]^{Ob(p)} \\
 Ob(D) \ar[r] \ar@{-->}[ru]^{Ob(L)}  & Ob(Y)
 }
  \end{equation}
 
 Now we want to show that whenever we have a (outer) commutative diagram, there exists a lift $L$ which makes the following diagram commutative in $\PCat$
  \begin{equation*}
  \label{lift-acy-cof}
 \xymatrix{
 C \ar[d]_F \ar[r] & X \ar@{->>}[d]^{p} \\
 D \ar[r]_G \ar@{-->}[ru]^{L}  & Y
 }
  \end{equation*}
  We will present a construction of the strict symmetric monoidal functor $L$ in the above diagram.  We choose a lift $Ob(L)$ in the diagram \eqref{free-mon-ob-ext} to be the object function of the functor $L$. Since $p$ is an acyclic fibration therefore for each pair of objects $y, z \in Ob(X)$, each function
  \[
  p_{y, z}:X(y, z) \to Y(p(y), p(z))
  \]
  is a bijection.
   For each pair of objects $d_1, d_2$ in $D$, we define a function
  $L_{d_1, d_2}:D(a, b) \to X(L(d_1), L(d_2))$ by the following composite diagram
  \begin{equation*}
  D(d_1, d_2) \overset{G_{d_1, d_2}} \to Y(G(d_1), G(d_2))\overset{\inv{p_{L(d_1), L(d_1)}}}  \to X(L(d_1), L(d_2)).
  \end{equation*}
  In order to check that our definition respects composition, it would be sufficient to check that for another object $d_3 \in Ob(D)$, the following diagram commutes:
   \begin{equation*}
 \xymatrix@C=28mm{
  Y(F(d_1), F(d_2)) \times Y(F(d_2), F(d_3)) \ar[d]_{- \circ -} \ar[r]^{\inv{p_{L(d_1), L(d_2)}} \times \inv{p_{L(d_2), L(d_3)}}} & X(L(d_1), L(d_2)) \times X(L(d_2), L(d_3)) \ar[d]^{- \circ -} \\
 Y(F(d_1), F(d_2)) \ar[r]_{\inv{p_{L(d_1), L(d_3)}}} \ar[r]  & X(L(d_1), L(d_3))
 }
\end{equation*}
The commutativity of the above diagram follows from the commutativity of the following diagram which is the result of the assumption that $p$ is a functor
\begin{equation*}
 \xymatrix@C=24mm{
  Y(F(d_1), F(d_2)) \times Y(F(d_2), F(d_3)) \ar[d]_{- \circ -}  & X(L(d_1), L(d_2)) \times X(L(d_2), L(d_3)) \ar[d]^{- \circ -} \ar[l]_{{p_{L(d_1), L(d_2)}} \times {p_{L(d_2), L(d_3)}}} \\
 Y(F(d_1), F(d_2))   & X(L(d_1), L(d_3)) \ar[l]^{{p_{L(d_1), L(d_3)}}}
 }
\end{equation*}
  Thus we have shown that the family of functions $\lbrace L_{d_1,d_2} \rbrace_{d_1,d_2 \in Ob(D)}$ together with the object function $Ob(L)$ defines a functor $L$. Now we have to check that $L$ is a strict symmetric monoidal functor. Clearly $L(d_1 \otimes d_2) = L(d_1) \otimes L(d_2)$ for each pair of objects $d_1, d_2 \in Ob(D)$ because the object function of $L$ is a homomorphism of monoids. The same equality holds for each pair of maps in $D$. Finally we will show that $L$ strictly preserves the symmetry isomorphism. By definition, $L(\gamma^D_{d_1, d_2}) = \inv{p}_{G(d_1 \otimes d_2), G(d_2 \otimes d_1)} \circ G(\gamma^D_{d_1, d_2})$. Since $G$ is a strict symmetric monoidal functor, therefore $G(\gamma^D_{d_1, d_2}) = \gamma^Y_{G(d_1), G(d_2)} = \gamma^Y_{pL(d_1), pL(d_2)}$. Since $p$
  is also a strict symmetric monoidal functor therefore
  \[
   \inv{p}_{G(d_1 \otimes d_2), G(d_2 \otimes d_1)}(\gamma^Y_{G(d_1), G(d_2)}) = \inv{p}_{G(d_1 \otimes d_2), G(d_2 \otimes d_1)}(\gamma^Y_{pL(d_1), pL(d_2)}) = \gamma^X_{L(d_1), L(d_2)}.
 \]
This means that $L(\gamma^D_{d_1, d_2}) = \gamma^X_{L(d_1), L(d_2)}$ for each pair of objects $d_1, d_2 \in Ob(D)$. Thus we have shown that $L$ is a strict symmetric monoidal functor which makes the entire diagram \eqref{lift-acy-cof} commutative \emph{i.e.} $F$ is an acyclic cofibration in $\PCat$.

Conversely let us assume that $F$ is an acyclic cofibration in $\PCat$. We want to show that the object function of $F$ lies in the class $L$. Let $f:M \to N$ be a surjective homomorphism of monoids. The homomorphism $f$
induces an acyclic fibration $E(f):EM \to EN$, where $EM$ and $EN$ are permutative categories whose monoid of objects are $M$ and $N$ respectively and there is exactly one map between each pair of objects.
By assumption the functor $F$ has the left lifting property with respect to all strict symmetric monoidal functors in the set
\[
\lbrace E(f):EM \to EN : f \in R \rbrace
\]
because every element of this set is an acyclic fibration in $\PCat$.
This implies that the object function of $F$ has the left lifting property with respect to all maps in $R$.

\end{proof}
%
The next proposition provides three equivalent characterizations of acyclic cofibrations in $\PCat$:
 \begin{prop}
 \label{char-acy-cof}
 Let $F:C \to D$ be a strict symmetric monoidal functor  between permutative categories $C$ and $D$, the following conditions on $F$ are equivalent
 \begin{enumerate}
 \item The strict symmetric monoidal functor $F$ is an acyclic cofibration in $\PCat$.
 \item There exist a strict symmetric monoidal
 functor $G:D \to C$ such that $GF = id_C$ and a unital monoidal natural isomorphism $\eta:id_D \cong FG$ which is the unit of an adjunction
 $(G, F, \eta, id):D \rightharpoonup C$.
 \item There is a (permutative) subcategory $S$ of $D$, an isomorphism $H:C \cong S$, in $\PCat$, a strict symmetric monoidal functor $T:D \to S$ and a unital monoidal natural isomorphism $\iota_S \circ T \cong id_D$, where $\iota_S:S \hookrightarrow D$ is the inclusion functor such that $T \circ \iota_S = id_S$ and $F = \iota_S \circ H$
 \end{enumerate}
 \end{prop}
\begin{proof}
$(1) \Rightarrow (2)$ Since $F$ is an acyclic cofibration in $\PCat$ therefore the (outer) commutative diagram has a diagonal filler $G$ such that the entire diagram is commutative in $\PCat$
\begin{equation*}
 \xymatrix{
 C \ar@{=}[r] \ar[d]_{F} & C \ar@{->>}[d]^p \\
 D \ar@{-->}[ru]^G \ar[r] & \ast
 }
  \end{equation*}
Now we construct the monoidal natural isomorphism $\eta:id_D \to FG$. For each object $d \in Ob(D)$ which is in the image of $F$, there exists a unique $c \in Ob(C)$ such that $d = F(c)$. In this case we define $\eta(d) = id_d$. Let $d \in Ob(D)$ lie outside the image of $F$. Since $F$ is an equivalence of categories, we may choose an object $c \in Ob(C)$ and an isomorphism $i_d^c:F(c) \cong d$ such that for each arrow $f:d \to e$ in $D$, there exists a unique arrow $g:c \to a$ in $C$ which makes the following diagram commutative in $D$:
\begin{equation}
\label{nat-isom-eq}
 \xymatrix{
 F(c) \ar[r]^{i_d} \ar[d]_{F(g)} & d \ar[d]^f \\
 F(a) \ar[r]_{i_e} & e
 }
  \end{equation}
  Whenever $d = d_1 \otimes d_2$, we may choose $c = c_1 \otimes c_2$ and $i_d = i_{d_1} \otimes i_{d_2}$.
 This gives us a composite isomorphism
 \[
  \eta(d):= d \overset{\inv{(i_d)}} \to FG(F(c)) \overset{FG(i_d)}       \to FG(d).
 \]
 In light of the commutative diagram \eqref{nat-isom-eq}, it is easy to see that this isomorphism is natural. Thus we have defined a (unital) natural isomorphism $\eta:id_D \Rightarrow FG$.
 Our choice for each pair of objects $d_1, d_2 \in Ob(D)$
 for $i_{d_1 \otimes d_2} = i_{d_1} \otimes i_{d_2}$ garuntees that $\eta$ is a monoidal natural isomorphism \emph{i.e.} $\eta(d_1 \otimes d_2) = \eta(d_1) \otimes \eta(d_2)$.

$(2) \Rightarrow (3)$
The permutative subcategory $S \subseteq D$ is the full subcategory of $D$
whose objects lie in the image of $F$ \emph{i.e.} the object set of $S$ is defined as follows:
\[
Ob(S) := \lbrace F(c) : c \in Ob(C) \rbrace
\]
If $F(c_1)$ and $F(c_2)$ lie in $S$ then $F(c_1) \underset{D} \otimes F(c_2) = F(c_1 \underset{C} \otimes c_2)$ also lies in $S$. Thus $S$ is a permutative subcategory. The isomorphism $H$ is obtained by restricting the codomain of $F$ to $S$.
The left adjoint of $\iota_S$ is the composite functor $FG$. The counit of the adjunction $(FG, \iota_S)$ is the identity natural isomorphism. The unit (monoidal) natural isomorphism is just $\eta$. Thus $S$ is reflective. 

$(3) \Rightarrow (1)$
If we assume $(4)$ then any (outer) commutative square
\begin{equation*}
 \xymatrix{
 C \ar[r]^{Q} \ar[d]_{F} & X \ar@{->>}[d]^p \\
 D \ar@{-->}[ru]^L \ar[r]_{R} & Y
 }
  \end{equation*}
 where $p$ is a fibration in $\PCat$, would have a diagonal filler $L$
 if and only if the lower square in the following (solid arrow) commutative diagram has a diagonal filler $K$
 \begin{equation*}
 \xymatrix{
 C \ar[d]_H \ar[rd]^Q \\
 S \ar[r] \ar@{_{(}->}[d]_{\iota_S} & X \ar@{->>}[d]^p \\
 D \ar@{-->}[ru]_K \ar[r]_{R} & Y
 }
  \end{equation*}
 Since $\iota_S$ has a strict symmetric monoidal left adjoint $T$ with an identity counit, therefore the composite $K = T \circ Q \circ \inv{H}$ is a diagonal filler of the lower square such that entire diagram commutes.
This implies that $F$ has the left lifting property with respect to fibrations in $\PCat$. Thus $F$ is an acyclic cofibration in $\PCat$.
\end{proof}

 \section[The model category of coherently commutatve monoidal categories]{The model category of coherently commutatve monoidal categories}
A $\gCat$ is a functor from $\gop$ to $\Cat$.
The category of 
functors from $\gop$ to $\Cat$ and natural transformations between them $\CatHom{\gop}{\Cat}{}$ will be denoted by $\gCAT$. We begin by describing a model category structure on
$\gCAT$ which is often referred to either as the \emph{projective model category structure}
or the \emph{strict model category structure}. Following \cite{schwede} we will use the latter terminology.
\begin{df}
 A morphism $F:X \to Y$ of $\gCats$ is called
 \begin{enumerate}
 \item a \emph{strict equivalence} of $\gCats$  if it is degreewise weak equivalence in the natural model category structure on $\Cat$ \emph{i.e.} $F(n^+):X(n^+) \to Y(n^+)$ is an equivalence of categories.
 
 \item a \emph{strict fibration}  of $\gCats$ if it is degreewise a fibration in the natural model category structure on $\Cat$ \emph{i.e.} $F(n^+):X(n^+) \to Y(n^+)$ is an isofibration.
 
 \item a \emph{Q-cofibration}  of $\gCats$ if it has the left lifting property with respect to
 all morphisms which are both strict weak equivalence and strict fibrations of $\gCats$.

  \end{enumerate}
 \end{df}
 In light of proposition \ref{char-of-acyc-fib} we observe that a map of $\gCats$ $F:X \to Y$ is a strict acyclic fibration of $\gCats$ if and only if it has the right lifting property with respect to all maps in the set
 \begin{equation}
 \label{gen-cof}
 \I = \lbrace \gn{n} \times \partial_0,
 \gn{n} \times \partial_1,  \gn{n} \times \partial_2 \mid \forall n \in Ob(\N) \rbrace.
 \end{equation}
 We further observe, in light of proposition \ref{char-of-isofib}, that $F$ is a strict fibration if and only it has the right lifting property with respect to all maps in the set
 \begin{equation}
 \label{gen-acyc-cof}
 \J = \lbrace \gn{n} \times i_0, \gn{n} \times i_1 \mid \forall n \in Ob(\N) \rbrace_{}.
 \end{equation}
 
 \begin{thm}
 \label{str-mdl-cat-gCat}
 Strict equivalences, strict fibrations and Q-cofibrations of $\gCats$ provide the category $\gCAT$ with a combinatorial model category structure.
 \end{thm}
 A proof of this proposition is given in \cite[Proposition A.3.3.2]{JL}.
 
 To each pair of objects $(X, C) \in Ob(\gCAT) \times Ob(\Cat)$ we can assign a $\gCat$ 
 $\TensP{X}{C}{}$
 which is defined in degree $n$ as follows:
 \[
 (\TensP{X}{C}{})(n^+) :=  X(n^+) \times C,
 \]
  This assignment
 is functorial in both variables and therefore we have a bifunctor
 \[
 \TensP{-}{-}{}:\gS\times \Cat \to \gCAT.
 \]
 Now we will define a couple of function objects for the category $\gCAT$.
 The first function object enriches the category $\gCAT$ over
 $\Cat$ \emph{i.e.} there is a bifunctor
 \[
 \MapC{-}{-}{\gCAT}:\gCAT^{op} \times \gS\to \Cat
 \]
 which assigns to any pair of objects $(X, Y) \in Ob(\gCAT) \times Ob(\gCAT)$, a category
 $\MapC{X}{Y}{\gCAT}$ whose set of objects is the following
 \[
 Ob(\MapC{X}{Y}{\gCAT}) := Hom_{\gCAT}(X, Y)
 \]
 and the morphism set of this category are defined as follows:
 \[
 Mor(\MapC{X}{Y}{\gCAT}) := Hom_{\gCAT}(X \times I, Y)
 \]
 For any $\gCat$ $X$, the functor $\TensP{X}{-}{}:\Cat \to \gCAT$ is
 left adjoint to the functor $\MapC{X}{-}{\gCAT}:\gS\to \Cat$. The counit of this adjunction
 is the evaluation map $ev:\TensP{X}{\MapC{X}{Y}{\gCAT}}{} \to Y$
 and the unit is the obvious functor $C \to \MapC{X}{\TensP{X}{C}{}}{\gCAT}$.
 To any pair of objects $(C, X) \in Ob(\Cat) \times Ob(\gCAT)$ we can assign a $\gCat$ $\bHom{C}{X}{\gCAT}$
 which is defined in degree $n$ as follows:
 \[
 (\bHom{C}{X}{\gCAT})(n^+) := \CatHom{C}{X(n^+)} \ .
 \]
 This assignment
 is functorial in both variable and therefore we have a bifunctor
 \[
 \bHom{-}{-}{\gCAT}:\Cat^{op} \times \gS\to \gCAT.
 \]
 For any $\gCat$ $X$, the functor $\bHom{-}{X}{\gCAT}:\Cat \to \gCAT^{op}$ is
 left adjoint to the functor $\MapC{-}{X}{\gCAT}:\gCAT^{op} \to \Cat$. 
 The following proposition summarizes the above discussion.
\begin{prop}
\label{two-var-adj-cat-gcat}
There is an adjunction of two variables
\begin{multline}
\label{two-var-adj-gcat}
(\TensP{-}{-}{}, \bHom{-}{-}{\gCAT}, \MapC{-}{-}{\gCAT}) : \gS\times \Cat
\\  \to \gCAT.
\end{multline}

\end{prop}

\begin{df}
  \label{Q-adj-2-var}
  Given model categories $\C$, $\D$ and $\E$, an adjunction
  of two variables, $\left(\otimes, \bhom_\C, \map_\C, \phi, \psi \right):
  \C \times \D \to \E$, is called a \emph{Quillen adjunction of two variables}, if, given a
  cofibration $f:U \to V$ in $\C$ and a cofibration $g:W \to X$ in $\D$,
  the induced map
  \[
   f \Box g:(V \otimes W) \underset{U \otimes W} \coprod (U \otimes X) \to V \otimes X
  \]
  is a cofibration in $\E$ that is trivial if either $f$ or $g$ is.
  We will refer to the left adjoint of a Quillen adjunction of two
  variables as a \emph{Quillen bifunctor}.

 \end{df}
 The following lemma provides three equivalent characterizations
 of the notion of a Quillen bifunctor. These will be useful in this paper
 in establishing enriched model category structures.
 \begin{lem}\cite[Lemma 4.2.2]{Hovey}
  \label{Q-bifunctor-char}
   Given model categories $\C$, $\D$ and $\E$, an adjunction
  of two variables, $\left(\otimes, \bhom_\C, \map_\C, \phi, \psi \right):
  \C \times \D \to \E$. Then the following conditions are equivalent:
  \begin{enumerate}
   \item [(1)] $\otimes:\C \times \D \to \E$ is a Quillen bifunctor.
   
   \item[(2)] Given a cofibration $g:W \to X$ in $\D$ and a fibration
   $p:Y \to Z$ in $\E$, the induced map
   \[
    \bhom_\C^{\Box}(g, p):\bhom_\C(X, Y) \to \bhom_\C(X, Z)
    \underset{\bhom_\C(W, Z)}\times \bhom_\C(W, Y)
   \]
   is a fibration in $\C$ that is trivial if either $g$ or $p$ is a
   weak equivalence in their respective model categories.
   
   \item[(3)] Given a cofibration $f:U \to V$ in $\C$ and a fibration
   $p:Y \to Z$ in $\E$, the induced map
   \[
    \map_\C^{\Box}(f, p):\map_\C(V, Y) \to \map_\C(V, Z) \underset{\map_\C(W, Z)}\times \map_\C(W, Y)
   \]
   is a fibration in $\C$ that is trivial if either $f$ or $p$ is a
   weak equivalence in their respective model categories.
  \end{enumerate}
 \end{lem}

\begin{df}
  \label{enrich-model-cat}
  Let $\bigS$ be a monoidal model category. An \emph{$\bigS$-enriched
  model category} or simply an $\bigS$-model category is an $\bigS$ enriched category $\bigA$ equipped with
  a model category structure (on its underlying category) such that
  there is a Quillen adjunction of two variables, see definition
   \ref{Q-adj-2-var}, $\left(\otimes, \bhom_{\bigA}, \map_{\bigA}, \phi, \psi \right):
  \bigA \times \bigS \to \bigA$.
  
 \end{df}
 \begin{thm}
  \label{enrich-GamCAT-CAT}
  The strict model category of $\gCats$, $\gCAT$, is a $\Cat$-enriched model category.
 \end{thm}
 \begin{proof}
  We will show that the adjunction of two variables \eqref{two-var-adj-gcat}
  is a Quillen adjunction for the strict model category structure
 on $\gCAT$ and the natural model category structure on $\Cat$.
  In order to do so, we will verify condition
 (2) of Lemma \ref{Q-bifunctor-char}. Let $g:C \to D$ be a cofibration
 in $\Cat$ and let $p:Y \to Z$ be a strict fibration of $\gCats$,
 we have to show that the induced map
 \[
  \bhom^{\Box}_{\gCAT}(g, p):\bHom{X}{Y}{\gCAT} \to \bHom{D}{Z}{\gCAT}
  \underset{\bHom{C}{Z}{\gCAT}} \times \bHom{C}{Y}{\gCAT}
 \]
 is a fibration in $\Cat$ which is acyclic if either of $g$ or $p$ is
 acyclic. It would be sufficient to check that the above morphism is degreewise
 a fibration in $\Cat$, i.e. for all $n^+ \in \gop$, the morphism
 \begin{equation*}
  \bhom^{\Box}_{\gCAT}(g, p)(n^+): [D, Y(n^+)]  \to
  [D, Z(n^+)] \underset{[C, Z(n^+)]} \times  [C, Y(n^+)] ,
 \end{equation*}
 is a fibration in $\Cat$. This follows from the observations that the functor $p(n^+):Y(n^+) \to Z(n^+)$
 is a fibration in $\Cat$ and the natural model category 
 $\Cat$ is a $\Cat$-enriched model category whose enrichment is provided by the bifunctor $[-, -]$.
 \end{proof}
 
 Let $X$ and $Y$ be two $\gCats$, the \emph{Day convolution product} of $X$ and $Y$ denoted by $X \ast Y$ is defined as follows:
 \begin{equation}
 \label{Day-Con-prod}
 X \ast Y(n^+) := \int^{(k^+, l^+) \in \gop} \gop(k^+\wedge l^+, n^+) \times X(k^+) \times Y(l^+).
 \end{equation}
 Equivalently, one may define the Day convolution product of $X$ and $Y$ as the left Kan extension of their \emph{external tensor product} $X \overline \times Y$ along the smash product functor
 \[
 - \wedge - :\gop \times \gop \to \gop.
 \]
 we recall that the external tensor product $X \overline \times Y$ is a bifunctor
 \begin{equation*}
 X \overline \times Y:\gop \times \gop \to \Cat
 \end{equation*}
 which is defined on objects by 
 \[
 X \overline \times Y(m^+, n^+) = X(m^+) \times Y(n^+).
 \]
\begin{prop}
\label{GCAT-SM}
The category of all $\gCats$ $\gCAT$ is a symmetric monoidal category under the Day convolution product \eqref{Day-Con-prod}.
The unit of the symmetric monoidal structure is the representable $\gCat$ $\gn{1}$.
\end{prop}
Next we define an internal function object of the category $\gCat$
which we will denote by
\begin{equation}
\label{Int-Map-GCAT}
 \MGCat{-}{-}:\gCAT^{op} \times \gS\to \gCAT.
 \end{equation}
 Let $X$ and $Y$ be two $\gCats$, we define the $\gCat$ $\MGCat{X}{Y}$ as follows:
 \begin{equation*}
 \MGCat{X}{Y}(n^+) := \MapC{X \ast \gn{n}}{Y}{\gCAT}.
 \end{equation*}
 \begin{prop}
 \label{closed-SM-cat-GCat}
 The category $\gCAT$ is a closed symmetric monoidal category under the Day convolution product. The internal Hom is given by the bifunctor \eqref{Int-Map-GCAT} defined above.
 \end{prop}
 The above proposition implies that for each $n \in \Nat$
 the functor $- \ast \gn{n}:\gS\to \gCAT$ has a right adjoint $\MGCat{\gn{n}}{-}:\gS\to \gCAT$. It follows from
 \cite[Thm.]{} that the functor $- \ast \gn{n}$ has another right adjoint which we denote by $-(n^+ \wedge -):\gS\to \gCAT$. We will denote $-(n^+ \wedge -)(X)$ by $X(n^+ \wedge -)$, where $X$ is a $\gCat$. The $\gCat$ $X(n^+ \wedge -)$ is defined by the following composite:
 \begin{equation}
 \label{defn-X-n-wedge}
  \gop \overset{n^+ \wedge -} \to \gop \overset{X} \to \Cat.
 \end{equation}
 The following proposition sums up this observation:
 \begin{prop}
  \label{rt-adjs-DayCP}
  There is a natural isomorphism
  \[
  \phi: -(n^+ \wedge -) \cong \MGCat{\gn{n}}{-}.
  \]
  In particular,
  for each $\gCat$ $X$ there is an isomorphism of $\gCats$
  \[
  \phi(X):X(n^+ \wedge -) \cong \MGCat{\gn{n}}{X}.
  \]

 \end{prop}

 The next theorem shows that the strict model category $\gCAT$ is compatible with the Day convolution product.
 \begin{thm}
 \label{SM-closed-mdl-str-GCat}
 The strict Q-model category $\gCAT$ is a symmetric monoidal closed model category under the Day convolution product.
 \end{thm}
 \begin{proof}
 Using the adjointness which follows from proposition \ref{closed-SM-cat-GCat} one can show that if a map $f:U \to V$ is a (acyclic) cofibration in the strict model category $\gCAT$ then the induced map
 $f \ast \gn{n}:U \ast \gn{n} \to V \ast \gn{n}$ is also a (acyclic) cofibration in the strict model category for all $n \in \Nat$. By $(3)$ of Lemma \ref{Q-bifunctor-char} it is sufficient to show that whenever
 $f$ is a cofibration and $p:Y \to Z$ is a fibration then the map
 \[
    \MGBoxCat{f}{p}:\MGCat{V}{Y} \to \MGCat{V}{Z} \underset{\MGCat{U}{Z}}\times \MGCat{U}{Y}.
 \]
   is a fibration in $\gCAT$ which is acyclic if either $f$ or $p$ is a
   weak equivalence. The above map is a (acyclic) fibration if and only if the map
\begin{multline*}
 \MGBoxCat{f \ast \gn{n}}{p}(n^+): \MapC{V \ast \gn{n}}{Y}{\gCAT} \to \\
  \MapC{V \ast \gn{n}}{Z}{\gCAT}
  \underset{\MapC{U \ast \gn{n}}{Z}{\gCAT}} \times \MapC{U \ast \gn{n}}{Y}{\gCAT}
\end{multline*}

    is a (acyclic) fibration in $\Cat$ for all $n \in \Nat$. Since $f \ast \gn{n}$ is a cofibration as abserved above, the result follows from theorem \ref{enrich-GamCAT-CAT}.
 \end{proof}
\subsection[The model category of coherently commutative monoidal categories]{The model category of coherently commutative monoidal categories}
\label{EInf-Cat}
  The objective of this subsection is to construct a new model
  category structure on the category $\gCAT$. This new model
  category is obtained by localizing the strict model category
 defined above and we call it the \emph{The model category of coherently commutative monoidal categories}. We will refer to this new model category structure
 as the \emph{model category structure of coherently commutative monoidal categories} on $\gCAT$. The aim of this new model structure is to endow its homotopy category with a semi-additive structure. In other words we want this new
 model category structure to have finite \emph{homotopy biproducts}.  We go on further to show that this new model category is symmetric monoidal with respect to
 the \emph{Day convolution product}, see \cite{Day2}.    We begin by recalling the notion of
 a \emph{left Bousfield localization}:
 
 \begin{df}
  Let $\M$ be a model category and let $\S$ be a class of maps in $\M$.
  The left Bousfield localization of $\M$ with respect to $\S$
  is a model category structure $L_\S\M$ on the underlying category of $\M$
  such that
  \begin{enumerate}
   \item The class of cofibrations of $L_\S\M$ is the same as the
   class of cofibrations of $\M$.
   
   \item A map $f:A \to B$ is a weak equivalence in $L_\S\M$ if it is an $\S$-local equivalence,
   namely, for every fibrant $\S$-local object $X$, the induced map on homotopy
   function complexes
   \[
    f^\ast:Map_{\M}^h(B, X) \to Map_{\M}^h(A, X)
   \]
   is a homotopy equivalence of simplicial sets. Recall
   that an object $X$ is called fibrant $\S$-local if $X$ is fibrant
    in $\M$ and for every element
   $g:K \to L$ of the set $\S$, the induced map on
   homotopy function complexes
   \[
    g^\ast:Map_{\M}^h(L, X) \to Map_{\M}^h(K, X)
   \]
   is a weak homotopy equivalence of simplicial sets.
    
  \end{enumerate}
where $\HMapC{-}{-}{\M}$ is the simplicial function object associated with the strict model category $\M$, see \cite{DK80}, \cite{DK1980} and \cite{DK3}.
 \end{df}
 \begin{rem}
 \label{simp-enrichment-str-mdl-cat}
 The strict model category of all $\gCats$ is a $\Cat$-enriched model category
 by theorem \ref{enrich-GamCAT-CAT}, this enrichment is equivalent to having a Quillen adjunction
 $- \otimes \gn{1} : \Cat \rightleftharpoons \gCAT:\MapC{\gn{1}}{-}{\gCAT}$ whose left adjoint preserves the tensor product,
 see \cite[Lemma 3.6]{CB1}. Further the adjunction $\tau_1:\sSets \rightleftharpoons \Cat:N$,
 see \cite{AJ1}, is a Quillen adjunction with respect to the Joyal model category
 structure on $\sSets$ and natural model category structure on $\Cat$ whose
 left adjoint $\tau_1$ preserves finite products (and thus the tensor product in the cartesian closed
 Joyal model category of simplicial sets).
 Again lemma \cite[Lemma 3.6]{CB1} implies that the strict model category of $\gCats$ is a $\sSets$-enriched model category
 with respect to the Joyal model category structure on $\sSets$. The right Hom bifunctor
 of this enrichment
 \[
 \Map({-}{-}):\gCAT^{op} \times \gS\to \sSets
 \]
 assigns to a pair of objects $(X, C)$, a simplicial sets $\Map({X}{C})$
 which is defined as follows:
 \[
 \Map({X},{C}) := N(\MapC{X}{C}{\gCAT}).
 \]
  \end{rem}
 
 We want to construct a left Bousfield localization of
 the strict model category of $\gCats$. For each pair $k^+, l^+ \in \gop$,
 we have the obvious \emph{projection maps} in $\gSC$
 \[
  \delta^{k+l}_k:(k+l)^+ \to k^+ \ \ \ \ and \ \ \ \ \delta^{k+l}_l:(k+l)^+ \to l^+.
 \]
 The maps
 \[
 \gop(\delta^{k+l}_k,-):\Gamma^{k} \to \Gamma^{k+l} \ \ \ \ and \ \ \ \ 
 \gop(\delta^{k+l}_l,-):\Gamma^{l} \to \Gamma^{k+l} 
 \]
 induce a map of $\gSs$ on the coproduct which we denote as follows:
 \[
  h_k^l:\Gamma^l \sqcup \Gamma^l \to \Gamma^{l+k}.
 \]
 
 We now define a class of
 maps $\E_\infty\S$ in $\gCAT$:
 \begin{equation*}
  \E_\infty\S := \lbrace h_k^l:\Gamma^l \sqcup \Gamma^l \to \Gamma^{l+k}:
  l, k \in \mathbb{Z}^+ \rbrace
 \end{equation*}
 We recall that $I$ is the category with two objects and one non-identity arrow between them. We define another class of maps in $\gCAT$:
 \begin{equation*}
 I \times \E_\infty\S := \lbrace I \times h_k^l: h_k^l \in \E_\infty\S \rbrace
 \end{equation*}
 \begin{df}
  We call a $\gCat$ $X$ a $(I \times \E_\infty\S)$-\emph{local object}
  if, for each map $h_k^l \in \E_\infty\S$, the induced simplicial map
  \begin{multline*}
  \HMapC{\Delta[n] \times h_k^l}{X}{\gCAT}: \HMapC{\Delta[n] \times \Gamma^{k+l}}{X}{\gCAT} \to \\
  \HMapC{\Delta[n] \times (\Gamma^l \sqcup \Gamma^l)}{X}{\gCAT},
  \end{multline*}
 is a homotopy equivalence of simplicial sets for all $n \ge 0$ where $\HMapC{-}{-}{\gCAT}$ is the simplicial function object associated with the strict model category $\gCAT$, see \cite{DK80}, \cite{DK1980} and \cite{DK3}.
 \end{df}
  Remark \eqref{simp-enrichment-str-mdl-cat} above and appendix \ref{Cat-Local} tell us that a model for $\HMapC{X}{Y}{\gCAT}$ is the Kan complex
 $J(N(\MapC{X}{Y}{\gCAT}))$ which is the maximal kan complex contained in the quasicategory $N(\MapC{X}{Y}{\gCAT})$.

  The following proposition gives a characterization of
 $\E_\infty\S$-local objects
  \begin{prop}
  \label{char-CCMC}
 \begin{sloppypar}
 A $\gCat$ $X$ is a $(I \times \E_\infty\S)$-local object in $\gCAT$ if and only if it satisfies the Segal condition namely the functor
  \end{sloppypar}
 \begin{equation*}
 (X(\partition{k+l}{k}), X(\partition{k+l}{l})):X(k+l^+) \to X(k^+) \times X(l^+)
 \end{equation*}
 is an equivalence of categories for all $k^+, l^+ \in \Ob(\gop)$.
 \end{prop}
 \begin{proof}
 We begin the proof by observing that each element of the set $\E_\infty\S$ is a map of $\gCats$ between cofibrant $\gCats$. Lemma \ref{char-lo-QCat-en} implies that $X$ is a $(I \times \E_\infty\S)$-local object if and only if the following functor
 \begin{equation*}
 \MapC{h^k_l}{X}{\gCAT}:\MapC{\gn{k+l}}{X}{\gCAT} \to \MapC{\gn{k} \sqcup \gn{l}}{X}{\gCAT}
 \end{equation*}
 is an equivalence of (ordinary) categories.
 We observe that we have the following commutative square in $\Cat$
 \begin{equation*}
 \xymatrix@C=24mm{
 	\MapC{\gn{k+l}}{X}{\gCAT}
  \ar[d]_{\cong}\ar[r]^{ \MapC{h^k_l}{X}{\gCAT}} &  \MapC{\gn{k} \sqcup \gn{l}}{X}{\gCAT}  \ar[d]^{\cong} \\
 X((k+l)^+) \ar[r]_{(X(\partition{k+l}{k}), X(\partition{k+l}{l}))} & X(k^+) \times X(l^+)
 }
 \end{equation*}
 This implies that the functor $(X(\partition{k+l}{k}), X(\partition{k+l}{l}))$ is an equivalence of categories if and only if the functor $ \MapC{h^k_l}{X}{\gCAT}$ is an equivalence of categories.
  \end{proof}

\begin{df}
 \label{CCMC}
 We will refer to a $(I \times \E_\infty\S)$-local object as a \emph{coherently commutative monoidal category}.
 \end{df}
 
 \begin{df}
 A morphism of $\gCats$ $F:X \to Y$ is a $(I \times \E_\infty\S)$-local equivalence if for each coherently commutative monoidal category $Z$
 the following simplicial map
 \[
 \HMapC{F}{Z}{\gCAT}:\HMapC{Y}{Z}{\gCAT} \to \HMapC{X}{Z}{\gCAT}
 \]
 is a homotopy equivalence of simplicial sets.
 \end{df}
 \begin{prop}
  \label{char-CCME}
 \begin{sloppypar}
 A morphism between two cofibrant $\gCats$ $F:X \to Y$ is an $(I \times \E_\infty\S)$-local equivalence if and only if the functor  \end{sloppypar}
 \begin{equation*}
 \MapC{F}{Z}{\gCAT}:\MapC{Y}{Z}{\gCAT} \to \MapC{X}{Z}{\gCAT}
  \end{equation*}
 is an equivalence of categories for each coherently commutative monoidal category $Z$.
 \end{prop}

 \begin{df}
  We will refer to a $(I \times \E_\infty \S)$-local equivalence as an \emph{equivalence of coheretly commutative monoidal categories}.
 \end{df}
 The main result of this section is about constructing a
 new model category structure on the category $\gCAT$,
 by localizing the strict model category of $\gCats$ with respect to
 morphisms in the set $\E_\infty\S$. We recall the following theorem
 which will be the main tool in the construction of the
 desired model category. This theorem first appeared in an unpublished work \cite{smith}
 but a proof was later provided by Barwick in \cite{CB1}.
 \begin{thm} \cite[Theorem 2.11]{CB1}
 \label{local-tool}
 If $\M$ is a combinatorial model category and $\S$ is a small
set of homotopy classes of morphisms of $\M$, the left Bousfield localization $L_\S\M$ of
$\M$ along any set representing $\S$ exists and satisfies the following conditions.
\begin{enumerate}
\item The model category $L_\S\M$ is left proper and combinatorial.
\item As a category, $L_\S\M$ is simply $\M$.
\item The cofibrations of $L_\S\M$ are exactly those of $\M$.
\item The fibrant objects of $L_\S\M$ are the fibrant $\S$-local objects $Z$ of $\M$.
\item The weak equivalences of $L_\S\M$ are the $\S$-local equivalences.
\end{enumerate}
\end{thm}
\begin{thm}
 \label{loc-semi-add}
 There is a closed, left proper, combinatorial model category structure on
 the category of $\gCats$, $\gCAT$, in which
 \begin{enumerate}
 \item The class of cofibrations is the same as the class of
 Q-cofibrations of $\gCats$.
 \item The weak equivalences are $\E_\infty$-equivalences.
 \end{enumerate}
 \begin{sloppypar}
 An object is fibrant in this model category if and only if it is a
  coherently commutative monoidal category. 
   \end{sloppypar}
 \end{thm}
 \begin{proof}
 The strict model category of $\gCats$ is a combinatorial
 model category therefore the existence of the model structure
 follows from theorem \ref{local-tool} stated above. \end{proof}
 \begin{nota}
 The model category constructed in theorem \ref{loc-semi-add} will
 be called the model category of $\EinCs$.
 \end{nota}
  \begin{sloppypar}
 The rest of this section is devoted to proving that the model
 category of $\EinCs$  is a symmetric monoidal closed model category.
 In order to do so we will need some general results which we
 state and prove now.
 \end{sloppypar}
 
\begin{prop}
 \label{criterion-acy-cof}
 A cofibration, $f:A \to B$, between cofibrant objects in a model category $\C$ is
 a weak equivalence in $\C$ if and only if it has the right
 lifting property with respect to all fibrations between fibrant
 objects in $\C$.
 \end{prop}
 \begin{proof}
 The unique terminal map $B \to \ast$ can be factored into
 an acyclic cofibration $\eta_B:B \to R(B)$ followed by a fibration
 $R(B) \to \ast$. The composite map $\eta_B \circ f$ can again be
 factored as an acyclic cofibration followed by a fibration $R(f)$ as shown
 in the following diagram:
 \begin{equation*}
  \label{fact-comp-cof}
 \xymatrix{
 A \ar[d]_f \ar[r]^{\eta_A} & R(A) \ar[d]^{R(f)} \\
 B \ar@{.>}[ru] \ar[r]_{\eta_B} &R(B)
 }
 \end{equation*}
 Since $B$ is fibrant and $R(f)$ is a fibration, therefore $R(A)$
 is a fibrant object in $\C$. Thus $R(f)$ is a fibration between
 fibrant objects in $\C$ and now by assumption, the 
 dotted arrow exists which makes the whole diagram commutative.
 Since both $\eta_A$ and $\eta_B$ are acyclic cofibrations, therefore
 the two out of six property of model categories implies that
 the map $F$ is a weak-equivalence in the model category $\C$.
 
 \end{proof}

 \begin{prop}
 \label{expn-ho-prod}
 \begin{sloppypar}
 Let $X$ be a $\EinC$, then for each $n \in Ob(\N)$,
 the $\gCat$ $X(n^+ \wedge -)$ is also a $\EinC$.
 \end{sloppypar}
\end{prop}
\begin{proof}
 We begin by observing that
 $X(n^+ \wedge -)(1^+) = X(n^+)$ and since $X$ is fibrant,
 the pointed category $X(n^+)$ is equivalent to $\overset{n}
 {\underset{1} \prod} X(1^+)$. Notice that the isomorphisms $(n^+ \wedge (k+l)^+) \cong 
 \overset{n} {\underset{1} \vee} (k+l)^+ \cong (\overset{n} {\underset{1}
 \vee} k^+) \vee (\overset{n} {\underset{1} \vee} l^+) \cong ((\overset{n}
 {\underset{1}  \vee} k^+) + (\overset{n} {\underset{1} \vee} l^+))$.
 The two projection maps $\delta^{k+l}_k:(k+l)^+ \to k^+$ and $\delta^{k+l}_l:(k+l)^+ \to l^+$
 induce an equivalence of categories $X((\overset{n} {\underset{1}
 \vee} k^+) + (\overset{n} {\underset{1} \vee} l^+)) \to
 X(\overset{n} {\underset{1} \vee} k^+) \times X(\overset{n} {\underset{1}
 \vee} l^+)$. Composing with the isomorphisms above, we get
 the following equivalence of pointed simplicial sets
 $X(n^+ \wedge -)((k + l)^+) \to X(n^+ \wedge -)(k^+) \times
 X(n^+ \wedge -)(l^+)$.
 \end{proof}
 
 \begin{coro}
  For each coherently commutative monoidal category $X$,
  the mapping object $\MGCat{\gn{n}}{X}$ is also a coherently commutative monoidal category for each $n \in \Nat$.
 \end{coro}
 \begin{proof}
  The corollary follows from proposition \ref{rt-adjs-DayCP}.
 \end{proof}

 The category $\gop$ is a symmetric monoidal category with respect
 to the smash product of pointed sets. In other words the smash product of
 pointed sets defines a
  bi-functor $- \wedge -: \Gamma^{op} \times \Gamma^{op} \to
 \Gamma^{op}$. For each pair $k^+, l^+ \in Ob(\gop)$, there are two
 natural transformations
 
 \[
 \delta^{k+l}_k \wedge -: (k+l)^+ \wedge - \Rightarrow k^+ \wedge - \ \ \ \ \text{and} \ \ \ \
 \delta^{k+l}_l \wedge -: (k+l)^+ \wedge - \Rightarrow l^+ \wedge -.
 \]
  Horizontal composition of either of these two natural transformations
  with a $\gCat$ $X$
 determines a morphism of $\gCats$
 \[
  id_X \circ (\delta^{k+l}_k \wedge -) =:X(\delta^{k+l}_k \wedge -):X((k+l)^+ \wedge -) \to X(k^+ \wedge -).
 \]

 \begin{prop}
 \label{Hom-prod-fib-gCat}
 Let $X$ be an $\EinC$, then for each pair $(k,l) \in Ob(\N) \times Ob(\N)$,
 the following morphism
 \[
 (X(\delta^{k+l}_k \wedge -),X(\delta^{k+l}_l \wedge -)):X((k+l)^+ \wedge -) \to
 X(k^+ \wedge -) \times X(l^+ \wedge -)
 \]
 is a strict equivalence of $\gCats$.
 \end{prop}
 
 Using the previous two propositions, we now show
 that the mapping space functor $\MGCat{-}{-}$ provides
 the homotopically correct function object when the
 domain is cofibrant and codomain is fibrant.
 
 \begin{lem}
  \label{Ein-map-obj-Cat}
 Let $W$ be a Q-cofibrant $\gCat$ and let $X$ be a coherently commutative monoidal category. Then the mapping object $\MGCat{W}{X}$ is also a $\EinC$.
 \end{lem}
 \begin{proof}
 We begin by recalling that
 \[
 \MGCat{W}{X}((k+l)^+) = \MapC{W}{X((k+l)^+ \wedge -)}{\gCAT}.
 \]
 Since $X$ is a $\EinC$, therefore
 $X((k+l)^+ \wedge -)$ is also a $\EinC$, for all $k, l \ge 0$ according to proposition \ref{expn-ho-prod}. The proposition \ref{Hom-prod-fib-gCat}
 tells us that the map $(X(\delta^{k+l}_k \wedge -), X(\delta^{k+l}_l \wedge -))$ is a strict equivalence of $\gCats$. Now Theorem \ref{SM-closed-mdl-str-GCat}
 implies that the following induced functor on the mapping (pointed) categories
 \begin{multline*}
 (\MapC{W}{X(\delta^{k+l}_k \wedge -)}{\gCAT}, \MapC{W}{X(\delta^{k+l}_l \wedge -)}{\gCAT}):\MapC{W}{X((k+l)^+ \wedge -)}{\gCAT} \\
 \to \MapC{W}{X((k)^+ \wedge -)}{\gCAT} \times \MapC{W}{X((l)^+ \wedge -)}{\gCAT}
 \end{multline*}
 is an equivalence of categories.
 \end{proof}

Finally we get to the main result of this section. All the lemmas proved above will be useful in proving the following theorem:
\begin{thm}
\label{SM-closed-CCMC}
The model category of coherently commutative monoidal categories is a symmetric monoidal closed model category under the Day convolution product.
\end{thm}
\begin{proof}
Let $i:U \to V$ be a Q-cofibration and $j:Y \to Z$ be another Q-cofibration. We will prove the theorem by showing that the following \emph{pushout product} morphism
\begin{equation*}
i \Box j:U \ast Z \underset{U \ast Y} \coprod V \ast Y \to V \ast Z 
\end{equation*}

is a Q-cofibration which is also an equivalence of coherently commutative monoidal categories whenever either $i$ or $j$ is an equivalence of coherently commutative monoidal categories.
We first deal with the case of $i$ being a generating Q-cofibraion. The closed symmetric monoidal model structure on the strict Q-model category, see theorem \ref{SM-closed-mdl-str-GCat}, implies that $i \Box j$ is a Q-cofibration. Let us assume that $j$ is an acyclic Q-cofibration \emph{i.e.} the Q-cofibration $j$ is also an equivalence of coherently commutative monoidal categories. According to proposition \ref{criterion-acy-cof} the Q-cofibration $i \Box j$ is an equivalence of coherently commutative monoidal categories if and only if it has the left lifting property with respect to all strict fibrations of $\gCats$ between coherently commutative monoidal categories. Let $p:W \to X$ be a strict fibration between two coherently commutative monoidal categories. A (dotted) lifting arrow would exists in the following diagram
\begin{equation*}
\xymatrix{
U \ast Z \underset{U \ast Y} \coprod V \ast Y \ar[r] \ar[d] & W \ar[d]^p \\
V \ast Z \ar@{..>}[ru] \ar[r] & Y
}
\end{equation*}
if and only if a (dotted) lifting arrow exists in the following adjoint commutative diagram
\begin{equation*}
\xymatrix{
C \ar[r] \ar[d]_j & \MGCat{V}{W} \ar[d]^{(j^*, p^*)} \\
D \ar@{..>}[ru] \ar[r] & \MGCat{U}{X} \underset{\MGCat{U}{Y}} \times \MGCat{V}{Y}
}
\end{equation*}
The map $(j^*, p^*)$ is a strict fibration of $\gCats$ by lemma \ref{Q-bifunctor-char} and theorem \ref{SM-closed-mdl-str-GCat}. Further the observation that both $V$ and $U$ are Q-cofibrant and the above lemma \ref{Ein-map-obj-Cat} together imply that $(j^*, p^*)$ is a strict fibration between coherently commutative monoidal categories and therefore a fibration in the model category of coherently commutative monoidal categories. Since $j$ is an acyclic cofibration by assumption therefore the (dotted) lifting arrow exists in the above diagram. Thus we have shown that if $i$ is a Q-cofibration and $j$ is a Q-cofibration which is also a weak equivalence in the model category of coherently commutative monoidal categories then $i \Box j$ is an acyclic cofibration in the model category of coherently commutative monoidal categories.
Now we deal with the general case of $i$ being an arbitrary Q-cofibration. Consider the following set:
\begin{equation*}
S = \lbrace i:U \to V | \ i \Box j \textit{ \ is an acyclic cofibration in }\rbrace
\end{equation*}
where $\gCAT$ is endowed with the model structure of coherently commutative monoidal categories. We have proved above that the set $S$ contains all generating Q-cofibrations.
We observe that the set $S$ is closed under pushouts, transfinite compositions and retracts. Thus $S$ contains all Q-cofibrations.
Thus we have proved that $i \Box j$ is a cofibration which is acyclic if $j$ is acyclic. The same argument as above when applied to the second argument of the Box product (\emph{i.e.} in the variable j)
shows that $i \Box j$ is an acyclic cofibration whenever $i$ is an acyclic cofibration in the model category of coherently commutative monoids.

\end{proof}

Finally we will provide a characterization of cofibrations in the model category of coherently commutative monoidal categories.
\begin{nota}
	For each morphism $F:X \to Y$ in $\gCAT$ we get a collection of object functions  $\lbrace Ob(F(k^+)):Ob(X(k^+) )\to Ob(Y(k^+)): k^+ \in Ob(\gop) \rbrace$. These functions glue together into a $\Gamma$-set, which we denote by $Ob(F)$, whose structure maps are just the object functions of the structure maps (functors) of $F$, \emph{i.e.}
	\[
	Ob(F)(f) := Ob(F(f)),
	\]
	for each $f \in Mor(\gop)$.
	\end{nota}

\begin{lem}
	\label{char-cof}
	A map $i:A \to B$ in $\gCAT$ is a cofibration in the model category of coherently commutative monoidal categories if and only if the $\Gamma$-set $Ob(i)$ has the left lifting property with respect to every surjective morphism of $\Gamma$-sets.
	\end{lem}
\begin{proof}
	Let $p:X \to Y$ be an acyclic fibration in the model category of coherently commutative monoidal categories. For any (outer) commutative diagram in $\gCAT$ 
	\begin{equation*}
	\xymatrix{
		A \ar[r]^H \ar[d]_i & X \ar[d]^p \\
		B \ar[r]_G \ar@{-->}[ru]^{L} & Y
	}
\end{equation*}
     we will construct a dotted arrow $L$ which will make the whole diagram commutative.
	The morphism $Ob(p)$ is a surjective map of $\Gamma$-sets. By assumption the (outer) commutative diagram
	\begin{equation*}
	\xymatrix{
	Ob(A) \ar[r]^{Ob(H)} \ar[d]_{Ob(i)} & Ob(X) \ar[d]^{Ob(p)} \\
	Ob(B) \ar[r]_{Ob(G)} \ar@{-->}[ru]^{Ob(L)} & Ob(Y)
   }
	\end{equation*}
	has a (dotted) lifting arrow (of $\Gamma$-sets) $Ob(L)$ which makes the whole diagram commutative. For each $k^+ \in Ob(\gop)$ and each pair of objects $a, b \in Ob(B(k^+))$ we want to define a function
	\[
	L(k^+)_{a, b}:Mor_{B(k^+)}(a, b) \to Mor_{X(k^+)}(Ob(L)(k^+)(a), Ob(L)(k^+)(b)).
	\]
	By assumption, the map $p$ is an acyclic fibration therefore for each $k^+ \in Ob(\gop)$, the functor $p(k^+)$ is an acyclic fibration in the natural model category structure on $\Cat$. This implies that the function
	\[
	p(k^+)_{v, w}:Mor_{X(k^+)}(Ob(L)(k^+)(a), Ob(L)(k^+)(b)) \to Mor_{Y(k^+)}(G(k^+)(a), G(k^+)(b))
	\]
	is a bijection.
	Now we define the function $L(k^+)_{a, b}$ to be the following composite
	\begin{multline*}
	Mor_{B(k^+)}(a, b) \overset{G(k^+)_{a, b}} \to Mor_{Y(k^+)}(G(k^+)(a), G(k^+)(b)) \\
	\overset{\inv{p(k^+)_{v, w}}} \to Mor_{X(k^+)}(Ob(L)(k^+)(a), Ob(L)(k^+)(b)),
	\end{multline*}
   where $(v, w) = (Ob(L)(k^+)(a), Ob(L)(k^+)(b))$
  An argument similar to the one in the proof of Lemma \ref{char-str-sym-mon-cof} shows that the above collection of maps $\lbrace L(k^+)_{a, b} : a, b \in Ob(B(k^+)) \rbrace$
  defines a functor $L(k^+)$ whose object function is the same as $Ob(L)(k^+)$. Now we want to check whether this collection of functors $\lbrace L(k^+): k^+ \in Ob(\gop) \rbrace$ glues together into a morphism of $\gCats$. It would be sufficient to show that for any $f:k^+ \to m^+ \in Mor(\gop)$, the following diagram commutes:
  \begin{equation}
 \label{mor-of-gCat}
  \xymatrix@C=24mm{
  Mor_{B(k^+)}(a, b) \ar[r]^{L(k^+)_{a, b}  } \ar[d]_{B(f)_{a, b}} & Mor_{X(k^+)}(v, w) \ar[d]^{X(f)_{v, w}} \\
Mor_{B(m^+)}(B(f)(a), B(f)(b)) \ar[r]_{ \ \ \ \ \ \ \ \ L(m^+)_{B(f)(a), B(f)(b)}  }  & Mor_{X(m^+)} (x, y)
}  
  \end{equation}
 where $(x, y) = (Ob(L)(m^+)(B(f)(a)), Ob(L)(m^+)(B(f)(b))) $. Since $p$ and $G$ are maps of $\gCats$ therefore we have following (solid arrow) commutative diagram of mapping sets:
 \begin{equation*}
 \xymatrix@C=16mm{
 	& Mor_{B(m^+)}(B(f)(a), B(f)(b)) \ar[r]^ { \ \ \ \ G(m^+) } & Mor_{Y(m^+)}(q, r) \ar@/^2pc/@{-->}[dd]^{\inv{p(m^+)_{x, y}}} \\
 Mor_{B(k^+)}(a, b) \ar[r]_{G(k^+)_{a, b} \ \ \ \ \ \ \ \ \ \ \ } \ar[ru]^{B(f)_{a, b}}  & Mor_{Y(k^+)}(G(k^+)(a), G(k^+)(b)) \ar[ru]_{Y(f)} \ar@/^2pc/@{-->}[d]^{\inv{p(k^+)_{v, w}}}  \\
 &  Mor_{X(k^+)}(v, w) \ar[u]^{p(k^+)_{v, w}}  \ar[r]_{X(f)_{v, w}} & Mor_{X(m^+)}(x, y) \ar[uu]^{p(m^+)_{x, y}}
}
 \end{equation*}
 where
 \[
  (q, r) = (G(m^+)(B(f)(a)), G(m^+)(B(f)(b))) = (Y(f)(G(k^+)(a)), Y(f)(G(k^+)(b))).
  \]
  Since the dotted arrows are the inverses to the associated solid arrows therefore the entire diagram is commutative. This commutativity implies that the diagram \ref{mor-of-gCat} is commutative.

	\end{proof}

  \section[Segal's Nerve functor]{Segal's Nerve functor}
\label{SegNerve}
In the paper \cite{segal}, Segal described a construction of a $\gCat$ from a (small) symmetric monoidal category which we call the \emph{Segal's nerve} of the symmetric monoidal category. His construction defined a functor which we call \emph{Segal's nerve functor}. This functor was further studied in \cite{SK}, \cite{May4}. Oplax and lax variations of Segal's Nerve functor were defined in \cite{mandell}, \cite{mandell2}. In this section we will review Segal's Nerve functor and describe a new representation of Segal's nerve functor. The Segal's nerve functor is built on a family of discrete categories which carry a partial symmetric monoidal structure namely
$\lbrace \P(n) \rbrace_{n \in \Nat}$, where $\P(n)$ denotes the \emph{power set} of the finite set $\underline{n}$. The partial symmetric monoidal structure is just the union of disjoint subsets of $\underline{n}$. Segal's nerve of a symmetric monoidal category consists of, in degree $n$, the category of all functors which preserve this partial symmetric monoidal structure upto isomorphism. One of our goals in this section is to further clarify the situation by firstly defining an unnormalized version of Segal's nerve and secondly by \emph{completing} the partial symmetric monoidal structures and thereby present a construction of Segal's nerve using (strict) symmetric monoidal functors. For each $n \in \Nat$ we construct a permutative category $\PStr(n)$ which is equipped with an inclusion functor $i:\P(n) \to \PStr(n)$ and which satisfies the following universal property:
\begin{equation*}
\xymatrix{
\P(n) \ar[d]_i \ar[r]^F &  C  \\
\PStr(n) \ar@{-->}[ru]_{\exists!}
}
\end{equation*}
where the functor $F$ satisfies $F(S \sqcup T) \cong F(S) \underset{C} \otimes F(T)$ for all $S, T \in \P(n)$
and $S \cap T = \emptyset$. This allows us to define our \emph{unnormalized  Segal's nerve}, in degree $n$, as follows:
\[
\K(C)(n^+) := \StrSMHom{\PStr(n)}{C}.
\]
We will show the existence of a functor $\PStr:\gCAT \to \PCat$ which is a left adjoint to the unnormalized Segal's Nerve functor $\K$. The main objective of this section is to show that the adjoint pair of functors $(\PStr, \KSeg)$ induces a Quillen equivalence between the natural model category $\PCat$ and the model category of coherently commutative monoidal categories $\gCAT$.
 We begin by reviewing Segal's construction.
 \begin{df}
 \label{n-Seg-bike}
 An $nth$ \emph{Segal bicycle} into a symmetric monoidal category
 $C$ is a triple $ \Psi = (\Psi, \sigma_{\Psi}, u_\Psi)$, where $\Psi$ 
 is a family of objects of $C$
 \[
 \Psi = \lbrace \Psi(S): \Psi(S) \in Ob(C) \rbrace_{S \in \P(n)}
 \]
 and $\sigma_\Psi$
 is a family of morphisms of $C$
 \begin{equation*}
  \sigma_\Psi = \lbrace \sigma_\Psi((S, T)):\Psi(S \sqcup T) \overset{\cong} \to \Psi(S) \otimes \Psi(T) : f_{(S, T)} \in Mor(C) \rbrace_{ (S, T) \in \Lambda },
  \end{equation*} 
  where the indexing set $\Lambda := \lbrace (S, T) : S, T \subseteq n, S \cap T = \emptyset \rbrace$. Finally $u_\Psi:\Psi(\emptyset) \overset{\cong} \to \unit{C}$ is an isomorphism in $C$. This triple is subject to the following conditions:
 \begin{enumerate}[label = {SB.\arabic*}, ref={SB.\arabic*}]
 \item \label{SB-unit} For each $S \in \P(n)$, the following diagram commutes:
 \begin{equation*}
\xymatrix@C=14mm{
 \Psi(\emptyset) \otimes \Psi(S) \ar[d]_{u_\psi} & \Psi(S) \ar@{=}[d] \ar[l]_{ \ \ \ \ \sigma_\Psi((\emptyset, S))} \ar[r]^{ \sigma_\Psi((S, \emptyset)) \ \ \ }
  & \Psi(S) \otimes \Psi(\emptyset) \ar[d]^{u_\Psi} \\
 \unit{C} \otimes \Psi(S) \ar[r]_{\beta_l} &\Psi(S)   & \Psi(S) \otimes \unit{C} \ar[l]^{\beta_r}
}
\end{equation*}
\item \label{SB-assoc} For each triple $S, T, U \in \P(n)$ of mutually disjoint subsets of $\underline{n}$, the following diagram commutes:
\begin{equation*}
\xymatrix@C=14mm{
 \Psi(S \sqcup T \sqcup U) \ar[d]_{\sigma_\Psi((S \sqcup T, U))} \ar[rr]^{\sigma_\Psi((S, T \coprod U))} && \Psi(S) \otimes \Psi(T \sqcup U) \ar[d]^{id \otimes \sigma_\Psi((T,U))  } \\
 \Psi(S \sqcup T) \otimes \Psi(U) \ar[rd]_{\sigma_\Psi((S,T))  \otimes id \ \ \ \ } &&\Psi(S) \otimes (\Psi(T) \otimes \Psi(U)) \ar[ld]^{ \ \ \ \ \alpha_C(\Psi(S), \Psi(T), \Psi(U))} \\
 & (\Psi(S) \otimes \Psi(T)) \otimes \Psi(U))
}
\end{equation*}
\item \label{SB-symm} For each pair $S, T \in \P(n)$ of disjoint subsets of $\underline{n}$, the following diagram commutes:
\begin{equation*}
\xymatrix@C=18mm{
 \Psi(S \sqcup T) \ar[r]^{\sigma_\Psi((S, T))} \ar[rd]_{\sigma_\Psi((T, S))} & \Psi(S) \otimes \Psi(T) \ar[d]^{\gamma^C((S, T))} \\
 &\Psi(T) \otimes \Psi(S)
}
\end{equation*}
 \end{enumerate}
 \end{df}
 Next we define morphisms of Segal bicycles
 \begin{df}
 \label{Mor-n-SB}
 A morphism of $nth$ Segal bicycles $\tau:(\Psi, \sigma_{\Psi}) \to (\Omega, \sigma_{\Omega})$ in a symmetric monoidal category $C$ is a family of maps of $C$
 \[
 \tau = \lbrace \tau(S):\Psi(S) \to \Omega(S) \rbrace_{S \in \P(\ud{n})}
 \]
 such that the following two diagram commutes
 \begin{equation*}
\xymatrix@C=14mm{
 \Psi(S \coprod T) \ar[r]^{\tau(S \coprod T)} \ar[d]_{\sigma_\Psi((S, T))} & \Omega(S \coprod T) \ar[d]^{\sigma_\Omega((S, T))}
  & \Psi(\emptyset) \ar[rd]_{u_\Psi} \ar[rr]^{\tau(\emptyset)} && \Omega(\emptyset) \ar[ld]^{u_\Omega} \\
 \Psi(S) \otimes \Psi(T) \ar[r]_{\tau(S) \otimes \tau(T)} & \Omega(S) \otimes \Omega(T)
 &  & \unit{C}
}
\end{equation*}
 \end{df}
 Given a permutative category $C$, $n$th Segal bicycles
 and morphisms of $nth$ Segal bicycles define a category
 which we denote by $\K(C)(n^+)$. 
 Next we want to compare the notion of an $nth$ Segal bicycle to that of a strict bicycle in the permutative category $C$:
 \begin{lem}
 	\label{strict-bikes-SB-eq}
 Let $C$ be a permutative category. For each $\underline{n}$, the category $\KSeg(C)(n)$ is isomorphic to the category of all strict bicycles from $\gn{n}$ to $C$,
  $\StrBikes{\gn{n}}{C}$.
 \end{lem}
 \begin{proof}
 We will prove the lemma by constructing a pair of inverse functors. We begin by defining a functor $F:\KSeg(C)(n) \to \StrBikes{\gn{n}}{C}$. Let $(\Psi, \sigma_\Psi, u_\Psi) \in Ob(\KSeg(C)(n))$, then for each $k \in Ob(\N)$ we define a functor $\phi(k):\gnk{n}{k} \to C$ as follows:
 \[
 \phi(k)(f) := \Psi(Supp(f)),
 \]
 where $Supp(f) \subset n$ is the \emph{support of} $f:n^+ \to k^+$. The collection
 $\lbrace \phi(k) \rbrace_{k \in Ob(\N)}$ defines a (strict) cone $\L_\Phi$ because for any $h:k \to l$ in $Mor(\N)$, $Supp(f) = Supp(h \circ f)$. Similarly for each pair $(k,l) \in Ob(\N) \times Ob(\N)$, we will define a natural isomorphism $\sigma_\Phi(k,l):\phi(k+l) \Rightarrow \phi(k) \odot \phi(l)$.
To each $g \in \gnk{n}{(k+l)}$ we can associate a pair of functions $(g_k, g_l)$ whose components are defined by the following two diagrams
\[
 \xymatrix{
 n^+ \ar[r]^{g \ \ } \ar[rd]_{g_k} & (k+l)^+  \ar[d]^{\partition{k+l}{k}}  &&  n^+ \ar[r]^{g \ \ } \ar[rd]_{g_l} & (k+l)^+  \ar[d]^{\partition{k+l}{l}} \\
 & k^+ && & l^+
 }
\]
we define
\[
 \sigma_\Phi(k,l)(g) := \sigma_\Psi(Supp(g_k), Supp(g_l))
\]
and $u_\Phi(0) := u_\Psi$. We define the object function of $F$ as follows:
\[
 F((\Psi, \sigma_\Psi, u_\Psi)) := (\Phi, \sigma_\Phi, u_\Phi).
\]

For each morphism $\tau:\Psi \to \Upsilon$ in $\K(C)(n^+)$, we define a map of bicycles $F(\tau)$
as follows:
\[
 F(\tau) = \lbrace F(\tau)(f) := \tau(Supp(f)):\phi(k)(f) \to \upsilon(k)(f)  :  f \in \underset{k \in \Nat}{\sqcup} \gn{n}(k^+) \rbrace
\]
It follows from definition \ref{Mor-n-SB} that $F(\tau)$ is a morphism of strict bicycles. 

 Now we define the inverse functor $G$ as follows:
 Let $(\L, \sigma, \tau) = \Bike{\Phi}{\gn{n}}{C}$ be a strict bicycle where $\L = (\phi, \sigma)$ is its underlying strict cone, see appendix \ref{NotionBike}, then we define an $n$-$th$ unnormalized Segal bicycle $(G(\Phi), \sigma_{G(\Phi)}, u_{G(\Phi)})$ as follows:
 \[
 G(\Phi) = \lbrace G(\Phi)(S): G(\Phi)(S) = \phi(f_S) \rbrace_{S \in \P(\ud{n})}
 \]
 where $f_S:n^+ \to S^+$ is the projection map whose support is $S \subseteq \ud{n}$. Next we define the family of isomorphism $\sigma_{G(\Phi)}$. For each pair $(S, T)$ of disjoint subsets of $\ud{n}$ we get a map $f_{S + T}:n^+ \to (S + T)^+$ in the category $\gn{n}((S + T)^+)$. We now define
 \[
 \sigma_{G(\Phi)}((S, T)) := \sigma((S, T))(f_{S + T}).
 \]
 Finally $u_{G(\Phi)} := \tau(id_{0^+})$.
 
 A map of strict bicycles $F: \Phi \to \Psi$
 determines a collection of maps of $C$  
 \[
 G(F) := \lbrace G(F)(S):G(\Phi)(S) \to G(\Psi)(S) \rbrace_{S \in \P(\ud{n})} 
 \]
This collection glues together to define a map of $n$-$th$ unnormalized Segal bicycles.
It is easy to see that the functors $F$ and $G$ are inverse of each other.
 \end{proof}

\begin{df}
\label{P-Str-n}
 For each $n \in \Nat$ we will now define a permutative groupoid $\PStr(n)$. The objects of this
 groupoid are finite sequences of subsets of $\ud{n}$. We will denote an object of this groupoid by $(S_1, S_2, \dots, S_r)$, where $S_1, \dots S_r$ are subsets of $\ud{n}$. A morphism $(S_1, S_2, \dots, S_r) \to (T_1, T_2, \dots, T_k)$ is an
 isomorphism of finite sets $F:S_1 \sqcup S_2 \sqcup \dots \sqcup S_r \overset{\cong} \to T_1 \sqcup T_2 \sqcup \dots \sqcup T_k$ such that the following
 diagram commutes
 \begin{equation*}
\xymatrix{
 S_1 \sqcup S_2 \sqcup \dots \sqcup S_r \ar[rr]^{F} \ar[rd] && T_1 \sqcup T_2 \sqcup \dots \sqcup T_k \ar[ld] \\
 &\ud{n}
}
\end{equation*}
 where the diagonal maps are the unique inclusions of the coproducts into $\ud{n}$.
 \end{df}

We define a subcategory $\PLbb(n)$ of $\Lbb(\gn{n})$ which will turn out to be a \emph{coreflective} subcategory. 
 An object of $\PLbb(n)$ is a finite sequence $S = (S_1, S_2, \dots, S_r)$,
where $S_i$ is a subset of $\underline{n}$ for $1 \le i \le r$.
\begin{nota}
An object $S = (S_1, S_2, \dots, S_r) \in Ob(\PLbb(n))$ uniquely determines a morphism (of unbased sets)
$\sigma(S):\underset{i=1} {\overset{r} \sqcup} S_i \to \underline{n}$.
We will refer to the map $\sigma(S)$ as the \emph{canonical inclusion of $S$ in} $\underline{n}$.
\end{nota}
\begin{nota}
An object $S = (S_1, S_2, \dots, S_r) \in Ob(\PLbb(n))$ uniquely determines a morphism (of unbased sets)
$\textit{Ind}(S):\underset{i=1} {\overset{r} \sqcup} S_i \to \underline{r}$.
We will refer to the map $\textit{Ind}(S)$ as the \emph{canonical index of $S$}.
\end{nota}
Given another object $T = (T_1, T_2, \dots, T_s)$ in $\PLbb(n)$, where
$T_j$ is a subset of $\underline{n}$ for $1 \le j \le r$, a morphism $F: S \to T$ in $\PLbb(n)$
is a pair $(h, p)$, where $h:\underline{s} \to \underline{r}$ is a map of finite unbased sets
and $p:\underset{i=1} {\overset{r} \sqcup} S_i \to \underset{j=1} {\overset{s} \sqcup} T_j$
is a bijection. The pair is subject to the following condition:
\begin{enumerate}
\item The following diagram commutes:
\begin{equation*}
 \xymatrix{
 & \underline{n} & \\
 \underset{i=1} {\overset{r} \sqcup} S_i \ar[ru]^{\sigma(S)}   \ar[d]_{\textit{Ind}(S)} \ar[rr]^{p}
 &&  \underset{j=1} {\overset{s} \sqcup} T_j \ar[d]^{\textit{Ind}(T)} \ar[lu]_{\sigma(T)}  \\ 
\underline{r} && \underline{s} \ar[ll]_{h}
 }
 \end{equation*}

\end{enumerate}
\begin{rem}
\label{cont-fun-PLbb}
The construction above defines a contravariant functor $\PLbb(-):\gop \to \PCat$. A map $f:n^+ \to m^+$ in $\gop$ defines a strict symmetric monoidal functor
$\PLbb(f):\PLbb(m) \to \PLbb(n)$. An object $(S_1, S_2, \dots, S_r) \in \PLbb(m)$ is mapped by this functor to $(\inv{f}(S_1), \inv{f}(S_2), \dots, \inv{f}(S_r)) \in \PLbb(n)$.
\end{rem}

\begin{df}
\label{P-Lax-n}
For each $n \in \Nat$ we define a permutative category $\Lbb(\gn{n})$ as follows:
\[
\Lbb(\gn{n}) := \int^{\vec{k} \in \Leins} \Leins(\gn{n})(\vec{k}),
\]
see \ref{Groth-Cons-gCat}.
This construction defines a functor $\Lbb(\gn{-})$ which is the following composite
\begin{equation*}
\gop \overset{y} \to \gCAT \overset{\Leins(-)} \to \SMHom{\Leins}{\Cat} \to \PCat
\end{equation*}
where $y$ is the Yoneda functor. $\Leins(-)$ is the functor defined \ref{SM-Ext-OL}.
\end{df}
\begin{prop}
The category $\PLbb(n)$ is isomorphic to the full subcategory of $\Lbb(\gn{n})$ whose
objects are finite sequences of projection maps in $\gop$ having domain $n^+$.
\end{prop}
\begin{proof}
We will define a functor $G:\PLbb(n) \to \Lbb(\gn{n})$. This functor is defined on objects as follows:
\[
G((S_1, S_2, \dots, S_r)) := (f_1, f_2, \dots, f_r),
\]
where $S = (S_1, S_2, \dots, S_r)$ is an object in $\PLbb(n)$ and each $f_i:n^+ \to S_i^+$ is a projection map onto $S_i$.
Let $T = (T_1, T_2, \dots, T_s)$ be another object in $\PLbb(n)$. A map
$(h, p):S \to T$ in $\PLbb(n)$ is also a map in $\mathfrak{L}$ such that
$\mathfrak{L}(\gn{n})((h, p))(S) = T$. This defines the functor $G$ which is fully faithful.
\end{proof}
\begin{rem}
	\label{coref-sub-cat-PLbb}
	The category $\PLbb(n)$ is a coreflective subcategory of $\Lbb(\gn{n})$ as a result of \ref{isom-PStr-gn} and \ref{PL-cref-SC-L}.
	\end{rem}
\begin{rem}
\label{inc-str-lax-Pn}
The functor $G:\PLbb(n) \to \Lbb(\gn{n})$ defined in the proof above is a strict symmetric monoidal functor. This functor is a component of a natural transformation between two contravariant functors
\[
i:\PLbb(-) \Rightarrow \Lbb(\gn{-}),
\]
where $i(n^+) := G$ and $\Lbb(\gn{-}):\gop \to \PCat$ is the functor that maps $n^+$ to $\Lbb(\gn{n})$.
\end{rem}
\begin{rem}
\label{inc-str-lax-PN-pi1}
Composing the natural transformation $i$ in the above remark with the functor $\Pi_1$
gives us a natural equivalence
\begin{equation*}
id_{\Pi_1} \circ i :\Pi_1\circ \PLbb(-) \Rightarrow \Pi_1\circ \Lbb(\gn{-})
\end{equation*}
\emph{i.e.} for each $n^+ \in \gop$ the functor
\begin{equation*}
id_{\Pi_1} \circ i (n^+) :\Pi_1(\PLbb(n)) \to \Pi_1(\Lbb(\gn{n}))
\end{equation*}
is an equivalence of categories.
\end{rem}

\begin{nota}
 Let $C$ be a strict symmetric monoidal category. Let us denote by $\underline{C}$ the underlying
groupoid of $C$ \emph{i.e.} the groupoid obtained by discarding all non-invertible maps in $C$.
We recall that $\underline{C}$ retains the strict symmetric monoidal structure of $C$.
 \end{nota}
 We recall that a strict bicycle $\Bike{\Psi = (\psi, \sigma_\Psi, u_\Psi)}{\gn{n} }{C}$ defines an oplax symmetric monoidal functor
 $\underline{\Psi}:\N \to \OplaxExp{\gn{n}}{\underline{C}}$, see appendix \ref{bikes-as-oplax-sections}. For each $\underline{k} \in \N$,
 there is a functor $\underline{\Psi}(k):\gn{n}(k) \to \underline{C}$ which is defined as follows:
 \[
 \underline{\Psi}(k)(f) := \psi(k)(f)
 \]
 for each $f \in \gn{n}(k)$. For each morphism $h:k \to l$ in $\N$,
 $\underline{\Psi}(h) := id$, \emph{i.e.} the identity natural transformation.
 The family of natural isomorphisms $\sigma_\Phi$ and the unit natural isomorphism $u_\Psi$ provide an oplax symmetric monoidal structure on $\underline{\Psi}$.
 The oplax symmetric monoidal inclusion functor $i:\OplaxExp{\gn{n}}{\underline{C}} \to \OplaxExp{\mathfrak{L}(\gn{n})}{\underline{C}}$ provides the following composite oplax symmetric monoidal functor
 \[
 \N \overset{\Phi} \to \OplaxExp{\gn{n}}{\underline{C}} \overset{i} \to \OplaxExp{\mathfrak{L}(\gn{n})}{\underline{C}}.
 \]
 This composite oplax symmetric monoidal functor extends uniquely, along the inclusion $\N \hookrightarrow \mathfrak{L}$, into a
 strict symmetric monoidal functor $\mathfrak{L}(\underline{\Psi}):\mathfrak{L} \to \OplaxExp{\mathfrak{L}(\gn{n})}{\underline{C}}$, see appendix \ref{OL-to-SM}. This functor uniquely determines another strict symmetric monoidal functor
 \begin{equation}
 \label{SSM-ext-str-bike}
 \widetilde{\Phi}:\Lbb(\gn{n}) \to \underline{C}.
 \end{equation}
 

Let $\phi:\PLbb(n) \to \underline{C}$ be a strict symmetric monoidal functor. The functor $\phi$
determines a strict Segal bicycle $(F(\phi), \sigma_{F(\phi)}, u_{F(\phi)})$ which we now define.
For each $S \subseteq \underline{n}$, we define $F(\phi)(S) = \phi((S))$.
The collection of isomorphism $\sigma_{F(\phi)}$ is defined as follows:
\[
\sigma_{F(\phi)}((S, T)) := \phi((m, id)),
\]
where $(m, id):(S \sqcup T) \to (S, T)$ is a map in $\PLbb(n)$ whose first component is given by the multiplication map $m:\underline{2} \to \underline{1}$.
Finally, the isomorphism $u_{F(\phi)}$ is defined as follows:
\[
u_{F(\phi)} := \phi((id, i)),
\]
where $(id, i):(\emptyset) \to ()$ is the following map in $\PLbb(n)$:
\begin{equation*}
\xymatrix{
\emptyset \ar@{=}[r] \ar[d] & \emptyset \ar[d] \\
\underline{1} & \underline{0} \ar[l]^{i}
}
\end{equation*}
 The conditions $SB1, SB2$ and $SB3$
follow from the strict symmetric monoidal functor structure of $\phi$. The above construction defines a functor
\[
F:\StrSMHom{\PLbb(n)}{\underline{C}}  \to \KSeg(C)(n^+).
\]
\begin{lem}
\label{SMFunc-PLn-n-bikes}
The functor $F$ is an isomorphism of categories.
\end{lem}
\begin{proof}
We will define a functor $\inv{F}:K(C)(n^+) \to \StrSMHom{\PLbb(n)}{\underline{C}}$ which is the inverse of $F$. An object $\Phi \in K(C)(n^+)$ is an $n$th strict Segal bicycle. An
$n$th strict Segal bicycle uniquely determines a strict symmetric monoidal functor $\widetilde{\Phi}:\Lbb(\gn{n}) \to \underline{C}$, see \eqref{SSM-ext-str-bike}. Now we define the strict symmetric monoidal functor $\inv{F}(\Phi)$ to be the following composite:
\[
\PLbb(n) \cong \Lbb(\gn{n})^{\textit{proj}} \hookrightarrow \Lbb(\gn{n}) \overset{\widetilde{\Phi}} \to \underline{C}.
\]
\end{proof}
\begin{rem}
\label{Pi-1-ind-equiv}
In the statement of the above lemma the functor category $\StrSMHom{\PLbb(n)}{\underline{C}}$ could be replaced by the isomorphic category $\StrSMHom{\Pi_1(\PLbb(n))}{C}$ where $\Pi_1(\PLbb(n))$ is the groupoid obtained by inverting all morphisms in $\PLbb(n)$.
\end{rem}
\begin{nota}
\begin{sloppypar}
Let $f = (f_1, f_2, \dots, f_r)$ be a finite sequence, where $f_i:n^+ \to k^+_i$ is
a map of based sets for $1 \le i \le r$. We denote the finite sequence $(Supp(f_1), Supp(f_2), \dots, Supp(f_r))$ by $Supp(f)$.
\end{sloppypar}
\end{nota}
\begin{nota}
Let $f = (f_1, f_2, \dots, f_r)$ be a finite sequence, where $f_i:n^+ \to k^+_i$ is
a map of based sets for $1 \le i \le r$. We denote the sum $\underset{i=1}{\overset{r} \sqcup} f_i|_{Supp(f_i)}$ by $tot(f)$, where the map $f_i|_{Supp(f_i)}:Supp(f_i) \to k_i$ is the restriction of $f_i$.
\end{nota}

We now define another category $\Lbb(n)$ which is equipped with an inclusion functor
\begin{equation}
\label{inc-func-PL-L}
\iota:\PLbb(n) \hookrightarrow \Lbb(n).
\end{equation}

\begin{df}
\label{L-2-n}
An object in $\Lbb(n)$ is a finite sequence $f = (f_1, f_2, \dots, f_r)$, where $f_i:n^+ \to k^+_i$ is
a map of based sets for $1 \le i \le r$. To each finite sequence $f$ one can (uniquely) associate the following zig-zag
\begin{equation*}
\label{assoc-zig-zag-L}
\xymatrix@C=16mm{
\underset{i=1}{\overset{r} \sqcup} Supp(f_i) \ar[d]_{\sigma(f)} \ar[r]^{tot(f)} & \underset{i=1}{\overset{r} \sqcup \ k_i} \ar[r]^{Ind(f)} & \underline{r} \\
\underline{n}
}
\end{equation*}
where $\sigma(f) := \sigma(Supp(f))$.
A map from $f$ to $g = (g_1, g_2, \dots, g_s)$ in $\Lbb(n)$ is a triple is a triple $(h,q,p)$, where $h:\underline{s}  \to \underline{r}$ is a map in $\N$, $q$ is a map in $\N$ and $p$ is a bijection in $\N$ such that the following diagram commutes
\begin{equation*}
\xymatrix{
& \underline{n} \\
\underset{i=1}{\overset{r} \sqcup} Supp(f_i) \ar[rr]^p \ar[ru]^{\sigma(f)} \ar[d]_{tot(f)} && \underset{i=1}{\overset{s} \sqcup} Supp(g_i) \ar[lu]_{\sigma(g)} \ar[d]^{tot(g)} \\
\underset{i=1}{\overset{r} \sqcup} k_i \ar[rr]^{q} \ar[d]_{\textit{Ind}(f)} && \underset{i=1}{\overset{r} \sqcup} l_i  \ar[d]^{\textit{Ind}(g)} \\
\underline{r}  && \underline{s} \ar[ll]_h
}
\end{equation*}
\end{df}
Each subset $S \subseteq \underline{n}$ uniquely determines a projection map
$f_S:n^+ \to S^+$. The inclusion functor \eqref{inc-func-PL-L} is defined on objects as follows:
\[
\iota((S_1, S_2, \dots, S_r)) := f_S = (f_{S_1}, f_{S_2}, \dots, f_{S_r}).
\]
The functor is defined on morphisms as follows:
\[
\iota((h, p)) := (h, p, p),
\]
where $(h, p):S = (S_1, S_2, \dots, S_r) \to T = (T_1, T_2, \dots, T_s)$ is a map in $\PLbb(n)$
and $(h, p, p): f_S \to f_T = (f_{T_1}, f_{T_2}, \dots, f_{T_s})$
is a map in $\Lbb(n)$ which is described by the following diagram:
\begin{equation*}
\xymatrix{
& \underline{n} \\
\underset{i=1}{\overset{r} \sqcup} Supp(f_{S_i}) = \underset{i=1}{\overset{r} \sqcup} S_i \ar[rr]^p \ar[ru]^{\sigma(f_S)} \ar[d]_{id} && \underset{j=1}{\overset{s} \sqcup} T_j = \underset{j=1}{\overset{s} \sqcup} Supp(f_{T_j}) \ar[lu]_{\sigma(f_T)} \ar[d]^{id} \\
\underset{i=1}{\overset{r} \sqcup} S_i \ar[rr]^{p} \ar[d]_{\textit{Ind}(f)} && \underset{j=1}{\overset{s} \sqcup} T_j  \ar[d]^{\textit{Ind}(g)} \\
\underline{r}  && \underline{s} \ar[ll]_h
}
\end{equation*}
\begin{lem}
\label{isom-PStr-gn}
For each $n \in \Nat$ the permutative category $\Lbb(n)$ is isomorphic to the permutative category $\Lbb(\gn{n})$.
\end{lem}
\begin{proof}
	The objects of both caregories are the same.
	We will show that each morphism in $\Lbb(\gn{n})$ uniquely defines a morphism in $ \Lbb(n)$ and  
therefore there is an isomorphism of permutative categories $J(n):\Lbb(\gn{n}) \to \Lbb(n)$ which is identity on objects.

Let $\vec{f} = (f_1, \dots, f_r)$ and $\vec{g} = (g_1, \dots, g_s)$ be two objects in $\Lbb(\gn{n})$ where $f_i:n^+ \to k_i^+$ for $1 \le i \le r$ and $g_i:n^+ \to l_i^+$ for $1 \le i \le s$.
We recall that a map $F:\vec{f} \to \vec{g}$ in $\Lbb(\gn{n})$ is a map $F = (h, p):\vec{k} = (k_1, \dots, k_r) \to \vec{l} = (l_1, \dots, l_s)$ in the category $\Leins$ such that
\begin{equation}
\label{map-in-Lnn-gn}
\Leins(\gn{n})((h, p))((f_1, \dots, f_r)) = (g_1, \dots, g_s).
\end{equation}
We recall that the map $F = (h, p)$ in $\Leins$ is a commutative square
\[
\xymatrix{
\ud{k} \ar[r]^p \ar[d]_{\vec{k}} & \ud{l} \ar[d]^{\vec{l}} \\
\ud{r}  & \ud{s} \ar[l]^h
}
\]
This map $(h, p)$ uniquely determines a (finite) sequence $(p_1, \dots, p_r)$ of maps in $\N$ where $p_i:k_i \to \underset{h(j) = i} + l_j$ for $1 \le i \le r$. The condition \eqref{map-in-Lnn-gn} implies that the following diagram commutes for each $i \in \ud{r}$ and $q \in \ud{s}$ such that $h(q) = i$:
\begin{equation*}
\xymatrix{
n^+ \ar[r]^{f_i} \ar[r] \ar[rd]_{g_q} & k_i^+ \ar[r]^{p_i^+ \ \ \ \ } & (\underset{h(j) = i} + l_j)^+ \ar[ld]^{u_q} \\
& l_q^+
}
\end{equation*}
where $u_q:(\underset{h(j) = i} + l_j)^+ \to l_q$ is the projection map onto $l_q$ where $h(q) = i$. Since each $p_i$ is a map in $\N$ and the supports of $u_q$ are distinct non-intersecting sets for $1 \le q \le s$ therefore we have the following equality:
\[
Supp(f_i) = Supp(p_i^+ \circ f_i) = \underset{h(j) = i} \sqcup Supp(g_j).
\]
The $r$ inclusion maps $\underset{h(j) = i} \sqcup Supp(g_j) \subseteq \underset{i \in \ud{s}} \sqcup Supp(g_i)$ for $1 \le i \le r$ determine a canonical bijection
\[
\sigma_F:\underset{i \in \ud{r}} \sqcup \underset{h(j) = i} \sqcup Supp(g_j) \to \underset{i \in \ud{s}} \sqcup Supp(g_i)
\]
such that the following diagram commmutes:
\begin{equation*}
\xymatrix{
& \ud{n} \\
\underset{i \in \ud{r}} \sqcup Supp(f_i) = \underset{i \in \ud{r}} \sqcup \underset{j \in \inv{h}(\lbrace i \rbrace)} \sqcup Supp(g_j) \ar[rr]^{\ \ \ \ \ \ \ \ \ \ \sigma_F} \ar[ru] \ar[d]_{tot(f)} &&
 \underset{i \in \ud{s}} \sqcup Supp(g_i) \ar[lu] \ar[d]^{tot(g)} \\
 \underset{i \in \ud{r}} \sqcup k_i \ar[rr]^p \ar[d]_{Ind(f)} && \underset{i \in \ud{s}} \sqcup l_i \ar[d]^{Ind(g)} \\
 \ud{r}  && \ud{s} \ar[ll]^h
}
\end{equation*}
Thus we have uniquely associated to an arrow $F = (h, p)$ in $\Lbb(\gn{n})$, an arrow in the category $\Lbb(n)$. We define the morphism function of $J(n)$ as follows:
\[
J(n)(F = (h, p)) := (h, p, \sigma_F).
\]
As mentioned above the object function of $J(n)$ is the identity.
This defines a functor which is an isomorphism of categories.
\end{proof}
\begin{rem}
\label{nat-transf-ext}
By proposition \ref{Ext-Fun-DegWise-iso} there exists a unique functor $\Lbb(-):\gop \to \PCat$ and the family isomorphisms
$J(n)$ in the lemma above glue together to define a natural isomorphism $J:\Lbb(\gn{-}) \Rightarrow \Lbb(-)$. This implies that we have a composite natural transformation
\[
\PLbb(-) \overset{i} \Rightarrow \Lbb(\gn{-}) \overset{J} \Rightarrow \Lbb(-)
\]
where $i$ is the natural transformation obtained in remark \ref{inc-str-lax-Pn}.
Further the natural equivalence of remark \ref{inc-str-lax-PN-pi1} extends to a natural equivalence $\Pi \circ(J \circ i)$
\begin{equation*}
id_{\Pi_1} \circ (J \circ i) :\Pi_1\circ \PLbb(-) \Rightarrow \Pi_1\circ \Lbb(\gn{-}) \Rightarrow \Pi_1\circ \Lbb(-).
\end{equation*}
\end{rem}

We define another functor $G:\Lbb(n) \to \PLbb(n)$. This functor is defined on objects as follows:
\[
G((f_1, f_2, \dots, f_r)) := Supp(f) = (Supp(f_1), Supp(f_2), \dots, Supp(f_r)).
\]
This functor is defined on morphisms as follows:
\[
G((h, q, p)) := (h, p).
\]
\begin{thm}
\label{PL-cref-SC-L}
The category $\PLbb(n)$ is a coreflective subcategory of $\Lbb(n)$.
\end{thm}
\begin{proof}
We will show that the functor $G$ defined above is a right adjoint to the inclusion functor
$\iota$. It is easy to see that $id_{\PLbb(n)} = G\iota$. We define a natural transformation
$\epsilon:\iota G \Rightarrow id_{L''(n)}$. Let $f = (f_1, f_2, \dots, f_r)$ be an object in $\Lbb(n)$.
We define
\[
\epsilon(f) := (id, tot(f), id):(f_{Supp(f_1)}, f_{Supp(f_2)}, \dots, f_{Supp(f_r)}) = f_{Supp(f)} = \iota G(f) \to f
\]
The following commutative diagram verifies that the triple on the right is a map in $\Lbb(n)$:
\begin{equation*}
\xymatrix{
& \underline{n} \\
\underset{i=1}{\overset{r} \sqcup} Supp(f_{Supp(f_i)})  \ar@{=}[rr]^{id} \ar[ru]^{\sigma(f_{Supp(f)}) \ \ \ \ } \ar@{=}[d]_{id} &&  \underset{i=1}{\overset{r} \sqcup} Supp(f_i) \ar[lu]_{\sigma(f)} \ar[d]^{tot(f)} \\
\underset{i=1}{\overset{r} \sqcup} Supp(f_{Supp(f_i)})  \ar[rr]^{tot(f)} \ar[d]_{\textit{Ind}(f_S)} && \underset{i=1}{\overset{r} \sqcup} k_i  \ar[d]^{\textit{Ind}(f)} \\
\underline{r}  && \underline{r} \ar@{=}[ll]_{id}
}
\end{equation*}
The following chain of equalities verifies that $\epsilon$ is a natural transformation:
\begin{multline*}
(h, q, p) \circ \epsilon(f) = (h, q, p) \circ (id, tot(f), id) = (h, q \circ tot(f), id) = \\
 (h, tot(g) \circ p, p) = (id, tot(g), id) \circ (h, p, p)
= \epsilon(g)  \circ \iota G((h, q, p)).
\end{multline*}

\end{proof}

Now we want to define another category $QL''(n)$ which is isomorphic to the full subcategory of $\PLbb(n)$ whose objects are finite sequences of, not necessarily distinct, singleton subsets of $\underline{n}$.
We will denote an object of $QL''(n)$ by $s = (s_1, s_2, \dots, s_r)$.
Equivalently we may describe this object $S$ by a map
\[
s: \underline{r} \to \underline{n}
\]
A map $p:(s_1, s_2, \dots, s_r) \to (t_1, t_2, \dots, t_r) = t$ in $QL''(n)$
is a bijection $p:\underline{r} \to
\underline{r}$ such that the following diagram commutes:
\begin{equation*}
\xymatrix{
\underline{r} \ar[rr]^p \ar[rd]_s && \underline{r} \ar[ld]^t \\
& \underline{n}
}
\end{equation*}

We observe that the category $QL''(n)$ is in fact a groupoid.
We define a functor
$H:\PLbb(n) \to QL''(n)$. Let $S = (S_1, S_2, \dots, S_r)$ be an object of $\PLbb(n)$, we define
$H(S)$ to be the following composite where the first map is the canonical bijection
\begin{equation*}
H(S): \underset{i=1}{\overset{r} + } S_i \overset{\inv{can}} \to \SQCup{i}{r} S_i \overset{\sigma(S)} \to \underline{n},
\end{equation*}
where $+$ denotes the tensor product in $\N$.
Let $(h, p):S \to T = (T_1, T_2, \dots, T_s)$ be a map in $\PLbb(n)$.
We define the morphism function of the functor $H$ as follows:
\begin{equation*}
H((h, p)) := \N(p),
\end{equation*}
where $\N(p):\underset{j=1}{\overset{s} + }  T_j  \to \underset{i=1}{\overset{r} + S_i }  $ is the
bijection in $\N$ which makes the following diagram commutative
\begin{equation*}
\xymatrix{
\SQCup{i}{r} S_i \ar[d]_{can} \ar[r]^p & \SQCup{j}{s} T_j \ar[d]^{can} \\
\underset{i=1}{\overset{r} + }  S_i  & \underset{j=1}{\overset{s} + }  T_j \ar[l]^{\N(p)}
}
\end{equation*}

It is easy to check that $H:\PLbb(n) \to QL''(n)$ is a functor.
 The following commutative diagram indicates the naturality in our definition of the functor $H$:
 \begin{equation}
 \label{Nat-in-H}
 \xymatrix{
 && \underline{n}  \\ 
  &\SQCup{i}{r} S_i \ar[ld]^p_\cong \ar[dd]^{\textit{Ind}(S)}  \ar[ru]^{\sigma(S)} \ar@/^1pc/@{=}[rr] && \SQCup{i}{r} S_i  \ar[lu]_{\sigma(S)} \ar[dd]^{can} \ar[ld]^p_\cong \\
  \SQCup{j}{s} T_j \ar@/^1pc/@{=}[rr] \ar@/^3pc/[rruu]^{\sigma(T)} \ar[dd]_{\textit{Ind}(T)} && \SQCup{j}{s} T_j  \ar@/_1pc/[uu]_{\sigma(T)} \ar[dd]^{can} \\
  & \underline{r} && \underset{i=1}{\overset{r} + }  S_i \ar@/^1pc/[ll]^{\textit{Ind}(S) \circ \inv{can}}  \\   \underline{s} \ar[ru]_h && \underset{j=1}{\overset{s} + }  T_j \ar@/^1pc/[ll]^{\textit{Ind}(T) \circ \inv{can}} \ar[ru]_{\N(p)}
 }
 \end{equation}
The above diagram will be useful in proving that $H$ is a left-adjoint-inverse. Now we define another functor
$\iota:QL''(n) \to \PLbb(n)$. Let $s:\underline{r} \to \underline{n}$ be an object in $QL''(n)$. The canonical inclusion of $s$ in $\underline{n}$ can be factored as follows:
\begin{equation}
\label{decom-ob-QL}
\xymatrix{
& \SQCup{i}{r} s(i) \ar[rd]^{\sigma(s)} \ar[ld]_{\Ind(s)} \\
\underline{r}  \ar[rr]_s && \underline{n}
}
\end{equation}
where $\Ind(s)$ is the bijection $s(i) \mapsto i$ and $\sigma(s)$ is the canonical inclusion map.
The functor $\iota$ is defined on objects as follows:
\[
\iota(s) := (s(1), s(2), \dots, s(r)).
\]
Let $p:s \to t$ be a map in $QL''(n)$, the functor $\iota$
is defined on morphisms as follows:
\[
\iota(p) := (\inv{p}, p'),
\]
where $p'$ is the unique bijection which makes the following diagram commute:
\begin{equation*}
\xymatrix{
\SQCup{i}{r} s(i) \ar[d]_{\Ind(s)} \ar[r]^{p'} & \SQCup{i}{r} t(i) \ar[d]^{\Ind(t)} \\
\underline{r} \ar[d]_{s}  & \underline{r} \ar[l]^{\inv{p}} \ar[ld]^{t} \\
\underline{n}
}
\end{equation*}
By the above commutative diagram and factorization \eqref{decom-ob-QL} we get the following commutative diagram which shows that $\iota(p) = (\inv{p}, p')$ is indeed a morphism in $\PLbb(n)$:
\begin{equation*}
 \xymatrix{
 & \underline{n} \\
 \SQCup{i}{r} s(i) \ar[ru]^{\sigma(s)} \ar[d]^{\Ind(s)} \ar[rr]^{p'} && \SQCup{i}{r} t(i) \ar[d]_{\Ind(t)} \ar[lu]_{\sigma(t)} \\
\underline{r} \ar@/^4pc/[ruu]^s  && \underline{r} \ar[ll]^{\inv{p}} \ar@/_4pc/[luu]_t
 }
\end{equation*}

\begin{thm} 
\label{refl-sub-cat-QPL}
The category $QL''(n)$ is isomorphic to a reflective subcategory of $\PLbb(n)$.
\end{thm}
\begin{proof}
We will show that the functor $H$
defined above is a left-adjoint-inverse to $\iota$ which is also defined above. Clearly $H\iota = id_{QL''(n)}$. We now construct a natural transformation $\eta:id \rightarrow  \iota H$. Let $S = (S_1, S_2, \dots, S_r)$ be an object of $\PLbb(n)$.
We define
\begin{equation*}
\eta(S) := (\textit{Ind}(S) \circ \inv{can}, id).
\end{equation*}
The following commutative diagram verifies that the pair on the right is a map in $\PLbb(n)$:
\begin{equation*}
\xymatrix{
& \underline{n}  \\ 
\SQCup{i}{r} S_i \ar[d]_{\textit{Ind}(S)}  \ar[ru]^{\sigma(S)} \ar@{=}[rr] && \SQCup{i}{r} S_i \ar[d]^{can}  \ar[lu]_{\sigma(S)} \\
\underline{r} && \underset{i=1}{\overset{r} + }  S_i \ar[ll]^{\textit{Ind}(S) \circ \inv{can}}
}
\end{equation*}
We claim that $\eta$ as defined above is a natural transformation.
Let $(h, p):S \to T = (T_1, T_2, \dots, T_j)$ be a map in $\PLbb(n)$.
In order to prove our claim we would like to show that the following diagram commutes in $\PLbb(n)$:
\begin{equation*}
\xymatrix{
S \ar[r]^{\eta(S) \ \ } \ar[d]_{(h, p)} &  \iota H(S) \ar[d]^{\iota H((h, p))} \\
T \ar[r]_{\eta(T) \ \ }  &  \iota H(T)
}
\end{equation*}
 The following chain of equalities verifies that $\eta$ is a natural transformation:
\begin{multline*}
\eta(T) \circ (h, p) = (\textit{Ind}(T) \circ \inv{can}, id) \circ (h, p) = (h \circ (\textit{Ind}(T) \circ \inv{can}), p) = \\
( (\textit{Ind}(S) \circ \inv{can}) \circ H(p), p) = (H(p), p) \circ (\textit{Ind}(S) \circ \inv{can}, id)
= \iota H((h, p)) \circ \eta(S).
\end{multline*}
we refer the reader to the commutative diagram \eqref{Nat-in-H}
for an explaination of the middle equalites.
The composite natural transformation
\[
id_{\PLbb(n)} \circ \iota \Rightarrow \iota H \iota \Rightarrow \iota \circ id_{QL''(n)}
\]
is the identity, this follows from the observation that
\[
\eta(\iota(s)) = (Ind(\iota(s)) \circ \inv{can}, id) = (can \circ \inv{can}, id) = \iota((id, id)) = id_{\iota(s)}.
\]
Similarly we claim that the following composite natural transformation
\[
H \circ id_{\PLbb(n)} \Rightarrow H \iota H \Rightarrow id_{QL''(n)} \circ H 
\]
is the identity. Our claim follows from the observation that
\[
H(\eta(S)) = H((Ind(S) \circ \inv{can}, id)) = (id, id) = id_{S}.
\]
\end{proof}

The above discussion can be summarized by the following diagram in which both pairs of functors are adjunctions
\begin{equation}
\label{reflect-coreflect}
\xymatrix{
QL''(n) \ar@{^{(}->}[r]  & \PLbb(n) \ar@/_1pc/[l]_{H} \ar@/_/[r] & \Lbb(n) \ar@/_/[l]_G
}
\end{equation}
We observe that the groupoid $\PStr(n)$, see definition \ref{P-Str-n}, is just the Gabriel factorization of the functor $H$. Since the functor $H$ has a right adjoint, proposition \ref{Gab-Fact-adjoints} implies that the groupoid $\PStr(n)$ is isomorphic to $\Pi_1(\PLbb(n))$. Thus we have the following lemma:
 \begin{lem}
 \label{rep-as-SM-func}
 For a permutative category $C$,  the category $K(C)(n^+)$ is isomorphic to the category of strict symmetric monoidal functors $\StrSMHom{\PStr(n)}{C}$. 
 \end{lem}
 \begin{proof}
 The above discussion and lemma \ref{SMFunc-PLn-n-bikes} give us the following chain of isomorphisms:
 \begin{equation}
 \KSeg(C)(n^+) \cong \StrSMHom{\Pi_1(\PLbb(n))}{C} \cong \StrSMHom{\PStr(n)}{C}
 \end{equation}
 \end{proof}
 \begin{rem}
 \label{nat-isom-PStr-n}
 There is a functor $\PLbb(-):\gop \to \PCat^{op}$, see \ref{P-Lax-n}, which gives us a composite functor $\Pi_1 \circ \PLbb(-)$. For each $n \in \Nat$ we have an isomorphism of categories $I(n):\Pi_1(\PLbb(n)) \cong \PStr(n)$. Now proposition \ref{Ext-Fun-DegWise-iso} implies that we have a functor
 $\PStr(-):\gop \to \PCat^{op}$ and a natural isomorphism $I:\Pi_1 \circ \PLbb(-) \Rightarrow \PStr(-)$.
 \end{rem}
 \begin{rem}
 There is bifunctor defined by the following composite:
 \begin{equation*}
 \StrSMHom{\PStr(-)}{-}: \gop \times \PCat \overset{\PStr(-) \times id} \to \PCat^{op} \times \PCat \overset{\StrSMHom{-}{-}} \to \Cat
 \end{equation*}
 where $\PStr(-)$ is the functor defined above and $\StrSMHom{-}{-}$ is the function object defined in appendix \ref{Mdl-CAT-Perm}.
 \end{rem}
 \begin{rem}
 The above lemma \ref{rep-as-SM-func} and the above remark together imply that for each pair $(n^+, C) \in \gop \times \PCat$ there is an isomorphism of categories $\eta(n):\StrSMHom{\PStr(n)}{C} \cong \K(C)(n^+)$.
 Now proposition \ref{Ext-Fun-DegWise-iso} implies that there is a bifunctor
 \begin{equation*}
 \K(-, -):\gop \times \PCat \to \Cat
 \end{equation*}
 defined by $\K(n^+, C) = \K(C)(n^+)$ which is equipped with a natural isomorphism
 $\eta:\StrSMHom{\PStr(-)}{-} \cong \K(-,-)$.
 This also implies that there is a functor
 \begin{equation*}
 \K:\PCat \to \gCAT
 \end{equation*}
 defined by $\K(C) := \K(-, C)$ for each permutative category $C$.
 \end{rem}
 \begin{rem}
 The above lemma \ref{rep-as-SM-func} implies that for each permutative category $C$, there is a $\gCat$
 \[
  \StrSMHom{\PStr(-)}{C}:\gop \to \Cat
 \]
   and it is isomorphic to $\K(C)$.
 \end{rem}
 \begin{rem}
 \label{nat-transf-ext-left}
 The natural equivalence from remark \ref{nat-transf-ext} extends to the following composite natural equivalence:
 \begin{equation*}
id_{\Pi_1} \circ (J \circ i) \circ \inv{I} :\PStr(n) \overset{\inv{I}} \Rightarrow \Pi_1\circ \PLbb(-) \Rightarrow \Pi_1\circ \Lbb(\gn{-}) \Rightarrow \Pi_1\circ \Lbb(-).
\end{equation*}
where $I$ is the natural isomorphism from remark \ref{nat-isom-PStr-n}
 \end{rem}
 
 The above lemma \ref{rep-as-SM-func} implies that the functor $\KSeg$ preserves limits in $\PCat$ because degreewise it is isomorphic to a functor which preserves limits. The category $\gCAT$ is complete and cocomplete. Now the \emph{formal criterion for existence of an adjoint} \cite[Thm. 2, Ch. X.7]{MacL} implies that $\KSeg$ has a left adjoint which we denote
 \begin{equation*}
 \label{left-adj-Seg-func}
 \PStr:\gCAT \to \PCat.
 \end{equation*}
 Each $n \in Ob(\N)$ uniquely defines $n$ \emph{projection} maps of based sets $\delta^n_k:n^+ \to 1^+$, $1 \le k \le n$. Each of these projection maps induce a strict symmetric monoidal functor $\PStr(\delta^n_k):\PStr(1) \to \PStr(n)$
 which maps the object $1 \in Ob(\PStr(1))$ to the inverse image of $1$ under the map $\delta^n_k$ \emph{i.e.} $\PStr(\delta^n_k)(1) = \inv{(\delta^n_k)}(\lbrace 1 \rbrace) = \lbrace k \rbrace \subset n$ in the category $\PStr(n)$. These inclusion maps together induce a strict symmetric monoidal functor in $\PCat$
 \begin{equation*}
 \underset{k=1} {\overset{n} \vee} \PStr(\delta_n^k): \underset{k=1} {\overset{n} \vee} \PStr(1) \to \PStr(n),
 \end{equation*}
 where $\underset{k=1} {\overset{n} \vee} \PStr(1)$ is the coproduct of $n$ copies
 of $\PStr(1)$ in $\PCat$. We will now present a concrete construction of
 the coproduct $\underset{k=1} {\overset{n} \vee} \PStr(1)$ and also construct
 a strict symmetric monoidal functor $\underset{k=1} {\overset{n} \vee} \PStr(\iota_n^k)$.
 An (non-unit) object $S$ of $\underset{k=1} {\overset{n} \vee} \PStr(1)$ is a (finite) sequence $(s_1, s_2, \dots, s_r)$ in which $s_i$ is either the empty set or a singleton subset of $\ud{n}$ for $1 \le i \le r$. We observe that the object $S$ is equipped with a (unique) morphism
 $\underset{i=1}{\overset{r} \sqcup} s_i \to \ud{n}$.  A morphism $f: S \to T = (t_1, t_2, \dots, t_q)$ is an isomorphism $f:\underset{i=1}{\overset{r} \sqcup} s_i  \to \underset{i=1}{\overset{q} \sqcup} t_i $ such that the following diagram commutes:
 \begin{equation*}
 \xymatrix{
 \underset{i=1}{\overset{r} \sqcup} s_i \ar[rr]^f_\cong \ar[rd] && \underset{i=1}{\overset{q} \sqcup} t_i \ar[ld] \\
 & \ud{n}
}
 \end{equation*}
 
 The  functor $ \underset{k=1} {\overset{n} \vee}
 \PStr(\delta^n_k)$ is now the obvious inclusion functor.
 
 \begin{lem}
 \label{acy-cof-inc}
 The strict symmetric monoidal functor $\underset{k=1} {\overset{n} \vee} \PStr(\delta^n_k)$
 is an acyclic cofibration in $\PCat$.
 \end{lem}
\begin{proof}
	The functor $\underset{k=1} {\overset{n} \vee} \PStr(\delta^n_k)$ is a cofibration because its object function is a monomorphism of free monoids.
	The inclusion functor $\underset{k=1} {\overset{n} \vee}
	\PStr(\delta^n_k)$ is fully-faithful. Each object of $\PStr(n)$ is isomorphic to an object of $\underset{k=1} {\overset{n} \vee} \PStr(1)$.
	\end{proof}
 \begin{coro}
 \label{Seg-func-Ein}
 For each permutative category $C$, $\KSeg(C)$ is a coherently commutative monoidal category.
 \end{coro}
 \begin{proof}
 By Lemma \ref{rep-as-SM-func}, $\KSeg(C)(n^+) \cong \StrSMHom{\PStr(n)}{C}$. Now we have the following commutative diagram in $\Cat$
 \begin{equation*}
\xymatrix{
 \KSeg(C)(n^+) \ar[r]^{} \ar[d]_{(\KSeg(C)(\delta^n_1), \dots, \KSeg(C)(\delta^n_n) )} & \StrSMHom{\PStr(n)}{C}  \ar[d]^{(\StrSMHom{\PStr(\delta^n_1)}{C}, \dots, \StrSMHom{\PStr(\delta^n_n)}{C})} 
 \\
 \underset{i=1}{\overset{n} \prod} \KSeg(C)(1^+) \ar[r] 
 & \underset{i=1}{\overset{n} \prod} \StrSMHom{\PStr(1)}{C}
}
\end{equation*}
According to the lemma \ref{acy-cof-inc}, the strict symmetric monoidal functor $\underset{k=1} {\overset{n} \vee} \PStr(\delta^n_k)$
is an acyclic cofibration therefore
the right vertical functor is an acyclic fibration in $\Cat$, see corollary \ref{Hom-Cof-Fib}. The two horizontal functors in this diagram are isomorphisms, therefore $(\KSeg(C)(\delta^n_1), \dots, \KSeg(C)(\delta^n_n) )$ is also an acyclic fibration in $\Cat$. Thus we have proved that $\KSeg(C)$ is a coherently commutative monoidal category for every $C \in Ob(\PCat)$.
 \end{proof}
 The above corollary will be extremely useful in proving that the adjunction $\left(\PStr, \KSeg \right)$ is a Quillen adjunction. We recall that a map in $\gCAT$ between two coherently commutative monoidal categories is a weak equivalence (resp. fibration) if and only if it is degreewise a weak equivalence (resp. fibration) in $\Cat$.
 \begin{lem}
 The adjunction $\left(\PStr, \KSeg \right)$ is a Quillen adjunction between the natural model category $\PCat$ and the model category of coherently commutative monoidal categories  $\gCAT$.
 \end{lem}
 \begin{proof}
 We will prove the lemma by showing that the right adjoint functor $\KSeg$ preserves fibrations and acyclic fibrations. Let $F:C \to D$ be a fibration in $\PCat$. In order to show that $\KSeg(F)$ is a fibration in the model category of coherently commutative monoidal categories $\gCAT$, it would be sufficient to show that $\KSeg(F)(n^+)$ is a fibration in $\Cat$, for all $n^+ \in Ob(\gop)$. For each $n \in \Nat$ the groupoid $\PStr(n)$ is a cofibrant object in $\PCat$. The natural model categort $\PCat$ is a $\Cat$-model category whose cotensor is given by the functor $\StrSMHom{-}{-}$. This implies that the functor
 \[
 \StrSMHom{\PStr(n)}{F}:\StrSMHom{\PStr(n)}{C} \to \StrSMHom{\PStr(n)}{D}
 \]
 is a fibration in $\Cat$ and it is an acyclic fibration in $\Cat$ whenever $F$ is an acyclic fibration.
 \end{proof}

 \section[The Thickened Nerve]{The Thickened Nerve}
\label{thick-Seg-nerve}
In this section we will describe a \emph{thickened} version of Segal's nerve functor which we will denote by $\Kbar$ and show that $\Kbar$ is the right Quillen functor of a Quillen equivalence. Unlike the left Quillen functor of the Quillen adjunction $(\PStr, \K)$ described in the previous section, whose mere existence was shown, we will explicitly describe a functor $\PNat:\gCAT \to \PCat$ and show that it is the left Quillen adjoint of $\Kbar$. The explicit description will play a vital role in proving that the Quillen pair $(\PNat, \Kbar)$
is a Quillen equivalence. In this section we will also present the main result of this paper which proves that the Quillen pair of functors $(\PStr, \K)$ is a Quillen equivalence. The Quillen equivalence $(\PNat, \Kbar)$ will be used to prove the main result.

We begin by defining a functor $\PNat:\gCAT \to \PCat$ as follows:
\begin{equation*}
\gCAT \overset{\Leins(-)} \to \SMHom{\Leins}{\Cat}{} \overset{L_H\int^{\vec{n} \in \Leins} - } \to \PCat 
\end{equation*}
where $\Leins(-)$ is the symmetric monoidal extention functor described in section \ref{real-funct}.  The second functor $L_H\int^{\vec{n} \in \Leins} - $ first performs the Grothendieck construction on a functor $F \in \SMHom{\Leins}{\Cat}{}$ to obtain a permutative category $\int^{\vec{n} \in \Leins} F$, see theorem \ref{perm-cat-elm}. and then it localizes (or formally inverts) the \emph{horizontal arrows} of the permutative category
$\int^{\vec{n} \in \Leins} F$. We recall that an arrow in the category $\int^{\vec{n} \in \Leins} F$ is a pair $(f, \phi)$ where $f$ is a map in $\Leins$ and $\phi$ is an arrow in the category $F(codom(f))$. An arrow $(f, \phi)$ is called horizontal if $\phi$ is the identity morphism. Thus for a $\gCat$ $X$,
$\PNat(X) = L_H\int^{\vec{n} \in \Leins} \Leins(X) $ is the permutative category obtained by \emph{localizing} with respect to the set of all
\emph{horizontal} morphisms
in the (permutative) category $\int^{\vec{n} \in \Leins}\Leins(X)$, see \cite[Ch. 1]{GZ} for a procedure of localization.
The results of \cite{Day-local} imply that the category $\int^{\vec{n} \in \Leins}\Leins(X)$ has the universal property that
any strict symmetric monoidal functor $F: \int^{\vec{n} \in \Leins}\Leins(X) \to C$ which maps every horizontal morphism in $\int^{\vec{n} \in \Leins}\Leins(X)$ to an isomorphism in $C$ extends uniquely to a strict symmetric monoidal functor $F_{Nat}:\PNat(X) \to C$ along the projection map $p: \int^{\vec{n} \in \Leins}\Leins(X) \to \PNat X$, \emph{i.e.} the functor $F_{Nat}$ makes the following diagram commute
\begin{equation}
\label{univ-prop}
\xymatrix{
 \int^{\vec{n} \in \Leins}\Leins(X) \ar[r]^F \ar[d]_{p} & C \\
 \PNat X \ar@{-->}[ru]_{F_{Nat}}
}
\end{equation}
 The localization construction is functorial in $X$ and therefore we get a functor $\PNat(-): \gCAT \to \PCat$. 

Now we define the thickened nerve functor $\Kbar$. We will first define this functor in the spirit of the papers \cite{May4}, \cite{SK}, \cite{mandell} and \cite{mandell2} and later we will provide a couple of new interpretation of this functor based on pseudo bicycles, see appendix \ref{NotionBike} and strict symmetric monoidal functors.
\begin{df}
An \emph{$n$th pseudo Segal bicycle} in a symmetric monoidal category $C$ is a quadruple $(\Phi, \alpha_\Phi, \sigma_\phi, u_\Phi)$ of families of objects or morphisms of the symmetric monoidal category $C$, where
\begin{enumerate}
\item $\Phi = \lbrace c_f \rbrace _{f \in A_n}$ is a family of objects of $C$, where the indexing set
\[
 A_n := \lbrace f \in \gop : domain(f) = n^+ \rbrace.
 \]
\item $\alpha_\Phi = \lbrace \alpha(h, f): c_f \to c_{h^+ \circ f} \rbrace_{(h, f) \in D}$ is a family of isomorphisms in $C$, where the indexing set 
\[
D := \lbrace (h, f) \in Mor(\N) \times A_n : dom(h)^+ = codom(f) \rbrace
\]
\item $\sigma_\Phi = \lbrace \sigma(k,l, f): c_f \to c_{f_k} \otimes c_{f_l} \rbrace_{(k,l,f) \in B}$ is a family of isomorphisms in $C$, where $f_k = \delta^{k+l}_k \circ f$ and $f_k = \delta^{k+l}_l \circ f$ and the indexing set

\[
B := \lbrace (k, l, f) \in \Nat \times \Nat \times A_n : codom(f) = (k + l)^+ \rbrace.
\]

\item $u_\phi = \lbrace u(f):c_f \to \unit{C} \rbrace_{f \in A_n(0)}$
is a family of isomorphisms in $C$, where the indexing set is the following subset of $A_n$
\[
A_n(0) := \lbrace f \in A_n : codom(f) = 0^+ \rbrace.
\]
\end{enumerate}
The quadruple $(\Phi, \alpha_\Phi, \sigma_\phi, u_\Phi)$ is subject to the following conditions:
\begin{enumerate}[label = {PSB.\arabic*}, ref={PSB.\arabic*}]
\item  
For any (pointed) function $f:n^+ \to m^+$ in the indexing set $A_n$, the map
 \[
 c_f \overset{\sigma(m, 0, f)} \to c_{f_m} \underset{C}\otimes c_{f_0} \overset{id \underset{C} \otimes u(f_0)}\to c_{f_m} \underset{C}\otimes \unit{C}
 \]
 is the inverse of the (right) unit isomorphism in $C$. Similarly the map
 \[
 c_f \overset{\sigma(0,m, f)} \to c_{f_0} \underset{C}\otimes c_{f_m} \overset{u(f_0) \underset{C} \otimes id}\to  \unit{C} \underset{C} \otimes  c_{f_m} 
 \]
 is the inverse of the (left) unit isomorphism in $C$. 
 
 \item  For each triple $(k,l, f) \in B$, where the indexing set $B$ is defined above,
 the following diagram commutes in the category $C$
 \begin{equation*}
 \xymatrix@C=12mm{
 c_f \ar[r]^{\alpha(\gamma^\N_{k, l})_f} \ar[d]_{\sigma(k,l, f)} &c_f \ar[d]^{\sigma(l,k, f)} \\
 c_{f_k} \underset{C} \otimes c_{f_l} \ar[r]_{\gamma^C_{c_{f_k}, c_{f_l}}} &c_{f_l} \underset{C} \otimes c_{f_k}
 }
 \end{equation*}
\item
For any triple $k,l,m \in \Nat$, and each $f:n ^+ \to (k+l+m)^+$ in the set $A_n$, the following diagram
commutes
\begin{equation*}
  \xymatrix@C=18mm{
 c_f \ar[r]^{\sigma(k+l,m, f)}  \ar[d]_{\sigma(k,l+m, f)}  &c_{k+l} \underset{C} \otimes c _m
 \ar[dd]^{\sigma(k,l, f) \underset{C} \otimes id_{c_{f_m}}}\\
 c_{f_k} \underset{C} \otimes c_{f_{l+m}} \ar[d]_{id_{c_{f_k}} \underset{C} \otimes \sigma(l,m, f)} \\
 c_{f_k} \underset{C} \otimes (c_{f_l} \underset{C} \otimes c_{f_m}) \ar@{<-}[r]_{\alpha^C_{c_{f_k},c_{f_l},c_{f_m}}} & (c_k \underset{C} \otimes c_l) \underset{C} \otimes c_m
 }
 \end{equation*}
 \item For each triple $(k, l , h) \in B$, where the indexing set $B$ is defined above,
 and each pair of active maps $f:k^+ \to p^+$, $g:l^+ \to q^+$ in $\gop$,
 the following diagram commutes in the category $C$
 \begin{equation*}
  \xymatrix@C=18mm{
 c_h \ar[r]^{\sigma(k,l, h)} \ar[d]_{\alpha(f+g, h)} &c_{h_k} \underset{C} \otimes c_{h_l}
 \ar[d]^{\alpha(f+g, h_k) \underset{C} \otimes \alpha(f+g, h_l)} \\
 c_{(f + g) \circ h} \ar[r]_{\sigma(k, l, (f + g) \circ h) \ \ \ \ \ \ \ \ \ \   } & c_{((f + g) \circ h)_k} \underset{C} \otimes c_{((f + g) \circ h)_l}
  }
 \end{equation*}
\end{enumerate}

\end{df}
Next we define the notion of a morphism of pseudo Segal bicycles:

\begin{df}

A morphism of nth unnormalized pseudo Segal bicycles
\[
 F:(\Phi, \alpha_\Phi, \sigma_\phi, u_\Phi) \to (\Psi, \alpha_\Psi, \sigma_\psi, u_\Psi)
 \]
  is a family $F = \lbrace F(f):c^\Phi_f \to c^\Psi_f \rbrace_{f \in A_n}$ of morphisms in $C$ which satisfies the following conditions:

 \begin{enumerate}
 
 \item For each $f \in A_n(0)$ $(codom(f) = 0^+)$, the following diagram commutes:
 \begin{equation*}
 \xymatrix{
c_f^\Phi \ar[rr]^{F(f)} \ar[rd]_{u_\Phi(f)} && c_f^\Psi \ar[ld]^{u_\Psi(f)} \\
& \unit{C}
 }
 \end{equation*} 
 
 \item For each pair $(f, h) \in A_n \times Mor(\N)$ such that the domain of $h^+$, namely $dom(h)^+$, is the same as the codomain of $f$ , the following diagram commutes
 \begin{equation*}
  \xymatrix@C=24mm{
 c_f^\Phi \ar[r]^{F(f)} \ar[d]_{\alpha_\Phi(h, f)} &c_f^\Psi
 \ar[d]^{\alpha_\Psi(h, f)} \\
 c_{h \circ f}^\Phi \ar[r]_{F( h \circ f)} & c_{h \circ f}^\Psi
  }
 \end{equation*}
 
  \item For each triple $(k, l, f) \in B$, where the index set $B$ is defined above, the following diagram commutes
 \begin{equation*}
  \xymatrix@C=24mm{
 c_f^\Phi \ar[r]^{F( f)} \ar[d]_{\sigma^\Phi(k, l, f)} &c_f^\Psi
 \ar[d]^{\sigma^\Psi(k, l, f)} \\
 c_{f_k}^\Phi \underset{C} \otimes c_{ f_l}^\Phi  \ar[r]_{F( f_k) \underset{C} \otimes F(f_l)} &  c_{f_k}^\Psi \underset{C} \otimes c_{ f_l}^\Psi
  }
 \end{equation*}
 \end{enumerate}

\end{df}

All $nth$ unnormalized pseudo Segal bicycles in a symmetric monoidal category $C$ and all morphisms of $nth$ unnormalized pseudo Segal bicycles in $C$ form a category which we denote by $\Kbar(C)(n^+)$.
\begin{lem}
\label{Bikes-nSegPsBikes-eq}
 Let $C$ be a permutative category. For each $\underline{n}$, the category $\Kbar(C)(n^+)$ is isomorphic to the category of all pseudo bicycles from $\gn{n}$ to $C$ namely $\Bikes{\gn{n}}{C}$.
 \end{lem}
The proof of this lemma is just the adaptation of the argument of the proof lemma \ref{strict-bikes-SB-eq}, which deals with the case of strict bicycles, to the present scenario of pseudo bicycles.

\begin{df}
For each $n \in Ob(\N)$ we will now define a permutative groupoid $\PNat(n)$. The objects of this
 groupoid are finite collections of morphisms in $\gop$ having domain $n^+$, in other words the object monoid of the category $\PNat(n)$ is the free monoid generated by the following set
 \[
  Ob(\PNat(n)) := \underset{k \in Ob(\N)} \sqcup \ \gn{n}(k^+).
 \]
 \begin{sloppypar}
We will denote an object of this groupoid by $(f_1, f_2, \dots, f_r)$. A morphism $(f_1, f_2, \dots, f_r) \to (g_1, g_2, \dots, g_k)$ is an
 isomorphism of finite sets
 \[
  F:\Supp{f_1} \sqcup \Supp{f_2} \sqcup \dots \sqcup \Supp{f_r} \overset{\cong} \to \Supp{g_1} \sqcup \Supp{g_2} \sqcup \dots \sqcup \Supp{g_k}
 \]
 such that the following diagram commutes
 \end{sloppypar}
 \begin{equation*}
\xymatrix{
 \Supp{f_1} \sqcup \dots \sqcup \Supp{f_r} \ar[rr]^{F} \ar[rd] && \Supp{g_1} \sqcup \dots \sqcup \Supp{g_k} \ar[ld] \\
 &n
}
\end{equation*}
 where the diagonal maps are the unique inclusions of the coproducts into $n$.
 \end{df}
\begin{rem}
\label{cont-fun-Lbar}
The construction above defines a contravariant functor $\PNat(-):\gop \to \PCat$. A map $f:n^+ \to m^+$ in $\gop$ defines a strict symmetric monoidal functor
$\PNat(f):\PNat(m) \to \PNat(n)$. An object $(f_1, f_2, \dots, f_r) \in \PNat(m)$ is mapped by this functor to $(f_1 \circ f, f_2 \circ f, \dots, f_r \circ f) \in \PNat(n)$.
\end{rem}
\begin{rem}
\label{isom-PNat-n-piPgn}
We observe that the category $\PNat(n)$ defined above is a Gabriel factorization of the composite functor $G \circ H$, see equation \eqref{reflect-coreflect}, and therefore by proposition \ref{Gab-Fact-adjoints} it is isomorphic to $\Pi_1 \Lbb(n)$, for each $n \in \Nat$. Further by proposition \ref{Ext-Fun-DegWise-iso} these isomorphisms glue together to define a natural isomorphism
$T:\Pi_1\circ \Lbb(-) \Rightarrow \PNat(-)$.
\end{rem}
\begin{rem}
 \label{nat-transf-ext-right}
 The natural equivalence from remark \ref{isom-PNat-n-piPgn} extends to the following composite natural equivalence:
 \begin{equation*}
T \circ (id_{\Pi_1} \circ (J \circ i) \circ \inv{I}) :\PStr(n) \overset{\inv{I}} \Rightarrow \Pi_1\circ \PLbb(-) \Rightarrow \Pi_1\circ \Lbb(\gn{-}) \Rightarrow \Pi_1\circ \Lbb(-) \overset{T} \Rightarrow \PNat(-).
\end{equation*}
where $T$ is the natural isomorphism from remark \ref{isom-PNat-n-piPgn}.
 \end{rem}

 \begin{prop}
 \label{equiv-PsBikes-StrSMFunc-LBar}
 For each $n \in Ob(\N)$, the permutative category $\PNat(n)$ represents the functor $\PsBikes{\gn{n}}{-}$. In other words there is a natural isomorphism
 \[
 \psi^n:\StrSMHom{\PNat(n)}{-} \cong \PsBikes{\gn{n}}{-}.
 \]
 
 \end{prop}
%
%
%
 
 It was proved in \cite[sec. 6]{SGA4} that, for any $\gCat$ $X$, the permutative
 category $\PNat(X)$ is a \emph{pseudo-colimt} of the functor $\Leins(X)$. In other words $\PNat(X)$ represents the category of \emph{pseudo-cones} \emph{i.e.}
 for any category $C$
 \[
 Ps[\Leins{X}, \Delta{C}] \cong [\PNat(X), C].
 \]
 This characterization provides the functor $\PNat$ with some very desirable homotopical properties.
 \begin{lem}
 \label{PNat-hocolim}
 The functor $\PNat$ preserves degreewise equivalences of $\Gamma$-categories.
 \end{lem}
 \begin{proof}
 The functor $\PNat$ is a composite of the functor $\Leins$
 followed by a pseudo-colimit functor.
 The functor $\Leins:\gCAT \to \StrSMHom{\Leins}{\Cat}$ preserves
 degreewise equivalences.
 The results of \cite{NG} show that a pseudo colimit functor is a
 homotopy colimt functor and it preserves degreewise equivalences.
 Hence the functor $\PNat$ preserves degreewise equivalences of $\gCats$.

 \end{proof}
 Before moving on we would like to observe that for any object $\vec{n} \in Ob(\L)$
  there exists the following zig-zag of maps in $\L$
  \begin{equation}
  \label{mult-partition}
  (1) \overset{(id, m_n)} \leftarrow (n) \overset{(m_r, id)} \to \vec{n} = (n_1, n_2, \dots, n_r)
  \end{equation}
  where $n = n_1 + n_2 + \dots + n_r$ and $m_n:n^+ \to 1^+$ is the unique
 \emph{multiplication} map from $n^+$ to $1^+$ in $\gop$. To be more precise, the left map is given by the following commutative diagram
 \begin{equation}
 \label{multiplication-map}
 \xymatrix{
 \underline{n} \ar[r]^{m_n} \ar[d] & \underline{1} \ar[d] \\
 \underline{1} \ar@{=}[r] & \underline{1}
 }
 \end{equation}
 and the right map is given by the following commutative diagram
 \begin{equation}
 \label{part-map}
 \xymatrix{
 \underline{n} \ar@{=}[r] \ar[d] & \underline{n} \ar[d]^{Ind(\vec{n})} \\
 \underline{1} & \underline{r} \ar[l]^{m_r}
 }
 \end{equation}
The following corollary provides a useful insight into the structure of the
localization of the category of elements of a coherently commutative monoidal category $X$, with respect to horizontal maps. It turns out that
this localized category is a thickening of $X(1^+)$. This thickening is indicative of the fact that the homotopy colimit of a diagonal functor $\Delta(c)$ is equivalent to $c$. The category $\PNat(X)$
is a further thickening of this localized category.
\begin{coro}
\label{Thicken-CCM-Cat}
For each coherently commutative monoidal category $X$ the inclusion functor $i:X(1^+) \to \PNat(X)$ is an equivalence of categories.
\end{coro}
\begin{proof}
The functor $i:X(1^+) \to \PNat(X)$ is an inclusion functor, it is defined on objects as follows:
\[
i(x) := (id_{\underline{1}}, x)
\]
and for a morphism $f:x \to y$ in $X(1^+)$ it is defined as follows:
\[
i(f) := ((id_{\underline{1}}, id_{\underline{1}}), f).
\]
Clearly the functor $i$ is fully faithful. Now we will show that $i$ is also
essentially surjective. In order to do se we will use the maps \eqref{multiplication-map} and \eqref{part-map} defined above.
For each object $(\vec{n}, \vec{x}) \in \PNat(X)$, the map \eqref{part-map} provides a functor
\[
\Leins(X)((m_r, id_{\underline{n}})):X(m^+) \to \underset{i=1}{\overset{r} \prod} X(m_i).
\]
Since $X$ is a coherently commutative monoidal category therefore the above functor is an equivalence of categories. Thus we may choose an object $x \in X(m^+)$ and an isomorphism $j:\Leins(X)((m_r, id_{\underline{n}}))(x) \to \vec{x}$ in $\underset{i=1}{\overset{r} \prod} X(m_i)$. We observe that the map
\[
((m_r, id_{\underline{n}}) ,j):((n), x) \to (\vec{n}, \vec{x})
\]
 is an isomorphism in $\PNat(X)$ because $\PNat(X)$ is obtained by inverting all horizontal maps in the category of elements of $\Leins(X)$.
 The map \eqref{multiplication-map} provides us with the following isomorphisms:
\begin{equation*}
((id_{\underline{1}}, m_n), id_{\Leins(X)((id_{\underline{1}}, m_n))((x))}):((n), (x)) \to 
 ((1), \Leins(X)((id_{\underline{1}}, m_n))((x)))
\end{equation*}
\begin{sloppypar}
The above two isomorphisms show that each object $(\vec{n}, \vec{x}) \in \PNat(X)$ is isomorphic to an object in the image of the functor $i$ namely
$((1), \Leins(X)((id_{\underline{1}}, m_n))((x)))$. The isomorphism is given by the composite
\[
 ((m_r, id_{\underline{n}}) ,j)  \circ \inv{((id_{\underline{1}}, m_n), id_{\Leins(X)((id_{\underline{1}}, m_n))((x))})} .
 \]
  Thus we have proved that $i$ is essentially surjective and therefore an equivalence.
\end{sloppypar}
\end{proof}
\begin{rem}
\label{str-inv-inc}
There exists an inverse functor $\inv{i}:\PNat(X) \to X(1^+)$ such that $\inv{i} \circ i = id_{X(1^+)}$.
\end{rem}
 
 Momentarily we will switch to the language of bicycles for the purpose of
 proving lemma \ref{unit-map-equiv}. Each object of a permutative category $C$ defines a \emph{trivial} bicycle from $\gn{1}$ to $C$ which
we denote by $\Phi_c = (\L_c, \sigma_c)$. We define this bicycle 
$\Bike{\Phi_c}{\gn{1}}{C}$ next. We begin by defining the underlying lax cone $\L_c = (\phi_c, \alpha_c)$. For each $k \in Ob(\N)$, we define the
functor $\phi_c(k):\gnk{1}{k} \to C$ as follows:
\begin{equation}
\phi_c(k)(f) = 
\begin{cases}
c, & \text{if} \ f \ne 0 \\
\unit{C}, & \text{otherwise}.
\end{cases}
\end{equation}
For a map $h:k \to l$ in the category $\N$, we define the map $\alpha_c(h)(f):\phi_c(f) \to \phi_c(h \circ f)$
as follows:
\begin{equation}
\alpha_c(h)(f) = 
\begin{cases}
id_c, & \text{if} \ f \ne 0 \\
id_{\unit{C}}, & \text{otherwise}.
\end{cases}
\end{equation}
It is easy to see that with the above definition, $\L_c = (\phi_c, \alpha_c)$ is
a lax cone. This lax cone is given a bicycle structure by
defining $\sigma_c(k,l): \phi_c(k+l) \Rightarrow \phi_c(k) \odot \phi_c(l)$
to be the identity natural transformation. Thus we have defined a
strict bicycle $(\L_c, \sigma_c)=\Bike{\Phi_c}{\gn{1}}{C} $. This construction
defines a functor
\begin{equation}
\label{trivial-bike-functor}
\Phi_{-}:C \to \PsBikes{\gn{1}}{C}
\end{equation}
\begin{lem}
 \label{isom-to-trivial-bike}
 Every bicycle $\Bike{(\L, \sigma)=\Phi}{\gn{1}}{C}$ is isomorphic to the trivial bicycle determined by the object $\phi(1)(id_{1^+}) \in Ob(C)$,
 namely $\Phi_{\phi(1)(id_{1^+})}$, where $\L = (\phi, \alpha)$ is the
 underlying lax symmetric monoidal cone of $\Phi$. 
\end{lem}
\begin{proof}
 We will construct an isomorphism of bicycles
 \[
  \eta(\Phi):\Phi \to \Phi_{\phi(1)(id_{1^+})}.
 \]
 In order to do so, we will use the natural isomorphism
 \[
 \alpha(m_k):\phi(k^+) \Rightarrow \phi(1^+) \circ \gn{1}(m_k)
 \]
 provided by the bicycle $\Phi$, where $m_k:k^+ \to 1^+$ is the
 \emph{multiplication} map. For each $k \in Ob(\N)$ we define
 a natural isomorphism $\eta(\Phi)(k)$ as follows:
 \[
 \eta(\Phi)(k)(f) := \alpha(m_k)(f):\phi(k)(f) \to \phi(1)(id_{1^+}),
 \]
 where $f \in \gnk{1}{k}$. 
 One can check that the
 natural isomorphisms in the  collection $\lbrace \eta(\Phi)(k) \rbrace_{k \in Ob(\N)}$
 glue together into an isomorphism of bicycles $\eta(\Phi):\Phi \to \Phi_{\phi(1)(id_{1^+})}$.
\end{proof}

\begin{coro}
\label{Kbar(C)(1)-equiv-C}
For any symmetric monoidal category $C$, the category $\Kbar C(1^+)$ is
equivalent to $C$.
\end{coro}
\begin{proof}
For each permutative category $C$ we define a functor 
$I(C):\PsBikes{\gn{1}}{C}  \to C$. On objects this functor is defined as follows:
\[
 I(C)(\Phi) = \phi(1)(id_{1^+}),
\]
where $\L = (\phi, \alpha)$ is the underlying lax cone of $\Phi$.
 For a morphism of (pseudo) bicycles $F:\Phi \to \Psi$ we define
 \[
  I(F) = F(1)(id_{1^+}).
 \]
 This functor is inverse of the functor $\Phi_{-}$.

\end{proof}

 \begin{sloppypar}
 It was shown by Leinster in \cite{Leinster} that the degree one category of a
 $\EinC$ has a symmetric monoidal structure. We want to explore the homotopy properties of the unit natural transformation $\eta$ of the adjunction
 $(\PNat, \Kbar)$.
 \end{sloppypar}
 \begin{lem}
 \label{unit-map-equiv}
 For each $\EinC$ $X$ the unit map
 \[
 \eta(X):X \to \Kbar(\PNat(X))
 \]
 is a strict equivalence of $\gCats$.
 \end{lem}
 \begin{proof}
 The $\gCat$ $\Kbar(\PNat(X))$ is a coherently commutative monoidal category, therefore
  $\eta(X)$ is a morphism between two $\EinCs$.
 Now in light of Lemma \ref{Bikes-nSegPsBikes-eq} it would be sufficient to show that the degree one functor
 \[
 \eta(X)(1^+):X(1^+) \to \PsBikes{\gn{1}}{\PNat(X)}
 \]
 is an equivalence of categories. We recall the definition of the functor $\eta(X)(1^+)$.
 For each $x \in X(1^+)$, the strict symmetric monoidal functor $\eta(X)(1^+)(x) = \Phi = (\L, \sigma)\Bike{}{\gn{1}}{\PNat(X)}$ is defined as follows:
 \[
 \phi(n)(f) := ((n), X(f)(x)),
 \]
 where $\L = (\phi, \alpha)$ is the underlying lax cone of $\Phi$. In light of corollaries \ref{Thicken-CCM-Cat} and \ref{Kbar(C)(1)-equiv-C} we have the following commutative diagram in $\Cat$:
 \begin{equation*}
 \xymatrix@C=18mm{
 X(1^+) \ar[r]^{\eta(X)(1^+) \ \ \ \ \ \ } \ar[rd]_{i} & \PsBikes{\gn{1}}{\PNat(X)} \ar[d]^{I} \\
 & \PNat(X)
 }
 \end{equation*}
 where the vertical functor $I$ in the diagram above is the functor from corollary \ref{Kbar(C)(1)-equiv-C} and the digonal functor $i$ in the above diagram is the functor from corollary \ref{Thicken-CCM-Cat}. The two corollaries mentioned above say that $i$ and $I$ are equivalences of categories therefore by the two out of three property of the natural model category $\Cat$ the unit map in degree one $\eta(X)(1^+)$ is an equivalence of categories.
 Hence we have proved that the map of coherently commutative monoidal categories $\eta(X)$ is a (strict) equivalence of $\gCats$. 
 \end{proof}
 The adjoint functors $\PNat$ and $\Kbar$ have sufficiently good properties which ensure that the above lemma implies that for each $\gCat$ $X$ the counit map $\eta(X)$ is a coherently commutative monoidal equivalence.
 \begin{coro}
  \label{unit-eqiv.}
  For a $\gCat$ $X$, the unit map $\eta(X):X \to \Kbar(\PNat(X))$ is a cohently commutative monoidal equivalence. 
 \end{coro}
 \begin{proof}
 Let $r:X \to X^f$ denote a fibrant replacement of $X$ in the model category of coherently commutative monoidal categories.
 In other words $X^f$ is a coherently commutative monoidal category and $r$ is an acyclic cofibration in the model category of coherently commutative monoidal categories.
 Since $\eta$ is a natural transformation therefore
 we have the following commutative diagram in $\gCAT$
 \begin{equation*}
  \xymatrix{
  X^f  \ar[r]^{\eta(X^f) \ \ \ \ } & \Kbar(\PNat(X^f))  \\
  X \ar[u]^{r} \ar[r]_{\eta(X) \ \ \ } & \Kbar(\PNat(X)) \ar[u]_{\Kbar(\PNat(r))}
  }
 \end{equation*}
 The above lemma tells us that the morphism $\eta(X^f)$ is a coherently commutative monoidal equivalence and so is $r$ by assumption. Theorem \ref{char-CC-Mon-eq} and lemma \ref{pres-reflect-Eq-cat}
 together imply that $\Kbar(\PNat(r))$ is a coherently commutative monoidal equivalence. Now the 2 out of 3 property of model categories implies that $\eta(X)$ is a coherently commutative monoidal equivalence.
 \end{proof}

 Finally we have developed enough machinery to provide a characterization of
 a coherently commutative monoidal equivalence.
 
 \begin{thm}
 \label{char-CC-Mon-eq}
 A morphism of $\gCats$ $F:X \to Y$ is a coherently commutative monoidal equivalence if and only if the strict symmetric monoidal functor
 $\PNat(F):\PNat(X) \to \PNat(Y)$ is an equivalence of (permutative) categories.
 \end{thm}
 \begin{proof}
 Let us first assume that the morphism of $\gCats$ $F$ is a
 coherently commutative monoidal equivalence.
 Any choice of a cofibrant replacement functor $Q$ for $\gCAT$ provides a commutative
 diagram
 \begin{equation*}
 \xymatrix{
 Q(X) \ar[r]^{Q(F)} \ar[d] &Q(Y) \ar[d] \\
 X \ar[r]_F & Y
 }
 \end{equation*}
 The vertical maps in this diagram are acyclic fibrations in the model
 category of coherently commutative monoidal categories which are
 strict equivalences of $\gCats$.
 Applying the functor $\PNat$ to this commutative diagram we get the following
 commutative diagram in $\PCat$
 \begin{equation*}
 \xymatrix{
 \PNat(Q(X)) \ar[r]^{\PNat(Q(F))} \ar[d] & \PNat(Q(Y)) \ar[d] \\
 \PNat(X) \ar[r]_{\PNat(F)} & \PNat(Y)
 }
 \end{equation*}
 The functor $\PNat$ is a left Quillen functor therefore it preserves
 weak equivalences between cofibrant objects. This implies that the top
 horizontal arrow in the above diagram is a weak equivalence in $\PCat$.
 The above lemma \ref{PNat-hocolim} implies that the vertical maps in the
 above diagram are weak equivalences in $\PCat$. Now the two out of three
 property of weak equivalences in model categories implies that
 $\PNat(F)$ is a weak equivalence in $\PCat$.
 
 Conversely, let us first assume that $\PNat(F):\PNat(X) \to \PNat(Y)$ is a weak equivalence between $\EinCs$ in $\PCat$.
 The functor $\Kbar$ preserves equivalences in $\PCat$, therefore the morphism
 $\Kbar(\PNat(F)):\Kbar(\PNat(X)) \to \Kbar(\PNat(Y))$ is a strict equivalence of $\gCats$. Now we have the following commutative diagram
 \begin{equation*}
 \xymatrix{
 \Kbar(\PNat(X)) \ar[r]^{\Kbar(\PNat(F))} & \Kbar(\PNat(Y)) \\
 X \ar[u]^{\eta(X)} \ar[r]_F & Y \ar[u]_{\eta(Y)}
 }
 \end{equation*}
 Lemma \ref{PNat-hocolim} implies that the two vertical arrows in the above
 diagram are strict equivalences of $\gCats$, therefore
 by the two-out-of-three property of model categories, $F$ is also a
 coherently commutative monoidal equivalence. Now we tackle the general case. Let $F:X \to Y$ be a morphism of $\gCats$ such that $\PNat(F)$ is
 an equivalence of categories. By a choice of a functorial factorization functor
 we get the following commutative diagram whose vertical arrows are acyclic
 cofibrations in the model category of coherently commutative monoidal categories and $R(X)$ and $R(Y)$ are coherently commutative monoidal categories:
 \begin{equation*}
 \xymatrix{
 R(X) \ar[r]^{R(F)} & R(Y) \\
 X \ar[u]^{\zeta(X)} \ar[r]_F & Y \ar[u]_{\zeta(Y)}
 }
 \end{equation*}
 Applying the functor $\PNat$ to the above diagram we get the following commutative diagram in $\PCat$:
 \begin{equation*}
 \xymatrix{
 \PNat(R(X)) \ar[r]^{\PNat(R(F))} & \PNat(R(Y)) \\
 \PNat(X) \ar[u]^{\PNat(\zeta(X))} \ar[r]_{\PNat(F)} & \PNat(Y) \ar[u]_{\PNat(\zeta(Y))}
 }
 \end{equation*}
 Since $\PNat$ is a left Quillen functor therefore it preserves
 acyclic cofibrations. This implies that the two vertical morphisms in the above
 diagram are equivalences of categories. By assumption $\PNat(F)$
 is an equivalence of categories therefore the two out of three property implies that $\PNat(R(F))$ is an equivalence of categories. The discussion earlier in this proof regarding strict equivalence between coherently commutative monoidal categories implies that $R(F)$ is a strict equivalence of $\gCats$.
 \end{proof}

 The lemmas proved in this section and the results of appendix \ref{Mdl-CAT-Perm} together imply the main result of this paper which is the following:
 \begin{thm}
 \label{main-res}
 The adjunction $(\PNat, \Kbar)$ is a Quillen equivalence.
 \end{thm}
 \begin{proof}
 We observe that $\PNat(n)$ is a cofibrant permutative category for all 
 $n \in \Nat$. Since the permutative category $\PNat(n)$ is cofibrant for all $n \ge 0$ therefore it is easy to check that the right adjoint functor $\Kbar$ preserves fibrations and trivial fibrations in $\PCat$ and therefore $(\PNat, \Kbar)$ is a Quillen adjunction.
 Let $X$ be a cofibrant object in the model category of coherently commutative monoidal categories and let $C$ be a permutative category.
  We will show that a map $F:\PNat(X) \to C$ is a coherently commutative monoidal equivalence if and only if its adjunct map $\phi(F):X \to \Kbar{C}$
   is an equivalence of categories. Let us first assume that $F$ is an equivalence in $\PCat$. The adjunct map $\phi(F)$ is defined by the following commutative diagram:
 \begin{equation*}
  \xymatrix{
  \Kbar(\PNat(X)) \ar[r]^{\Kbar(F)} & \Kbar(C) \\
   X \ar[u]^\eta \ar[ru]_{\phi(F)}
  }
  \end{equation*}
  The right adjoint functor $\Kbar$ preserves weak equivalences therefore the top horizontal arrow is a strict equivalence of $\gCats$.
  The unit map $\eta$ is a coherently commutative monoidal equivalence by corollary \ref{unit-eqiv.}. Now the 2 out of 3 property of model categories implies that $\phi(F)$ is also a coherently commutative monoidal equivalence. 
  
Conversely, let us assume that $\phi(F)$ is a coherently commutative monoidal equivalence. The 2 out of 3 property of model categories implies that top horizontal arrow in the above commutative diagram, namely $\Kbar(F)$ is a coherently commutative monoidal equivalence and therefore a strict equivalence of $\gCats$. Now Lemma \ref{pres-reflect-Eq-cat} implies that the strict symmetric monoidal functor $F$
is an equivalence of categories.
 \end{proof}
Now we are ready to state the main result of this paper which is a corollary of the above theorem:
\begin{coro}
 \label{main-res-paper}
 The adjunction $(\PStr, \K)$ is a Quillen equivalence.
 \end{coro}
 \begin{proof}
 Remark \ref{nat-transf-ext-right} gives a natural equivalence of permutative categories
 \begin{equation*}
T \circ (id_{\Pi_1} \circ (J \circ i) \circ \inv{I}) :\PStr(n) \overset{\inv{I}} \Rightarrow \Pi_1\circ \PLbb(-) \Rightarrow \Pi_1\circ \Lbb(\gn{-}) \Rightarrow \Pi_1\circ \Lbb(-) \overset{T} \Rightarrow \PNat(-).
\end{equation*}
We observe that for all $n \in \Nat$ 
$T \circ (id_{\Pi_1} \circ (J \circ i) \circ \inv{I})(n)$ is a weak equivalence in $\PCat$ between cofibrant (and fibrant) permutative categories.

We recall from appendix \ref{Mdl-CAT-Perm} that the bifunctor $\StrSMHom{-}{-}$ is the Hom functor of the $\Cat$-model category $\PCat$. This implies that for each $n \in \Nat$ and each permutative category $C$, the functor
\begin{equation*}
\StrSMHom{T \circ (id_{\Pi_1} \circ (J \circ i) \circ \inv{I})(n)}{C}:\StrSMHom{\PStr(n)}{C} \to \StrSMHom{\PNat(n)}{C}
\end{equation*}
is an equivalence of categories. In other words the natural transformation
\begin{equation*}
\StrSMHom{T \circ (id_{\Pi_1} \circ (J \circ i) \circ \inv{I})(-)}{C}:\StrSMHom{\PStr(-)}{C} \to \StrSMHom{\PNat(-)}{C}
\end{equation*}
is a strict equivalence if $\gCats$.
 This morphism of $\gCats$ uniquely determines a strict equivalence of $\gCats$ $\eta(C):\K(C)\Rightarrow \Kbar(C)$ for each permutative category $C$. The family $\lbrace \eta(C) \rbrace_{C \in Ob(\PCat)}$ glues together to define a natural equivalence $\eta:\K \Rightarrow \Kbar$.
 The natural equivalence $\eta$ induces a natural isomorphism between the derived functors of $\K$ and $\Kbar$. Now the corollary follows from the above theorem \ref{main-res}.
 \end{proof}

%
%
%
%
%
%
%
 \appendix
\section[Model category structure on $\PCat$]{Model category structure on $\PCat$}
\label{Mdl-CAT-Perm}
In this appendix we provide a proof of the natural model category structure on the category of all (small)
permutative categories $\PCat$ which was defined in Theorem \ref{nat-model-str-Perm}. $\PCat$ is a reflective subcategory of $\Cat$ \emph{i.e.} the forgetful functor 
 $U:\PCat \to \Cat$ has a left adjoint $F:\Cat \to \PCat$ which assigns to a category $B$,
 the free permutative category generated by $B$.
 We will construct the model category structure
by \emph{transfer} along the adjunction $(F, i)$. The main tool used here
will be the following theorem
\begin{thm} \cite[Theorem 3.6]{goer-sch}
\label{mdl-str-transfer-tool}
Let $F : C \rightleftharpoons D : G$ be an adjoint pair and suppose $C$ is a
cofibrantly generated model category. Let $I$ and $J$ be chosen sets of generating
cofibrations and acyclic cofibrations, respectively. Define a morphism $f : X \to Y$
in $D$ to be a weak equivalence or a fibration if $G(f)$ is a weak equivalence or fibration
in $C$. Suppose further that
\begin{enumerate}
\item The right adjoint $G : D \to C$ commutes with sequential colimits; and
\item Every cofibration in $D$ with the LLP with respect to all fibrations is a weak
equivalence.
\end{enumerate}
Then $D$ becomes a cofibrantly generated model category. Furthermore the collections
$\lbrace F(i) | i \in I \rbrace$ and $\lbrace F(j) | j \in J \rbrace$ generate the cofibrations and the acyclic cofibrations
of D respectively.
\end{thm}
The categories $\Cat$ and $\PCat$ are locally presentable so the first condition is satisfied by the
adjunction $(F, i)$, see \cite{AR94}. The first half of this section is devoted to verifying the second condition
of the above theorem. In the second half of this appendix we will show that the transferred
model category structure on $\PCat$ is the same as the one claimed in Theorem \ref{nat-model-str-Perm}.
We will verify the second condition in Theorem \ref{mdl-str-transfer-tool} by using a factorization
of maps in $\PCat$ as acyclic cofibrations followed by fibrations. We begin by constrcting this factorization.

  Let $C$ be a small category, the functor $\partial_1=[d_1,C]:[J,C]\to C$ is the \emph{source functor} which which takes an isomorphism $a:A_0\to A_1$ to its source $A_0$, and $\partial_0=[d_0,C] $ is the \emph{target functor} which takes an isomorphism $a:A_0\to A_1$ to its target $A_1$. If $s$ denotes the functor $J \to 0$, then the functor $\sigma=[s,C]:C\to [J,C]$ is the \emph{unit functor} which takes an object $A\in C$ to the unit isomorphism $id_A:A\to A$. The relation $s d_1 = id_1=s d_0$  implies that we have $ \partial_1 \sigma =id_{C}=\partial_0 \sigma$. The functors $\partial_1,\partial_0$ and $\sigma$ are equivalences of categories, since the functors $d_1,d_0$ and $s$ are equivalences. We observe that the following functor
\[
[J, (pr_1,pr_2)]:[J, C \times C] \to [J, C] \times [J, C]
\]
is an isomprphism of categories, where $pr_1:C \times C \to C$ and
$pr_2:C \times C \to C$ are the obvious projection functors.
 Let us further assume that $C$ is a permutative category, then the
 arrow category $[J, C]$ inherits the structure of a permutative category by
 the following bifunctor
 \begin{equation*}
\label{inh-tens-prod-isom-cat}
- \underset{[J,C]} \otimes -:[J, C] \times [J, C] \overset{\inv{[J, (pr_1,pr_2)]}} \to [J, C \times C] \overset{[J, - \underset{C} \otimes -]} \to [J, C].
\end{equation*}
and the following symmetry natural isomorphism
\begin{equation*}
\label{inh-sym-nat-isom-cat}
\gamma_{[J, C]}:[J, C] \times [J, C] \overset{\inv{[J, (pr_1,pr_2)]}} \to [J, C \times C] \overset{[J, \gamma_C]} \to[J, [J, C]].
\end{equation*}
We observe that the following two diagrams commute
\[
  \xymatrix{
 [J, C] \times [J, C]  \ar[d]_{\partial_0 \times \partial_0} \ar[r]^{\ \ \ \ \ - \underset{[J,C]} \otimes -}  &[J, C] \ar[d]^{\partial_0}  && [J, C] \times [J, C]  \ar[d]_{\partial_1 \times \partial_1} \ar[r]^{\ \ \ \ \ - \underset{[J,C]} \otimes -}  &[J, C] \ar[d]^{\partial_1} \\
 C \times C \ar[r]_{- \underset{C} \otimes -} &C && C \times C \ar[r]_{- \underset{C} \otimes -} &C
 }
\]
 which implies that the functors $\partial_0$ and $\partial_1$ preserve
 the symmetric monoidal structure. We further observe that the following
 two diagrams commute
 \[
  \xymatrix{
 [J, C] \times [J, C]  \ar[d]_{\partial_0 \times \partial_0} \ar[r]^{ \gamma_{[J,C]}}  &[J,[J, C]] \ar[d]^{\partial_0}  && [J, C] \times [J, C]  \ar[d]_{\partial_1 \times \partial_1} \ar[r]^{\ \ \ \ \ \gamma_{[J,C]}}  &[J,[J, C]] \ar[d]^{\partial_1} \\
 C \times C \ar[r]_{\gamma_C} &C && C \times C \ar[r]_{\gamma_C} &C
 }
\]
which implies that the functors $\partial_0$ and $\partial_1$ preserve
 the symmetry natural isomorphisms. Thus the functors
 $\partial_0$ and $\partial_1$ are strict symmetric monoidal functors.
 Similarly we see that the following two diagrams commute
 \[
  \xymatrix{
 [J, C] \times [J, C]  \ar[r]^{\ \ \ \ \ - \underset{[J,C]} \otimes -}  &[J, C]  && [J, C] \times [J, C]   \ar[r]^{\ \ \ \gamma_{[J,C]}}  &[J,[J, C]] \\
 C \times C  \ar[u]^{\sigma \times \sigma} \ar[r]_{\ \ \ \ \ - \underset{C} \otimes -} &C  \ar[u]_{\sigma} && C \times C  \ar[u]^{\sigma \times \sigma} \ar[r]_{\gamma_C} &C \ar[u]_{\sigma}
 }
\]
 This implies that the functor $\sigma$ is a strict symmetric monoidal functor.
 The cartesian product $C \times C$ is a product of two copies of $C$
 in the category $\PCat$, therefore the functor
 \[
 (\partial_0, \partial_1):[J, C] \to C \times C
 \]
 is a strict symmetric monoidal functor.
 \begin{prop}
 The functor
 \[
 (\partial_0, \partial_1):[J, C] \to C \times C
 \]
 is an isofibration and the functor $\sigma:C\to [J,C]$ is an equivalence of categories. Moreover, the functors $\partial_1$ and $\partial_0$ are equivalences surjective on objects.
 \end{prop}
 \begin{proof}
 Let us show that the functor $(\partial_1,\partial_0)$ is an isofibration. Let $ a:A_0\to A_1$ be an object of $[J, C]$ and let $(u_0,u_1):(A_0,A_1)\to (B_0,B_1)$ be an isomorphism in $C \times C$. There is then a unique isomorphism $b:B_0\to B_1$ such that the square
 \[
  \xymatrix@C=11mm{
 A_0  \ar[d]_{u_0} \ar[r]^a  &A_1 \ar[d]^{u_1}   \\
 B_0 \ar[r]_{b} &B_1
 }
\]
 commutes, namely $b = u_1 \circ a \circ \inv{u_0}$. The pair $u=(u_0,u_1)$ defines an isomorphism $a\to b$ in the category $[J, C]$, and we have $(\partial_1,\partial_0)(u)=(u_0,u_1)$. This proves that $(\partial_1,\partial_0)$ is an isofibration. We saw above that the functor  $\partial_1,\partial_0$ and $\sigma$ are equivalences of categories. The functor $\partial_1$ is surjective on objects, since $\partial_1\sigma=id_{C}$. Similarly, the functor $\partial_0$ is surjective on objects. 
 \end{proof}
 \begin{df}
 The mapping path object of a strict symmetric monoidal functor $F:X \to Y$ is the category $\POb{F}$ defined by the following pullback square.
 \begin{equation*}
 \label{mapping-path-object}
  \xymatrix@C=11mm{
  X \ar@{-->}[rd]_{i_X} \ar@/^/[rrd]^{\sigma F}  \ar@/_/[rdd]_{(id_X, F)} \\
 &\POb{F} \ar[d]^{(P_X, P_Y)} \ar[r]^P  &[J, Y] \ar[d]^{ (\partial_0, \partial_1)}   \\
 &X \times Y \ar[r]_{F \times id_Y} &Y \times Y
 }
\end{equation*}
 \end{df}
 There is a (unique) functor $ i_{X}:{X} \to \mathbf{P}(F) $ such that $Pi_X= \sigma F$,  $P i_{X} = \sigma F$ and $P_{{X}} i_{{X}}=id_{{X}}$ since square \eqref{mapping-path-object} is cartesian and we have $ \partial_1\sigma F=id_{{Y}} F =F id_{{X}}$. Let us put $P_{{Y}}=\partial_0 P$. Then we have
$F=P_{{Y}} i_{{X}}:{X}\to \mathbf{P}(F)\to {Y}$
since $P_{{Y}} i_{{X}}=\partial_0 P i_{{X}}=\partial_0 \sigma F=id_{{Y}} F =F$. This is the \emph{mapping path factorisation} of the functor $F$ in the category $\PCat$. We now present a concrete construction of the pullback above. An object of $\POb{F}$ is a triple $ (y,A,B)$, where $A$ is an object of $X$, $B$ is an object of $Y$ and $y:F(A)\to B$ is an isomorphism in $Y$. We have $P(y,A,B)=y$, $P_{X}(y,A,B)=A$ and $ P_{Y}(y,A,B)=B$. A morphism $(y,A,B)\to (y',A',B')$ in the category $ \POb{F}$ is a pair of maps $u:A\to A'$ and $v:B\to B'$ such that the  following diagram commutes:
\begin{equation*}
 \label{mor-in-mapping-path-object}
  \xymatrix@C=11mm{
 F(A) \ar[r]^y \ar[d]_{F(u)} &B \ar[d]^{v}   \\
 F(A') \ar[r]_{y'} & B'
 }
\end{equation*}
 Our construction defines a permutative category with the obvious tensor product namely, $(y,A,B) \otimes (y',A',B') = (y \otimes y', A \otimes A', B \otimes B')$ and $(u, v) \otimes (u_1, v_1) = (u \otimes u_1, v \otimes v_1)$.
\begin{lem}
\label{fact-acy-cof-perm}
The functor $P_{{Y}}$ in the mapping path factorisation
\[
F = P_Y i_X:X \to \mathbf{P}(F) \to Y
\]
is an isofibration and the functor $i_{{X}}$ is an equivalence of categories.
\end{lem}
\begin{proof}
 The map $(P_X, P_Y)$ is a pullback (in $\Cat$) of the isofibration $(\partial_0, \partial_1)$ therefore it is an isofibration. Clearly the projection map $\pi_2:X \times Y \to Y$
 is an isofibration. We observe that $P_Y = \pi_2 \circ (P_X, P_Y)$ and therefore it is also an isofibration. Now we show that the functor $i_{X}$ is an equivalence. For this it suffices to exibit a natural isomorphism $\alpha: i_{X} P_{X} \simeq id_X$  since we already have $P_{X}i_{X}= id_{X}$. But $ i_{X} P_{X}(y,A,B)=(id_{F A},A,F A)$ and the pair $(id_A, y)$ defines a natural isomorphism $\alpha((y,A,B)):(id_{F A}, A, F A) \to (y,A,B)$.

\end{proof}
The above results lead us to the main result of this section which was stated in Theorem \ref{nat-model-str-Perm}
\begin{proof} (of Theorem \ref{nat-model-str-Perm})
	As indicated at the beginning of this section, the main tool for proving this theorem will be \ref{mdl-str-transfer-tool}. Since both $\Cat$ and $\PCat$ are locally presentable categories therefore
	Theorem \ref{mdl-str-transfer-tool} $(1)$ follows from \cite[Prop. 2.23]{AR94}. Now we verify Theorem \ref{mdl-str-transfer-tool} $(2)$. Let $i:C \to D$ be a map in $\PCat$ which has the right lifting property with respect to all fibrations. By Lemma \ref{fact-acy-cof-perm} we have the following (outer) commutative diagram in $\PCat$ in which $i_C$ is an acyclic cofibration and $P_D$ is a fibration in $\PCat$:
	\begin{equation*}
	\xymatrix{
	C \ar[r]^{i_C} \ar[d]_i & P(i) \ar[d]^{P_D} \\
	D \ar@{=}[r] \ar@{-->}[ru] & D
	 }
	\end{equation*}
	By assumption the strict symmetric monoidal functor $i$ has the left lifting property with respect to all fibrations in $\PCat$ therefore there exists the (dotted) lifting arrow in the above commutative square. By two-out-of-six property of model categories $i$, $P_D$ and the lifting arrows are all weak equivalences because $i_C$ is a weak equivalence. Thus we have verified Theorem \ref{mdl-str-transfer-tool} $(2)$ and hence proved the Theorem \ref{nat-model-str-Perm}.
\end{proof}

 Now we have verified all conditions of Theorem \ref{mdl-str-transfer-tool} and thus we have
 transferred a model category structure on the category $\PCat$.
Finally, we show that the thickened version of the Segal Nerve functor
preserves weak equivalences.

\begin{thm}
\label{transfer-cond.}
Let $F: C \to D$ be a strict symmetric monoidal functor in $\PCat$ such that
$F$ has the left lifting property with respect to all maps $p:A \to B$ in $\PCat$
having the property that $\Kbar(p)$ is a fibration in the strict model category structure on
$\gCAT$. Then $F$ is an equivalence of categories.
\end{thm}
\begin{proof}
Using the above factorization, the functor $F$ can be factored as $P_C i_C$,
where $P_C$ is an isofibration and therefore has the property that $\Kbar(P_C)$
is a fibration in the strict model category structure on $\gCAT$. This means that
there is a dashed lifting arrow $r$ in the following diagram:
\[
  \xymatrix@C=11mm{
 C  \ar[d]_{F} \ar[r]^{i_X} &\mathbf{P}(F) \ar[d]^{P_C}   \\
 D \ar@{=}[r]  \ar@{-->}[ru]_r &D
 }
\]
We can view the above diagram in $\Cat$ by forgetting the symmetric monoidal structures.
The existence of the lifting arrow $r$ in the above diagram implies that the map $F$ is a retract of $i_X$ \emph{i.e.} we have the
following commutative diagram
\[
  \xymatrix@C=11mm{
 C  \ar[d]_{F} \ar@{=}[r] & C \ar[d]_{id_X} \ar@{=}[r] & C \ar[d]_{F}   \\
 D \ar[r]_r  &\mathbf{P}(F) \ar[r]_{P_C} &D
 }
\]
Lemma \ref{fact-acy-cof-perm} says that the map $i_X$ is a weak equivalence in the
natural model category structure on $\Cat$ therefore $F$ is also a weak equivalence in the
same model category. Since weak equivalences in the natural model category structure on $\Cat$
are equivalences of categories, we have proved that the map $F$ is an equivalence of categories.
\end{proof}

Now we will like to show that the notion of weak equivalence in the
transferred model category structure on $\PCat$ is the same as that in
the natural model category structure on $\PCat$, see \ref{nat-model-str-Perm}.
\begin{lem}
\label{pres-reflect-Eq-cat}
 A strict symmetric monoidal functor in $\PCat$ is an equivalence of categories
 if and only if its image under the right adjoint functor $\Kbar$ is a strict
 equivalence in $\gCAT$.
 \end{lem}
 \begin{proof}
  Let $F:C \to D$ be a strict symmetric monoidal functor in $\PCat$ which
  is an equivalence of categories. We consider the following diagram
  \begin{equation}
  \label{natural-equiv-Kbar(C)(1)-C}
  \xymatrix@C=20mm{
 \Kbar(C)(1^+)= \PsBikes{\gn{1}}{C}  \ar[d]_{I(C)}
 \ar[r]^{\PsBikes{\gn{1}}{F}} & \PsBikes{\gn{1}}{D} = \Kbar(D)(1^+) \ar[d]^{I(D)}   \\
 C \ar[r]_{F} &D
 }
 \end{equation}
 where $I(C)$ and $I(D)$ are equivalences defined in the proof of corollary \ref{Kbar(C)(1)-equiv-C}.
 If the functor $F$ is an equivalence then the two out of three property of
 weak equivalences says that the map $\Kbar(F)(1^+) = \PsBikes{\gn{1}}{F}$ is
 an equivalence of categories. Since $\Kbar(F)$ is a map between two
 $\Ein$ $\gCats$, therefore it is a strict weak equivalence if and only if
 $\Kbar(F)(1^+)$ is an equivalence of categories.
 
 Conversely, if the morphism of $\gCats$ $\Kbar(F)$ is a (strict) weak equivalence
 then $\Kbar(F)(1^+)$ is an equivalence of categories. Another application of
 the two out of three property of weak equivalences to the commutative diagram
 \eqref{natural-equiv-Kbar(C)(1)-C} tells us that the functor $F$ is an equivalence of categories.
  
 \end{proof}
 
 Finally we would like to show that the natural model category of permutative categories $\PCat$ is enriched over $\Cat$. The enrichment of $\PCat$ is given by the following bifunctor which assigns to each pair of permutative categories $(C, D)$ the category $\StrSMHom{C}{D}$ whose objects are strict symmetric monoidal functors from $C$ to $D$ and unital monoidal natural transformations between them:
 \begin{equation}
 \StrSMHom{-}{-}:\PCat^{op} \times \PCat \to \Cat
 \end{equation}
  The category $\PCat$ is tensored over $\Cat$ and the \emph{tensor product} is given by a bifunctor
 \begin{equation}
 \label{tensP-bifunc}
 - \boxtimes -:\Cat \times \PCat \to \PCat
 \end{equation}
 which maps a pair $(A, C) \in Ob(\Cat \times \PCat)$ to the (product) permutative category $A \boxtimes C$, where $F:\Cat \to \PCat$ is the free permutative category functor. 
 
%

 We recall that the category of functors $[A, C]$ gets a permutative category structure. We claim that this category of functors is the \emph{cotensor} of $A$ with $C$. For each permutative category $D$ we consider the following two functors $\StrSMHom{-}{D}:\PCat^{op} \to \Cat$ and $[-, D]:\Cat \to \PCat^{op}$. We will define a natural transformation
 \[
 \eta_r:id_{\Cat} \Rightarrow \StrSMHom{[-, D]}{D}.
 \]
 For each category $A$ we will define a functor $\phi(A):A \to \StrSMHom{[A, D]}{D}$. For each object $a \in Ob(A)$, we define
 \[
 \phi(A)(a) := ev_a:[A, D] \to D,
 \]
 where $ev_a$ is the evaluation functor at $a$. Clearly this is a strict symmetric monoidal functor. To each map $f:a \to b$ in $A$ we have a natural transformation $\phi(A)(f) := ev_f:ev_a \Rightarrow ev_b$ which is defined at each $F \in Ob([A, D])$ by $\phi(A)(f)(F) := F(f):F(a) \to F(b)$. In order to prove our claim that $[A, D]$ is a cotensor, we will establish the universality of $\eta_r$ \emph{i.e.} for each functor $G:A \to \StrSMHom{C}{D}$, we will show that there exists a unique strict symmetric monoidal functor $g:C \to [A, D]$ such that the following diagram commutes
 \begin{equation*}
 \xymatrix{
 A \ar[r]^{\eta_r(A) \ \ \ \ \ \ \ } \ar[rd]_G &\StrSMHom{[A, D]}{D} \ar[d]^{\StrSMHom{g}{D}} \\
 & \StrSMHom{C}{D}
 }
 \end{equation*}
 For each $c \in Ob(C)$, the functor $g(c):A \to D$ is defined on objects as follows:
 \[
 g(c)(a) := G(a)(c).
 \]
 The functor $g(c)$ is defined similarly on morphims \emph{i.e.} for any map $f \in Mor(C)$
 \[
 g(c)(f) := G(f)(c).
 \]
 This defines a functor because $G$ is a functor and $g$ is strict symmetric monoidal because $G(a)$ is strict symmetric monoidal for each $a \in Ob(A)$. Now the following equality is a consequence of the definition of $g$:
 \[
  G = \StrSMHom{g}{D}  \circ \eta_r(A)
  \]
 The uniqueness of $g$ is obvious. This implies that we have the following isomorphism
 \[
 \PCat(C, [A, D]) \cong \Cat(A, \StrSMHom{C}{D}).
 \]
 The universality of $\eta_r$ implies that the above isomorphism is natural in $A$ and $C$. One can check that it is also natural in $D$.
 Thus we have proved our claim that $\PCat$ is cotensored over $\Cat$. A consequence of our definition of cotensor is that it defines a bifunctor which we denote as follows:
 \begin{equation}
 [-,-]:\Cat^{op} \times \PCat \to \PCat.
 \end{equation}
 
 \begin{prop}
 \label{tens-Quillen-bifunc}
 The tensor product bifunctor \eqref{tensP-bifunc} is a Quillen bifunctor.
 \end{prop}
 \begin{proof}
 The pentuple $(- \boxtimes -, [-,-], \StrSMHom{-}{-}, \inv{\eta_r}, \eta_l)$ defines an adjunction of two variables. In light of \cite[lemma 4.2.2]{Hovey} it would be sufficient to show that whenever we have a functor $i:W \to X$ which is monic on objects and a strict symmetric monoidal functor $p:Y \to Z$ which is a fibration in $\PCat$, the following induced functor is a fibration in $\PCat$ which acyclic whenever $i$ or $p$ is acyclic:
 \begin{equation*}
 i \Box p:[X, Y] \to [X, Z] \underset{[W, Z]} \times [W, Y]
 \end{equation*}
 The strict symmetric monoidal functor $i \Box p$ is a (acyclic) fibration if and only if the (ordinary) functor 
 \[
 U(i \Box p):U([X, Y]) \to U([X, Z] \underset{[W, Z]} \times [W, Y]) = U([X, Z]) \underset{U([W, Z])} \times U([W, Y])
 \]
  is a (acyclic) fibration in $\Cat$. Since the natural model category $\Cat$ is enriched over itself, therefore $U(i \Box p)$ is a fibration in $\Cat$ which is acyclic whenever $i$ or $p$ are acyclic because $i$ is a cofibration in $\Cat$ and $p$ is a fibration in $\Cat$.
 \end{proof}
 \begin{coro}
 \label{Hom-Q-Cof}
 The functor
 \[
  \StrSMHom{C}{-}:\PCat \to \Cat
  \]
   preserves fibrations and acyclic fibrations whenever the permutative category $C$ is cofibrant.
 \end{coro}
 \begin{proof}
 The above proposition and \cite[lemma 4.2.2]{Hovey} together imply that
 whenever we have a strict symmetric monoidal functor $i:W \to C$ which is a cofibration in $\PCat$ and another strict symmetric monoidal functor $p:Y \to Z$ which is a fibration in $\PCat$, 
 the (ordinary) functor
 \[
 i \Box p:\StrSMHom{C}{Y} \to \StrSMHom{C}{Z} \underset{\StrSMHom{W}{Z}} \times \StrSMHom{W}{Y}
 \]
 is a fibration in $\Cat$ which is acyclic if either $i$ or $p$ are acyclic.
 Now the corollary follows by choosing $W$ to be the terminal permutative category $\ast$.

 \end{proof}
 \begin{coro}
 \label{Hom-Cof-Fib}
 For each permutative category $C$, the functor
 \[
  \StrSMHom{-}{C}:\PCat \to \Cat
  \]
   maps cofibrations and acyclic cofibrations in $\PCat$ to fibrations and acyclic fibrations respectively in $\pCat$.
 \end{coro}
 The proof is an easy consequence of proposition \ref{tens-Quillen-bifunc}.

 \section[The notion of a Bicycle]{The notion of a Bicycle}
\label{NotionBike}
In the paper \cite{segal}, Segal described a functor from the category
of all (small) symmetric monoidal categories to the category of $\gCat$.
The $\gCat$ assigned by this functor to a symmetric monoidal category was
described by constructing a sequence of (pointed) categories
whose objects are a pair of families of objects and maps in the symmetric monoidal category satisfying some coherence conditions, see
\cite{mandell}, \cite{SK} for a complete definition. The
objective of this section is to present a \emph{thicker} version of Segal's pair of families as \emph{pseudo cones} which satisfy some additional coherence conditions which are usually associated to oplax symmetric monoidal functors. We begin by providing a definition of a pseudo cone in the spirit of Segal's families:

\begin{df}
A \emph{pseudo cone} from $X$ to $C$ is a pair $(\phi, \alpha)$, 
where $\phi = \lbrace \phi(\ud{n}) \rbrace_{\ud{n} \in Ob(\N)}$
is a family of functors $\phi(\ud{n}):X(n^+) \to C$
and $\alpha = \lbrace \alpha(f) \rbrace_{f \in Mor{\N}}$ is a family of natural isomorphisms
\[
\alpha(f):\phi(\ud{n}) \Rightarrow \phi(\ud{m}) \circ X(f),
\]
where $f:\ud{n} \to \ud{m}$ is a map in $\N$ which can also be regarded as an active map
$f:n^+ \to m^+$ in $\gop$. The pair $(\phi, \alpha)$ is subject to the
following conditions
\begin{enumerate}
%

\item In the family $\alpha$, the natural isomorphism indexed by the
 identity morphism of an object in $\N$ should be the identity natural transformation
 \emph{i.e.}
\[
\alpha(id_{\ud{n}}) = id_{\phi(\ud{n})},
\]
 for all $\ud{n} \in Ob(\N)$.
\item For every pair $\left(f:\ud{n} \to \ud{m}, g:\ud{m} \to \ud{l} \right)$ of composable arrows in $\N$, the following diagram commutes:
\[
\xymatrix@C=12mm{
 & \phi(\ud{n})   \ar@{=>}[rd]^{\alpha(g \circ f)}  \ar@{=>}[ld]_{\alpha(f)}\\
 \phi(\ud{m}) \circ X(f) \ar@{=>}[rr]_{\alpha(g) \circ id_{X(f)}} &&\phi(\ud{m}) \circ X(g \circ f)
 }
\]
\end{enumerate}
\end{df}
\begin{df}
\label{strict-cone}
A a \emph{strict cone} $(\phi, \alpha)$ from $X$ to $C$ is a pseudo cone such that all natural isomorphisms in the family
$\alpha$ are identity natural transformations.
\end{df}
Now we define a morphism between two pseudo cones
from $X$ to $C$, $\L_1 = (\phi, \alpha)$
and $\L_2 = (\psi, \beta)$.
\begin{df}
A morphism of pseudo cones $F:\L_1 \to \L_2$ consists of a family of natural transformations
$F = \lbrace F(\ud{n}) \rbrace_{\ud{n} \in Ob(\N)}$, having domain $\phi(\ud{n})$ and codomain $\psi(\ud{n})$, which is compatible with the families $\alpha$ and $\beta$ \emph{i.e.} the following diagram commutes
 \[
  \xymatrix@C=16mm{
 \phi(\ud{n}) \ar@{=>}[r]^{F(\ud{n})}  \ar@{=>}[d]_{\alpha(f)}  &\psi(\ud{n})
 \ar@{=>}[d]^{\beta(f)}\\
 \phi(\ud{m}) \circ X(f) \ar@{=>}[r]_{F(\ud{m}) \circ id_{X(f)}} &\psi(\ud{m}) \circ X(f)
 }
 \] 

\end{df}
Let $G: \L_2 \to \L_3 = (\upsilon, \delta)$ be another morphism of pseudo cones
from $X$ to $C$, then their composition is defined degreewise \emph{i.e.}
$G \circ F$ consists of the collection $\lbrace (G(\ud{n}) \cdot F(\ud{n})) \rbrace_{\ud{n} \in Ob(\N)} $. We observe that using the interchance law one gets the following equalities
\begin{multline*}
(G(\ud{m}) \circ id_{X(f)}) \cdot (F(\ud{m}) \circ id_{X(f)}) = (G(\ud{m}) \cdot F(\ud{m})) \circ (id_{X(f)} \cdot id_{X(f)}) \\
 = (G(\ud{m}) \cdot F(\ud{m})) \circ id_{X(f)}.
\end{multline*}
In other words the following diagram commutes
 \[
  \xymatrix@C=24mm{
 \phi(\ud{n}) \ar@{=>}[r]^{G(\ud{n}) \cdot F(\ud{n})}  \ar@{=>}[d]_{\alpha(f)}  &\upsilon(n)
 \ar@{=>}[d]^{\delta(f)}\\
 \phi(m) \circ X(f) \ar@{=>}[r]_{(G(\ud{m}) \cdot F(\ud{m})) \circ id_{X(f)}}   &\upsilon(\ud{m}) \circ X(f)
 }
 \] 
Thus the composite of two morphisms of pseudo cones, as
defined above, is a morphism of pseudo cones.
The associativity of vertical composition of natural transformations and the interchange law of natural transformations one can prove the following proposition:
\begin{prop}
The composition of morphisms of pseudo cones as defined above is strictly associative.
\end{prop}
Thus we have defined a category.
We denote this category of pseudo cones from $X$ to $C$
by $\PsCns{X}{C}$. In this paper we will be mainly concerned with 
pseudo cones having some additional structure which we describe next.
\begin{df}
A \emph{ Pseudo Bicycle} $\Phi$ from $X$ to $C$, denoted $\Phi:X \leadsto C$, consists of a triple
$\Phi = (\L, \sigma, \tau)$, where $\L = (\phi_\L, \alpha_\L)$ is the underlying pseudo
cone, $\tau:\phi_\L(\ud{0}) \Rightarrow \Delta(\unit{C})$ is a natural transformation to the constant functor on the category $X(0^+)$ taking value $\unit{C}$, and  $\sigma = \lbrace \sigma(\ud{k},\ud{l}) \rbrace_{(\ud{k}, \ud{l}) \in Ob(\N) \times Ob(\N)}$ is a family of natural transformations
\[
\sigma(\ud{k}, \ud{l}):\phi(\ud{k+l}) \Rightarrow \phi(\ud{k}) \odot \phi(\ud{l}).
\]
The functor $\phi(\ud{k}) \odot \phi(\ud{l}):X((k+l)^+) \to C$ on the right is defined by the following composite
\[
X((k+l)^+) \overset{X(\delta^{k+l}_k) \times (\delta^{k+l}_l)}\to X(k^+) \times X(l^+)
 \overset{\phi(\ud{k}) \times \phi(\ud{l})}\to C \times C \overset{-\otimes-}  \to C.
\]
This triple is subject to the following conditions:
\begin{enumerate}[label = {C.\arabic*}, ref={C.\arabic*}]
\item \label{unit} 
For any object $x \in X(m^+)$, the map
 \begin{multline*}
 \sigma(\ud{m}, \ud{0})(x):\phi(\ud{m + 0})(x) \to \phi(\ud{m})(x) \underset{C}\otimes \phi(0)(X(\delta^m_0)(x)) = (\phi(\ud{m}) \odot \phi(\ud{0}))(x) \\
 \overset{id \underset{C} \otimes \tau(X(\delta^m_0)(x))} \to \phi(\ud{m})(x) \underset{C} \otimes \unit{C}
 \end{multline*}
  is required to be the inverse of the (right) unit isomorphism in $C$. The map $\delta^m_0:m^+ \to 0^+$ in the arrow above is the unique map in $\gop$ from $m^+$ to the terminal object. Similarly the map
 \begin{multline*}
 \phi(\ud{0+m})(x) \overset{\sigma(\ud{0}, \ud{m})(x)} \to  \phi(\ud{0})(X(\delta^m_0)(x)) \underset{C} \otimes \phi(\ud{m})(x) = (\phi(\ud{0}) \odot \phi(\ud{m}))(x) \\
 \overset{\tau(X(\delta^m_0)(x)) \underset{C} \otimes id} \to    \unit{C}  \underset{C} \otimes  \phi(\ud{m})(x)
 \end{multline*}
 is the inverse of the (left) unit isomorphism in $C$. \item \label{symmetry} For each pair of objects $\ud{k}, \ud{l} \in Ob(\N)$,
we define a natural transformation
$\gamma_{\phi(\ud{k}),\phi(\ud{l})}:\phi(\ud{k}) \odot \phi(\ud{l}) \Rightarrow \phi(\ud{l}) \odot \phi(\ud{k})$
as follows:
\[
\gamma_{\phi(\ud{k}),\phi(\ud{l})} := \gamma^C \circ id_{\phi(\ud{k}) \times \phi(\ud{l})} \circ
 id_{(X(\delta^{k+l}_k), X(\delta^{k+l}_l))},
\]
where $\gamma^C$ is the symmetry natural isomorphism of $C$.
We require that each natural transformation $\sigma(\ud{k}, \ud{l})$
 in the collection $\sigma$ to satisfy the following symmetry condition
\begin{equation*}
 \xymatrix@C=12mm{
\phi(\ud{k+l}) \ar@{=>}[r]^{\phi(\gamma^{\N}_{\ud{k}, \ud{l}})} \ar@{=>}[d]_{\sigma(\ud{k}, \ud{l})} &\phi(\ud{l+k}) \ar@{=>}[d]^{\sigma(\ud{l},\ud{k})} \\
\phi(\ud{k}) \odot \phi(\ud{l}) \ar@{=>}[r]_{\gamma_{\phi(\ud{k}),\phi(\ud{l})}} & \phi(\ud{l}) \odot \phi(\ud{k})
}
\end{equation*}

\item \label{associativity}
For any triple of objects $\ud{k}, \ud{l}, \ud{m}$ in $\N$, the following diagram
commutes
\begin{equation*}
  \xymatrix@C=18mm{
 \phi((\ud{k+ l + m})) \ar@{=>}[r]^{\sigma(\ud{k+l},\ud{m})}  \ar@{=>}[d]_{\sigma(\ud{k}, \ud{l+m})}  &\phi(\ud{k+l}) \odot \phi(\ud{m})
 \ar@{=>}[dd]^{\sigma(\ud{k}, \ud{l}) \odot id_{\phi(\ud{m})}}\\
 \phi(\ud{k}) \odot \phi(\ud{l+m}) \ar@{=>}[d]_{id_{\phi(\ud{k})} \odot \sigma(\ud{l}, \ud{m})} \\
 \phi(\ud{k}) \odot (\phi(l)\odot \phi(m)) \ar@{<=}[r]_{\alpha_{\phi(\ud{k}),\phi(l),\phi(\ud{m})}} & (\phi(\ud{k}) \odot \phi(\ud{l})) \odot \phi(\ud{m})
 }
 \end{equation*}
 where the natural isomorphism $\alpha_{\phi(\ud{k}),\phi(\ud{l}),\phi(\ud{m})}$ is defined by the following diagram which, other than the bottom rectangle, is commutative:
 \begin{equation*}
\xymatrix{
X((k+l+m)^+) \ar[rd]_{F_2} \ar[r]^{F_1} & X((k+l)^+) \times X(m^+) \ar[r]^{F_3} & (X(k^+) \times X(l^+)) \times X(m^+) \ar@{=}[dd] \\
& X(k^+) \times X((l+m)^+) \ar[d]_{id \times (X(\partition{l+m}{l}), X(\partition{l+m}{m}))} \\
& X(k^+) \times (X(l^+) \times X(m^+)) \ar[d]_{\phi(\ud{k}) \times (\phi(\ud{l}) \times \phi(\ud{m}))} \ar[r]^\alpha & (X(k^+) \times X(l^+)) \times X(m^+) \ar[d]^{(\phi(\ud{k}) \times \phi(\ud{l}))\times \phi(\ud{m})} \\
& C \times (C \times C) \ar[d]_{-\otimes(-\otimes-)} \ar[r]^\alpha & (C \times C) \times C \ar[d]^{(-\otimes-) \otimes-)} \ar@{=>}[ld]^{\alpha^C} \\
& {C \ } \ar@{=}[r] & C
}
\end{equation*}
 where $\alpha^C$ is the associator of the symmetric monoidal category $C$, the arrow $F_3 = (X(\partition{k+l}{k}), X(\partition{k+l}{l})) \times id$, the arrow
 $F_1 = (X(\partition{k+l + m}{k+l}), X(\partition{k+l+m}{m}))$ and the arrow $F_2 = (X(\partition{k+l + m}{k}), X(\partition{k+l+m}{l+m}))$. We observe that the to and right vertical composite arrows ate just the functor
 $(\phi(\ud{k}) \odot \phi(\ud{l})) \odot \phi(\ud{m})$
 and the diagonal and left vertical arrow are just the functor $\phi(\ud{k}) \odot (\phi(\ud{l}) \odot \phi(\ud{m}))$.

\item \label{naturality}
 For each pair of maps $f:\ud{k} \to \ud{p}$, $g:\ud{l} \to \ud{q}$ in $\N$, the following diagram should commute
 \begin{equation*}
  \xymatrix@C=20mm{
 \phi(\ud{k+l}) \ar@{=>}[r]^{\sigma(\ud{k}, \ud{l})}  \ar@{=>}[d]_{\alpha(f+g)}  &\phi(\ud{k}) \odot \phi(\ud{l})
 \ar@{=>}[d]^{\alpha(f) \odot \alpha(g)}\\
  \phi(\ud{p + q}) \circ X(f+g) \ar@{=>}[r]_{\sigma(\ud{p}, \ud{q}) \circ id_{X(f+g)} \ \ \ \ }   &(\phi(\ud{p}) \circ X(f)) \odot  (\phi(\ud{q}) \circ X(g))
 }
 \end{equation*}
 where $\alpha(f) \odot \alpha(g) = id_{- \otimes -} \circ (\alpha(f) \times \alpha(g)) \circ id_{X(f+g)}$. We observe that $(\phi(\ud{p}) \odot \phi(\ud{q})) \circ X(f+g) = (\phi(\ud{p}) \circ X(f)) \odot  (\phi(\ud{q}) \circ X(g))$.

\end{enumerate}

\end{df}
Let $\Bike{\Psi}{X}{C}$ be another bicycle which is composed of the pair
$(\K, \delta)$
\begin{df}
A morphism of bicycles $F:\Phi \to \Psi$ is a morphism of pseudo cones $F:\L \to \K$ which is compatible with the additional structure of the two bicycles, \emph{i.e.} for all pairs $(\ud{k}, \ud{l}) \in Ob(\N) \times Ob(\N)$, the following diagram commutes
 \[
  \xymatrix@C=16mm{
 \phi(\ud{k+l}) \ar@{=>}[r]^{F(\ud{k+l})}  \ar@{=>}[d]_{\sigma(\ud{k}, \ud{l})}  &\psi(\ud{k+l})
 \ar@{=>}[d]^{\delta(\ud{k}, \ud{l})}\\
 \phi(k) \odot \phi(l) \ar@{=>}[r]_{F(\ud{k}) \odot F(\ud{l})}   &\psi(\ud{k}) \odot \psi(\ud{l})
 }
 \]
\end{df}
For any pair $(\ud{k}, \ud{l}) \in Ob(\N) \times Ob(\N)$, the natural transformations
$F(\ud{k})$ and $F(\ud{l})$ determine another natural tranformation
\[
F(\ud{k}) \times F(\ud{l}): \phi(\ud{k}) \times \phi(\ud{l}) \Rightarrow \psi(\ud{k}) \times \psi(\ud{l})
\]
which is defined on objects as follows:
\[
(F(\ud{k}) \times F(\ud{l}))(x, y) := (F(k)(x), F(k)(y)),
\]
where $(x, y) \in Ob(X(k^+)) \times Ob(X(l^+))$. It is defined similarly
on morphisms of the product category $X(k^+) \times X(l^+)$.
The natural transformation $F(\ud{k}) \odot F(\ud{l})$ in the diagram above is
defined by the following composite
\begin{equation}
  \xymatrix@C=12mm@R=3mm{
 X((k+l)^+) \ar[r]^{L}  &X(k^+) \times X(l^+) \ar[rr]^{ \ \ \ \ \ \phi(\ud{k}) \times \phi(\ud{l})} 
  \ar@/_{8mm}/[rr]_{\psi(\ud{k}) \times \psi(\ud{l})}& \ar@{=>}[d]_{F(\ud{k}) \times F(\ud{l})} &C \times C  \ar[r]^{- \underset{C}\otimes -}  &C \\
&  &&
 }
\end{equation}
where $L = (X(\delta^{k+l}_k), X(\delta^{k+l}_l))$. Composition of morphisms
of bicycles is done by treating them as morphisms of  pseudo cones.
We will use the following lemma to show that the composition of two
composable morphisms of bicycles is always a morphism of bicycles.
\begin{lem}
\label{odot-comp-comm}
Let $F:\Phi \to \Psi$ and $G:\Psi \to \Upsilon$ be two morphisms
of bicycles, then for all pairs $(\ud{k}, \ud{l}) \in Ob(\N) \times Ob(\N)$
\[
 (G(\ud{k}) \odot G(\ud{l})) \cdot (F(\ud{k}) \odot F(l)) = (G(\ud{k}) \cdot F(\ud{k})) \odot (G(\ud{l}) \cdot F(\ud{l})).
\]
\end{lem}
\begin{proof}
The proof of the above lemma follows from the interchange law of
compositions of natural transformations.
\end{proof}
\begin{coro}
\label{comp-mor-bicycles}
Let $F:\Phi \to \Psi$ and $G:\Psi \to \Upsilon$ be two morphisms
of bicycles, then their composite $G \circ F$ is a morphism of bicycles.
\end{coro}
\begin{proof}
We know that $G \circ F$ is a morphism of pseudo cones. All that we have to verify that the following diagram commutes
\[
  \xymatrix@C=22mm{
 \phi(k+l) \ar@{=>}[r]^{G \cdot F(k+l)}  \ar@{=>}[d]_{\sigma(k,l)}  &\upsilon(k+l)
 \ar@{=>}[d]^{\delta(k,l)}\\
 \phi(k) \odot \phi(l) \ar@{=>}[r]_{(G \cdot F(k)) \odot (G \cdot F(l))}   &\upsilon(k) \odot \upsilon(l)
 }
 \]
 Since $F$ and $G$ are morphisms of bicycles, therefore the following
 equality always holds
 \[
 \delta(k,l) \cdot (G \cdot F(k+l)) = ((G(k) \odot G(l)) \cdot (F(k) \odot F(l))) \cdot \sigma(k,l).
 \]
 The lemma \ref{odot-comp-comm} tells us that
 \[
 (G(k) \odot G(l)) \cdot (F(k) \odot F(l)) = (G(k) \cdot F(k)) \odot (G(l) \cdot F(l)).
\]
Thus the above diagram commutes.
\end{proof}
\begin{prop}
\label{functorial-pseudo-bike}
The definition of the category of all pseudo $\bikes$ from $X$ to $C$ determines a bifunctor
\[
\PsBikes{-}{-}:\gCAT^{op} \times \PCat \to \Cat.
\]
\end{prop}

\begin{df}
\label{strict-bicycle}
A \emph{strict bicycle} $(\L, \sigma)$ is a bicycle such that $\L$ is
a strict cone and all natural transformations in the collection
$\sigma$ are natural isomorphisms.
\end{df}
Strict $\bikes$ from $X$ to $C$ constitute a full subcategory of the category of pseudo bicycles
$\Bikes{X}{C}$, which we denote by $\StrBikes{X}{C}$.
The definition of the category of all strict bicycles from $X$ to $C$ is functorial
in both variables.
\begin{prop}
\label{functorial-strict-bike}
The definition of the category of all pseudo $\bikes$ from $X$ to $C$ determines a bifunctor
\[
\StrBikes{-}{-}:\gCAT^{op} \times \PCat \to \Cat.
\]
\end{prop}

\section[Bicycles as oplax sections]{Bicycles as oplax sections}
\label{bikes-as-oplax-sections}
In this appendix we want to describe a (pseudo) bicycle as an oplax symmetric monoidal functor from the category $\N$.
 We will construct a symmetric monoidal category $\OplaxExp{X}{C}$. The objects
 of this category are all pairs $(\underline{n}, \phi)$ where $n \in Ob(\N)$
 and $\phi:X(n^+) \to C$ is a functor. 
  A map from $(\underline{n}, \phi)$ to $(\underline{m}, \psi)$
 in $\OplaxExp{X}{C}$ is a pair $(f, \eta)$ where $f:n \to m$ is a map in the category
 $\N$ and $\eta:\phi(\ud{n}) \Rightarrow \psi(\ud{m}) \circ X(f)$ is a
 natural isomorphism. Let $(g,\beta):(\underline{m}, \psi) \to (\ud{k}, \alpha)$
 be another map in $\OplaxExp{X}{C}$, then we define their composition
 as follows:
 \[
 (g,\beta) \circ (f, \eta) := (g \circ f, \beta \ast \eta),
 \]
 where $\beta \ast \eta$ is the composite natural transformation $ (\beta \circ X(f)) \cdot \eta$ in which $\phi \circ X(f)$ is the \emph{horizontal} composition of
 the natural transformations $\phi(g)$ and $id_{X(f)}$ and
 $(\beta \circ X(f)) \cdot \eta$ is the \emph{vertical} composition
 of the two natural transformations. Using the \emph{interchange law}
 and the associativity of compositions we will now show that the
 composition defined above is associative.
 \begin{prop}
 \label{assoc-comp-B}
 The composition law for the category $\OplaxExp{X}{C}$, as defined above, is
 strictly associative.
 \end{prop}
 \begin{proof}
 Let $(f, \eta(f)):(\underline{n}, \phi)$ to $(\underline{m}, \psi)$, $(g,\beta(g)):(\underline{m}, \psi) \to (\ud{k}, \alpha)$ and $(h,\theta(h)):(\ud{k}, \alpha) \to (\ud{j}, \delta)$ be three composable
 morphisms in $\OplaxExp{X}{C}$. We want to show that
 \[
 ((h,\theta) \circ (g,\beta) \circ (f, \eta = (h,\theta) \circ ((g,\beta \circ (f, \eta)).
 \]
 In order to do so it would be sufficient to verify the associativity of the operation $\ast$, \emph{i.e}, to verify $(\theta \ast \beta) \ast \eta)  = \theta \ast (\beta \ast \eta)$. The situation is depicted in the following diagram
 \[
\label{comp-B}
  \xymatrix{
 X(n^+) \ar[rrrr]^{\phi} \ar@{=}[d]   && \ar@{}[d]|{\Downarrow \eta} && C \ar@{=}[d] \\
 X(n^+) \ar[r]_{X(f)} \ar@{=}[d] \ar@{}[d]|{ \ \ \ \ \ \ \ \ \ \ \ \ \ \ \ \ \ \ \ \ \ \ \ \ \Downarrow id_{X(f)}} &X(m^+) 
\ar[rrr]_{\psi}  \ar@{=}[d]   & \ar@{}[d]|{\Downarrow \beta} && C \ar@{=}[d] \\
 X(n^+) \ar[r]_{X(f)} \ar@{=}[d] \ar@{}[d]|{ \ \ \ \ \ \ \ \ \ \ \ \ \ \ \ \ \ \ \ \ \ \ \ \ \Downarrow id_{X(f)}}  &X(m^+)
 \ar[r]_{X(g)} \ar@{=}[d] \ar@{}[d]|{ \ \ \ \ \ \ \ \ \ \ \ \ \ \ \ \ \ \ \ \ \ \ \ \ \Downarrow id_{X(g)}}  &X(k^+) \ar[rr]_{\alpha} \ar@{=}[d]   &
 \ar@{}[d]|{\Downarrow \theta} & C \ar@{=}[d] \\
 X(n^+) \ar[r]_{X(f)}  &X(m^+) \ar[r]_{X(g)}  &X(k^+) \ar[r]_{X(h)} &X(j^+) \ar[r]_{\delta}   & C
 }
\]
 We begin by considering the left hand side, namely
\[
\theta \ast (\beta \ast \eta) = (\theta \circ id_{X(g)} \circ id_{X(f)}) \cdot ((\beta \circ id_{X(f)}) \cdot \eta),
\]
where $\cdot$ represents vertical composition of natural transformations which
is an associative operation. Therefore by rearranging we get
\begin{multline*}
\theta \ast (\beta \ast \eta) = (\theta \circ id_{X(g)} \circ id_{X(f)}) \cdot ((\beta \circ id_{X(f)}) \cdot \eta) \\
((\theta \circ id_{X(g)} \circ id_{X(f)}) \cdot (\beta \circ id_{X(f)})) \cdot \eta.
\end{multline*}
Now the interchange law says that the vertical composite
$((\theta \circ id_{X(g)} \circ id_{X(f)}) \cdot (\beta \circ id_{X(f)}))$
is the same as $((\theta \circ id_{X(g)}) \cdot \beta) \circ (id_{X(f)} \cdot id_{X(f)})) \cdot \eta$
which is the same as $(\theta \ast \beta) \ast \eta$.
 \end{proof}
 
  Thus we have defined
 the category $\OplaxExp{X}{C}$. Next we want to define a symmetric monoidal
 structure on the category $\OplaxExp{X}{C}$. Let $(\underline{n}, \phi)$ and $(\underline{m}, \psi)$
 be two objects of $\OplaxExp{X}{C}$, we define
 \[
 (\underline{n}, \phi) \otimes (\underline{m}, \psi) := (\ud{n} + \ud{m}, \phi \odot \psi),
 \]
 where the second component on the right is defined as
 the following composite
 \[
 X((n+m)^+) \overset{X(\delta^{n+m}_n) \times X(\delta^{n+m}_m)} \to X(n^+) \times X(m^+) \overset{\phi \times \psi} \to C \times C \overset{-\underset{C}\otimes-} \to C
 \]
 Let $(f,\eta):(\underline{n}, \phi) \to (\ud{k}, \delta)$ and $(g,\beta):(\underline{m}, \psi) \to (l, \alpha)$ be two maps in $\OplaxExp{X}{C}$, then we define
 \[
 (f,\eta) \otimes (g,\beta) := (f+g, \eta \odot \beta),
 \]
 where $f+g:\ud{n}+\ud{m} \to \ud{k} + \ud{l}$ is the map determined by the symmetric
 monoidal structure on $\N$ and the natural transformation
 $\eta \odot \beta$ is defined to be the following composite:
 \begin{equation*}
 \label{tens-prod-arrows-elem-exp}
 id_{- \underset{C} \otimes -} \circ (\eta \times \beta) \circ id_{X(\partition{n+m}{n}) \times
 X(\partition{n+m}{m})}.
 \end{equation*}
 In other words for any $x \in X((n+m)^+)$
 \[
 (\eta \odot \beta)(x) := \eta(X(\delta^{n+m}_n)(x)) \underset{C}\otimes
 \beta(X(\delta^{n+m}_m)(x)),
 \]
 where $x \in X((n+m)^+)$. It is easy to see that this defines
 a natural transformation between the functors
 \[
 \phi \odot \psi:X((n+m)^+) \to C
 \]
 and the following composite functor
 \begin{eqnarray*}
 X((n+m)^+) \overset{X(\delta^{n+m}_n) \times X(\delta^{n+m}_m)}   \to X(n^+) \times X(m^+) \overset{X(f) \times X(g)} \to X(k^+) \times X(l^+) \\
 \overset{\phi \times \psi} \to C \times C \overset{-\underset{C}\otimes-} \to C.
\end{eqnarray*}
We observe that for any two maps $f:\ud{n} \to \ud{k}$ and $g:\ud{m} \to \ud{l}$
in the category $\N$, the following diagram
 \[
\label{splitting-sum-maps-gCat}
  \xymatrix@C=28mm{
 X((n+m)^+)  \ar[r]^{(X(\delta^{n+m}_n), X(\delta^{n+m}_m)) } \ar[d]_{X(f+g)}  & X(n^+) \times X(m^+) \ar[d]^{X(f) \times X(g)}\\
  X((k+l)^+) \ar[r]_{(X(\delta^{k+l}_k), X(\delta^{k+l}_l))} &X(k^+) \times X(l^+)
 }
\]
This shows that $\eta \odot \beta$ is a natural transformation between
the functors $\phi \odot \psi$ and $(\phi \odot \psi) \circ X(f+g)$.
 \begin{prop}
 \label{sym-mon-B(X;C)}
 The category $\OplaxExp{X}{C}$ is a symmetric monoidal category.
 \end{prop}
 \begin{proof}
 
 The unit object of $\OplaxExp{X}{C}$ is the pair $(0, \phi(0))$, where
 $\phi(0):X(0^+) \to C$ is the constant functor assigning the value
 $\unit{C}$.
 We begin by verifying that the tensor product defined
 above defines a bifunctor
 \[
  - \otimes -: \OplaxExp{X}{C} \times \OplaxExp{X}{C} \to \OplaxExp{X}{C}.
 \]
 \begin{sloppypar}
 Let $((f, \eta), (g, \beta)):((\ud{k}, \phi),(\ud{l}, \phi)) \to ((\ud{m}, \phi), (\underline{n}, \phi))$ and $((p, \delta) (q, \theta)):((\ud{m}, \phi(m)), (\underline{n}, \phi)) \to ((a, \phi(a)),(b, \phi(b)))$ be a pair of
 composable arrows in the product category $\OplaxExp{X}{C} \times \OplaxExp{X}{C}$.
 We will show that
 \end{sloppypar}
 \[
 (p \circ f, \eta(f) \ast \delta(p)) \otimes (q \circ g, \beta(f) \ast \theta(p)) =
 ((p, \delta(p)) \otimes (q, \theta(q))) \circ ((f, \eta(f)) \otimes (g, \beta(g))).
 \]
 Throughout this proof we will refer to the following commutative diagram
 \begin{equation*}
 \xymatrix@C=23mm@R=12mm{
 X((k+l)^+) \ar[r]^{(X(\partition{k+l}{k}), X(\partition{k+l}{l})) \ \ \ } \ar[d]_{X(f+g)} & X(k^+) \times X(l^+) \ar[r]^{\phi(k) \times \phi(l)} \ar[d]_{X(f) \times X(g)} & C \times C \ar@{=>}[ld]^{\ \ \ \ \eta(f) \times \beta(g)}  \ar@{=}[d]\\
 X((m+n)^+) \ar[r]^{(X(\partition{m+n}{m}), X(\partition{m+n}{n})) \ \ } \ar[d]_{X(p+q)} & X(m^+) \times X(n^+) \ar[r]_{\phi(m) \times \phi(n)} \ar[d]_{X(p) \times X(q)} & C \times C \ar@{=>}[ld]^{\ \ \ \ \delta(p) \times \theta(q)} \ar@{=}[d]  \\
 X((a+b)^+) \ar[r]_{(X(\partition{a+b}{a}), X(\partition{a+b}{b})) \ \ \ } & X(a^+) \times X(b^+) \ar[r]_{\phi(a) \times \phi(b)}  & C \times C  \\
 }
 \end{equation*}
 Since the addition operation, $+$, is the symmetric monoidal structure on $\N$, therefore
 $p \circ f + q \circ g = (p + q) \circ (f + g)$.
 We recall that
 \[
 (p \circ f, \eta(f) \ast \delta(p)) \otimes (q \circ g, \beta(f) \ast \theta(p)) =
 (p \circ f + q \circ g, (\eta \ast \delta) \odot (\beta \ast \theta)(p \circ f + q \circ g)).
 \]
 By definition, the natural transformation $(\eta \ast \delta) \odot (\beta \ast \theta)(p \circ f + q \circ g)$ is the following composite:
 \begin{multline*}
 id_{- \otimes -} \circ (((\delta(p) \circ id_{X(f)}) \cdot \eta(f)) \times (((\theta(q) \circ id_{X(g)}) \cdot \beta(f))) \circ id_{X(\partition{k+l}{k}) \times X(\partition{k+l}{l})}.
 \end{multline*}
 We observe that the above composite is the same as the following composite:
 \begin{multline*}
 id_{- \otimes -} \circ ((\delta(p) \times \theta(q)) \circ (id_{X(f)} \times id_{X(g)}) \cdot (\eta(f) \times \beta(g))) \circ id_{(X(\partition{k+l}{k}), X(\partition{k+l}{l}))}.
 \end{multline*}
 The composite natural transformation $((p, \delta) \otimes (q, \theta)) \circ ((f, \eta) \otimes (g, \beta(g)))$ is, by definition, the same as $(\theta \odot \delta \circ id_{f+g}) \cdot (\eta \odot \beta)$. Unwinding definitions gives us the following equality
 \begin{multline*}
 (\theta \odot \delta \circ id_{f+g}) \cdot (\eta \odot \beta(f+g)) = \\
 (id_{- \otimes -} \circ (\delta \times \theta) \circ id_{(X(\partition{m+n}{m}), X(\partition{m+n}{n}))} \circ id_{X(f + g)}) \\ \cdot ( id_{- \otimes -} \circ (\eta \times \beta) \circ id_{(X(\partition{k+l}{k}), X(\partition{k+l}{l}))}).
 \end{multline*}
 The above diagram tells us that $(X(\partition{m+n}{m}), X(\partition{m+n}{n})) \circ X(f + g) = (X(f) \times X(g)) \circ (X(\partition{k+l}{k}), X(\partition{k+l}{l}))$.
 Now the interchange law of composition of natural transformations gives the
 following equalities
 \begin{multline*}
 (\theta \odot \delta \circ id_{f+g}) \cdot (\eta \odot \beta(f+g)) = \\
 (id_{- \otimes -} \circ (\delta \times \theta) \circ (id_{X(f)} \times  id_{X(g)}) \circ id_{(X(\partition{k+l}{k}), X(\partition{k+l}{l}))} \\ \cdot ( id_{- \otimes -} \circ (\eta \times \beta) \circ id_{(X(\partition{k+l}{k}), X(\partition{k+l}{l}))}) = \\
 id_{- \otimes -} \circ ((\delta \times \theta) \circ (id_{X(f)} \times id_{X(g)}) \cdot (\eta \times \beta)) \circ id_{(X(\partition{k+l}{k}), X(\partition{k+l}{l}))}.
 \end{multline*}
 \dots
 \end{proof}
 We will refer to the category $\OplaxSec{X}{C}$ as the \emph{category of elements of the exponential from $X$ to $C$}.
 
 \begin{prop}
 \label{cat-oplax-sections-bifunctor}
 The construction of the category of elements of the exponential described above
 defines a bifunctor
 \begin{equation}
 \OplaxExp{-}{-}:\gCAT^{op} \times \PCat \to \PCat.
 \end{equation}
 \end{prop}

 The category $\OplaxExp{X}{C}$ has an associated projection functor
 $pr_\N: \OplaxExp{X}{C} \to \N$ which projects the first coordinate. Now we are ready
 to define a $\bike$
 \begin{df}
 \label{oplax-section-def}
 A \emph{oplax symmetric monoidal section of $\OplaxExp{X}{C}$}
 is a unital oplax symmetric monoidal functor  $\Phi:\N \to \OplaxExp{X}{C}$ such that
 $pr_N \circ \Phi = id_{\N}$. A \emph{morphism of oplax symmetric monoidal sections of $\OplaxExp{X}{C}$} is an oplax natural transformation between two oplax symmetric monoidal section of $\OplaxExp{X}{C}$.
 \end{df}
 We will denote the category of all oplax symmetric monoidal section
 of $\OplaxExp{X}{C}$ by $\OplaxSec{X}{C}$. 
  \begin{prop}
 \label{Bikes-as-sections}
 The category of all oplax symmetric monoidal section 
 of $\OplaxExp{X}{C}$ is isomorphic to the category of all bicycles from $X$ to $C$.
 We begin by defining $I$. 
 \end{prop}
 \begin{proof}
 We will define a pair of functors $I:\OplaxSec{X}{C} \to \Bikes{X}{C}$ and
 $J:\Bikes{X}{C} \to \OplaxSec{X}{C}$ and show that they are inverses of one another.
 For an
 object $\Phi \in Ob(\OplaxSec{X}{C})$ we define the bicycle $I(\Phi)$ to be the
 pair $(\L_\Phi, \sigma_{\Phi})$ where $\L_\Phi$ is a pair $(\phi, \alpha_\Phi)$ consisting
 of a collection of functors $\phi$ which is composed of a functor $\phi(n):X(n) \to C$, for each
 $n \in Ob(\N)$, which is defined as follows:
 \[
 \phi(n) := \Phi(n).
 \] 
 and a  collection of natural transformations $\alpha_\Phi$ consisting of one natural transformation $\alpha_\Phi(f)$ for each $f \in Mor(\N)$, which is defined as follows:
 \[
 \alpha_\Phi(f) := \Phi(f).
 \]
 Finally $\sigma_\Phi$ is a collection consisting of a natural transformation $\sigma_\Phi(k,l)$, for each pair of objects $(k, l) \in Ob(\N) \times Ob(\N)$ which is defined as follows:
 \[
 \sigma_\Phi(k,l) := \lambda_\Phi(k, l),
 \]
 where $\lambda_\Phi$ is the natural transformation providing the oplax structure
 to the functor $F$. The pair $\L_\phi = (\phi, \alpha_\Phi)$ is a normalized lax cone because $\Phi$ is a functor from $\N$ to $\OplaxExp{X}{C}$. The conditions in the definition of a bicycle, namely \ref{unit},
 \ref{symmetry}, \ref{associativity} follow from the oplax structure on $\Phi$.
 Thus we have defined a bicycle $I(\Phi)$.
 A morphism $F:\Phi \to \Theta$ in $\OplaxSec{X}{C}$ determines a collection
 of natural transformations $C_F$ consisting of a natural transformation $F(n)$
 for each $n \in Ob(\N)$. This collection defines a morphism of bicycles because
 $F$ is an oplax symmetric monoidal functor.
 
 Now we define the functor $J$. Let $\Bike{\Psi}{X}{C}$ be a bicycle from $X$ to $C$
 which is represented by a pair $(\L_\Psi, \sigma_\Psi)$ and whose underlying lax monoidal cone is given by a pair $\L_\Psi = (\psi, \alpha_\Psi)$. We define a oplax symmetric monoidal section of $\OplaxExp{X}{C}$, $\Phi$, as follows:
 \[
 \Phi(n) := \psi(n), \ \ \ \ \ \ \text{and} \ \ \ \ \ \ \Phi(f) := \alpha_\Psi(f).
  \]
  This defines a functor $\Phi$ which is given the oplax symmetric monoidal structure
  by a natural transformation $\lambda_\Phi:\Phi \circ (- \underset{\N} \otimes -) \Rightarrow (- \underset{\OplaxExp{X}{C}} \otimes -) \circ (\Phi \times \Phi)$ which is defines as follows:
  \[
  \lambda_\Phi(k, l) := \sigma_\Psi(k, l).
  \]
 \end{proof}
 
 Along the lines of the symmetric monoidal category $\OplaxExp{X}{C}$,
 we want to define another symmetric monoidal category $\SMExp{X}{C}$ for every
 pair $(X,C) \in Ob(\gop) \times Ob(\PCat)$. The objects of $\SMExp{X}{C}$
 are all pairs $(\vec{n}, \phi(\vec{n}))$, where $\vec{n} \in Ob(\Leins)$
 and $\phi(\vec{n}):\Leins(X)(\vec{n}) \to C$ is a basepoint preserving functor.
 A map from $(\vec{n}, \phi(\vec{n}))$ to $(\vec{m}, \psi(\vec{n}))$
 in $\SMExp{X}{C}$ is a pair $(f, \eta(f))$ where $f:\vec{n} \to \vec{m}$
 is a map in the category
 $\Leins$ and $\eta(f):\phi(\vec{n}) \Rightarrow \psi(\vec{m}) \circ \Leins(X)(f)$ is a
 natural transformation. Let $(g,\beta(g)):(\vec{m}, \psi(\vec{m})) \to (\vec{k}, \alpha(\vec{k}))$
 be another map in $\SMExp{X}{C}$, then we define their composition
 as follows:
 \[
 (g,\beta(g)) \circ (f, \eta(f)) := (g \circ f, \beta(g) \ast \eta(f)),
 \]
 where $\beta(g) \ast \eta(f)$ is the composite natural transformation
 $ (\beta(g) \circ id_{X(f)}) \cdot \eta(f)$ in which $\phi(g) \circ id_{X(f)}$
 is the \emph{horizontal} composition of
 the natural transformations $\beta(g)$ and $id_{X(f)}$ and
 $(\beta(g) \circ X(f)) \cdot \eta(f)$ is the \emph{vertical} composition
 of the two natural transformations. Using the \emph{interchange law}
 and the associativity of compositions, an argument similar to \ref{assoc-comp-B}
 can be written which proves that the composition defined above is associative.
 The category $\SMExp{X}{C}$ is a
 symmetric monoidal category with the symmetric monoidal structure
 being an extension of the symmetric monoidal structure of $\OplaxExp{X}{C}$.
 Let $(\vec{n}, \phi(\vec{n}))$ and $(\vec{m}, \psi(\vec{m}))$
 be two objects of $\SMExp{X}{C}$, we define
 \[
 (\vec{n}, \phi(\vec{n})) \otimes (\vec{m}, \psi(\vec{m})) := (\vec{n} \Box \vec{m}, \phi(\vec{n}) \boxdot \psi(\vec{m})),
 \]
 where the second component on the right is defined as
 the following composite
 \begin{equation}
 \label{boxdot-def}
 \Leins(X)(\vec{n} \Box \vec{m}) \overset{\lambda_{\Leins(X)}(n, m)} \to \Leins(X)(\vec{n}) \times \Leins(X)(\vec{m}) \overset{\phi(\vec{n}) \times \psi(\vec{m})} \to C \times C \overset{-\underset{C}\otimes-} \to C,
 \end{equation}
 where $\lambda_{\Leins(X)}(n, m)$ is the map given by the pseudo-functor structure
 of $\Leins(X)$.
 Let $(f,\eta(f)):(\vec{n}, \phi(\vec{n})) \to (\vec{k}, \delta(k))$ and $(g,\beta(g)):(\vec{m}, \psi(\vec{m})) \to (\vec{l}, \alpha(\vec{l}))$ be two maps in $\SMExp{X}{C}$, then we define
 \[
 (f,\eta(f)) \otimes (g,\beta(g)) := (f \Box g, \eta \boxdot \beta(f \Box g)),
 \]
 where $f \Box g:n \Box m \to k \Box l$ is the map determined by the symmetric
 monoidal structure on $\Leins$ and the natural transformation
 $\eta \boxdot \beta(f \Box g)$ is defined as follows
 \[
 (\eta \boxdot \beta(f \Box g)) := (id_{-\underset{C}\otimes-}) \circ (\eta(f) \times \beta(g)),
 \]
 where $\eta(f) \times \beta(g):\phi(\vec{n}) \times \psi(\vec{m}) \Rightarrow \delta(\vec{k}) \times \alpha(\vec{l})$ is the product of $\eta(f)$ and $\beta(f)$.  An argument similar to
 Proposition \ref{sym-mon-B(X;C)} shows that $\SMExp{X}{C}$ is a symmetric monoidal category
 which is permutative if $C$ is permutative.
 We will refer to the category $\SMSec{X}{C}$ as the \emph{symmetric monoidal completion of the category of elements of the exponential from $X$ to $C$}.
 
 \begin{prop}
 \label{cat-sections-bifunctor}
 The symmetric monoidal completion of the category of elements of the exponential described above
 defines a bifunctor
 \begin{equation}
 \SMExp{-}{-}:\gCAT^{op} \times \PCat \to \PCat.
 \end{equation}
 \end{prop}

 The category $\SMExp{X}{C}$ has an associated projection functor
 $pr_\Leins: \SMExp{X}{C} \to \Leins$ which to projects the first coordinate.
 \begin{df}
 \label{SM-section-def}
 A \emph{strict symmetric monoidal section of $\SMExp{X}{C}$}
 is a strict symmetric monoidal functor  $\Phi:\N \to \SMExp{X}{C}$ such that
 $pr_A \circ \Phi = id_{\Leins}$. A \emph{morphism of strict symmetric monoidal sections of $\SMExp{X}{C}$} is a symmetric monoidal natural transformation between two strict symmetric monoidal section of $\OplaxExp{X}{C}$.
 \end{df} 

 We will denote the (pointed) category of all strict symmetric monoidal sections of
 $\SMExp{X}{C}$ by $\SMSec{X}{C}$.
 There is an obvious inclusion
 functor $\I:\OplaxExp{X}{C} \hookrightarrow \SMExp{X}{C}$ which is defined on objects as follows:
 \[
  (\underline{n}, \phi) \mapsto ((n), \phi((n))).
 \]
 \begin{prop}
 \label{oplax-inclusion}
 The inclusion functor $\I:\OplaxExp{X}{C} \hookrightarrow \SMExp{X}{C}$
 is a unital oplax symmetric monoidal functor.
 \end{prop}
 \begin{proof}
 Let $(k, \phi(k))$ and $(l, \psi(l))$ be two objects in the category $\OplaxExp{X}{C}$.
 Then
  \[
  \I((k + l, \phi(k) \odot \psi(l))) = ((k+l), \phi(k) \odot \psi(l)).
 \]
 There is a partition map $p_{k,l}:(k+l) \to (k, l)$ in $\Leins$ which makes the following
 diagram commutative:
 \begin{equation*}
 \xymatrix{
 \Leins(X)((k+l)) \ar[rr]^{\phi(k) \odot \psi(l)} \ar[rd]_{\Leins(X)(p_{k,l})} && C \\
 & \Leins(X)((k, l)) \ar[ru]_{\phi((k)) \boxdot \psi((l))}
 }
 \end{equation*}
 This diagram implies that the partition map $p_{k,l}$ defines a map
 \[
 (p_{k,l}, id):((k+l), \phi((k)) \odot \psi((l)))) \to ((k,l), \phi((k)) \boxdot \psi((l))))
 \]
 in $\SMExp{X}{C}$. We denote this map by $\lambda_\I((k, \phi(k)),(l, \psi(l)))$.
 We observe that $\I(0, \phi(0)) = ((),\phi(()))$. Thus $\I$ strictly preserves the unit.
 Now we need to check the unit, symmetry and associativity conditions, we begin
 by checking the symmetry condition. We observe that the following diagram commutes
 \begin{equation*}
\label{extn-bicycle-A}
  \xymatrix{
 (k+l, \phi(k) \odot \psi(l)) \ar[d]_{\lambda_\I((k, \phi(k)),(l, \psi(l)))} \ar[r]^{\I(\gamma)}  & (l+k, \phi(l) \odot \psi(k)) \ar[d]^{\lambda_\I((l, \phi(l)),(k, \psi(k)))}\\
 ((k,l), \phi((k)) \boxdot \psi((l))) \ar[r]_{\gamma} & ((l,k), \phi((l)) \boxdot \psi((k)))
 }
 \end{equation*}
 because $\gamma^{\OplaxExp{X}{C}}_{(k, \phi(k)), (l, \psi(l))} = \gamma^{\SMExp{X}{C}}_{((k), \phi((k))), ((l), \psi((l)))} \circ id_{\Leins(X)(p_{k,l})}$. This equality follows from the following commutative diagram:
 \begin{equation*}
 \xymatrix@C=20mm{
 \Leins(X)((k+l)) \ar[d]_{\Leins(X)(\gamma^{\N}_{k,l})}  \ar[r]^{\Leins(X)(p_{k,l})} &\Leins(X)((k, l))
 \ar[rr]^{\phi((k)) \boxdot \psi((l))} \ar[d]_{\Leins(X)(\gamma^{\Leins}_{(k),(l)})} & \ar@{=>}[d]^{G}& C \ar@{=}[d] \\
 \Leins(X)((l+k)) \ar[r]_{\Leins(X)(p_{l,k})} & \Leins(X)((l,k)) \ar[rr]_{ \ \ \ \ \phi((l)) \boxdot \psi((k))} &&C
 }
 \end{equation*}
 where $(\gamma^{\Leins}_{(k), (l)},G) = \gamma^{\SMExp{X}{C}}_{((k), \phi((k))), ((l), \psi((l)))}$. A similar argument shows that the pair $(\I, \lambda_\I)$ satisfies the
 associativity condition \ref{OpL-associativity}. Thus we have proved that $\I$ is a unital
 oplax symmetric monoidal functor.
 \end{proof}

 \begin{thm}
  \label{oplax-SM-cones-isom}
  The category $\OplaxSec{X}{C}$ is isomorphic to the category
  $\SMSec{X}{C}$ for every pair $(X, C) \in Ob(\gCAT) \times Ob(\PCat)$.
 \end{thm}
 \begin{proof}
  We will define a functor $E:\OplaxSec{X}{C} \to \SMSec{X}{C}$ which is the
  inverse of the functor $i_\N^\ast:\SMSec{X}{C} \to \OplaxSec{X}{C}$. Let $\Phi$
  be a oplax symmetric monoidal section of $\OplaxSec{X}{C}$, then composition with $\I$ gives us a unital oplax symmetric monoidal functor $\I \circ \Phi: \N \to \SMSec{X}{C}$. Now proposition \ref{Ext-gCat} and the isomorphism of categories $\OLSMHom{\N}{\SMSec{X}{C}} \cong [\gop, \SMSec{X}{C}]$ tells us that $\I \circ \Phi$ uniquely extends to a strict symmetric monoidal functor $\Leins(\I \circ \Phi)$ along the inclusion map $i:\N \to \Leins$. Moreover this functor is a strict symmetric monoidal section of $\SMExp{X}{C}$. We define
  \begin{equation*}
 E(\Phi) :=  \Leins(\I \circ \Phi).
 \end{equation*}
 The uniqueness of the extension implies that the object function of the functor $E$ is a bijection. A morphism $F:\Phi \to \Psi$ in 
 $\OplaxSec{X}{C}$ can be seen as an oplax symmetric monoidal functor $F: \N \to [I;\OplaxExp{X}{C}]$, where $I$ is the category having two objects $0$ and $1$ and exactly one non-identity morphism $0 \to 1$, such that the following two diagrams commute
 \begin{equation*}
 \label{Defn-OLSec-Mor}
 \xymatrix{
 \N \ar[r]^{F \ \ \ \ } \ar[rd]_{\Phi} & [I;\OplaxExp{X}{C}] \ar[d]^{[i_0; \OplaxExp{X}{C}]}  & & \N \ar[r]^{F \ \ \ \ } \ar[rd]_{\Phi} & [I;\OplaxExp{X}{C}] \ar[d]^{[i_1; \OplaxExp{X}{C}]}\\
 & \OplaxExp{X}{C} & & & \OplaxExp{X}{C}
 }
 \end{equation*}
 where $i_0:0 \to I$ and $i_1:1 \to I$ are the inclusion functors.
 We recall that the codomain functor category inherits a strict symmetric monoidal (permutative) structure from $\OplaxExp{X}{C}$. We can compose this functor with the oplax symmetric monoidal functor $[I, \I]$ to obtain a composite functor
 \[
 \N \overset{F} \to [I;\OplaxExp{X}{C}] \overset{[I;\I]} \to [I;\SMExp{X}{C}]
 \]
 This composite oplax symmetric monoidal functor extends uniquely to a strict symmetric monoidal functor
 \[
 \Leins([I;\I] \circ F): \Leins \to [I;\SMExp{X}{C}]
 \]
 along the inclusion map $i:\N \to \Leins$. This extended strict symmetric monoidal functor can be seen as a morphism in the category $\SMSec{X}{C}$. We define
 \[
 E(F) :=  \Leins([I;\I] \circ F).
 \]
 One can check that $E(F \circ G) = E(F) \circ E(G)$.
 The uniqueness of the extension of $[I;\I] \circ F$ to
 $\Leins([I;\I] \circ F)$ implies that the functor $E$ is fully faithful.
 Thus we have proved that the functor $E$ is an isomorphism of categories.
 \end{proof}

 \section[The adjunction $\PNat \dashv \Kbar$]{The adjunction $\PNat \dashv \Kbar$}
\label{Sec:AdjunctionPR}
\begin{sloppypar}
In this Appendix we will establish an adjunction $\PNat \dashv \Kbar$, where
$\PNat:\gCAT \to \PCat$ is the realization functor defined in section \ref{thick-Seg-nerve} and
$\Kbar:\PCat \to \gCAT$ is the functor which is also defined in section \ref{thick-Seg-nerve}. We will
establish the desired adjunction in two steps. In the first step we show
that the mapping set $\PCat(\PNat(X), C)$ is isomorphic to the
set of all strict symmetric monoidal sections from $\Leins$
to $\OplaxExp{\Leins(X)}{C}$, namely $Ob(\SMSec{X}{C})$.
 Throughout this section $X$ will denote a $\gCat$ and
$C$ will denote a permutative category. In the second step we show that
the Hom set $\gCAT(X, \Kbar(C))$ is isomorphic to the set of oplax symmetric monoidal sections $Ob(\OplaxSec{X}{C})$.
We begin by constructing a strict symmetric monoidal functor
 $i:\Leins \to \OplaxExp{\Leins(X)}{\PNat(X)}$. For each $\vec{n} \in Ob(\Leins)$ we will define a functor
$i(\vec{n}):\Leins(X)(n^+) \to \PNat(X)$. For an $\vec{x} \in Ob(\Leins(X)(n^+))$, we define
\[
i(\vec{n})(\vec{x}) := (\vec{n},\vec{x}).
\]
\end{sloppypar}
For a morphism $a:\vec{x} \to \vec{y}$ in $\Leins(X)(n^+)$, we define
\[
i(a) := (id_{\vec{n}}, a),
\]
\begin{sloppypar}
where $(id_{\vec{n}}, a):(\vec{n},\vec{x}) \to (\vec{n},\vec{y})$ is a morphism
in $\Leins(X)$. For each morphism $(h, \phi):\vec{n} \to \vec{m}$ in $\Leins$
we will define a natural transformation $i((h, \phi):i(\vec{n}) \Rightarrow i(\vec{m}) \circ \Leins(X)((h, \phi))$.
Let $\vec{x} \in Ob(\Leins(X)(n^+))$, we define
\end{sloppypar}
\[
i((h, \phi)) (\vec{x}) := ((h, \phi), id_{\Leins(X)((h, \phi))(\vec{x})}),
\]
where $((h, \phi), id_{\Leins(X)((h, \phi))(\vec{x})}):(\vec{n},\vec{x}) \to (\vec{m},\Leins(X)((h, \phi)(\vec{x}))$ is a morphism
in $\Leins(X)$. It is easy to see that for any morphism $(h, \phi):\vec{x} \to \vec{y}$
in $\Leins(X)(n^+)$ the following diagram commutes
\[
  \xymatrix@C=16mm{
 (\vec{n},\vec{x})  \ar[r]^{i((h, \phi))(\vec{x}) \ \ \ \ \ \ \ \ \ \ \ } \ar[d]_{(id_{\vec{n}}, a)} & (\vec{m},\Leins(X)((h, \phi))(\vec{x})) \ar[d]^{(id_{\vec{m}}, \Leins(X)((h, \phi))(a))} \\
 (\vec{n},\vec{y}) \ar[r]_{i((h, \phi))(\vec{y}) \ \ \ \ \ \ \ \ \ \ \ } & (\vec{m},\Leins(X)((h, \phi))(\vec{y}))
 }
\]
in the category $\Leins(X)$.
Thus we have defined a natural transformation $i((h, \phi))$.
\begin{prop}
\label{univ-A-sec-defn}
 The collection of functors $\lbrace i(\vec{n}) \rbrace_{\vec{n} \in \N}$
 glue together to define a strict symmetric monoidal section
 of $\OplaxExp{\Leins(X)}{\PNat(X)}$.
\end{prop}
 \begin{proof}
 Clearly $pr_\Leins \circ i = id_\Leins$. Further $i(\vec{n} \Box \vec{m})$ =
 $i(\vec{n}) \boxdot i(\vec{m})$.
\end{proof}

A strict symmetric monoidal section of $\OplaxExp{\Leins(X)}{\PNat(X)}$ , $\phi: \Leins \to \OplaxExp{\Leins(X)}{\PNat(X)}$, and a strict monoidal functor $\overline{\phi}:\PNat(X) \to C$ determine another
strict symmetric monoidal section of $\OplaxExp{\Leins(X)}{C}$, namely $\OplaxExp{\Leins(X)}{\overline{\phi}}\circ \phi$.
We want to show that $i$ is a \emph{universal strict symmetric monoidal section} \emph{i.e.} for any strict symmetric monoidal section of
$\phi:\Leins \to \OplaxExp{\Leins(X)}{C}$, there exists a
unique strict symmetric monoidal functor $\overline{\phi}:\PNat(X) \to C$.
 such that $\OplaxExp{\Leins(X)}{\overline{\phi}} \circ i = \phi$.
\begin{lem}
\label{univ-bike}
The section $i$ is a universal strict symmetric monoidal section.
\end{lem}
\begin{proof}
Let $\phi:\Leins \to \OplaxExp{\Leins(X)}{C}$ be a lax symmetric monoidal section.
 We begin by constructing a strict monoidal functor
 $\overline{\phi}:\PNat(X) \to C$ such that $\OplaxExp{\Leins(X)}{\overline{\phi}} \circ i = \phi$.
 On objects of $\PNat(X)$, the functor $\overline{\phi}$ is defined as follows:
 \[
 \overline{\phi}((\vec{n}, \vec{x})) := \phi(\vec{n})(\vec{x}).
 \]
 The morphism function of $\overline{\phi}$ is defined as follows:
 \[
 \overline{\phi}(((h, \psi),a)) := \phi(\vec{m})(a) \circ \phi((h, \psi))(\vec{x}).
 \]
 One can easily check that $\overline{\phi}$ is a functor.
 
 Let $(\vec{m}, \vec{y})$ be another object in $\PNat(X)$, we consider
 \[
 \overline{\phi}((\vec{n}, \vec{x}) \underset{\PNat(X)}\otimes (\vec{m}, \vec{y}))
 = \overline{\phi}(\vec{n} \Box \vec{m}, \inv{\lambda(\vec{n}, \vec{m})}((\vec{x},\vec{y})))
 = \phi(\vec{n} \Box \vec{m})(\inv{\lambda(\vec{n}, \vec{m})}((\vec{x},\vec{y}))),
 \]
 where $(\vec{x},\vec{y})$ is the concatenation of $\vec{x}$ and
 $\vec{y}$. Since $\phi$ is a symmetric monoidal functor, therefore
 $\phi(\vec{n} \Box \vec{m}) = \phi(\vec{n}) \boxdot \phi(\vec{m})$.
 Now we observe that
 \begin{multline*}
 \phi(\vec{n} \Box \vec{m})(\inv{\lambda(\vec{n}, \vec{m})}((\vec{x},\vec{y})))
 = \phi(\vec{n}) \boxdot \phi(\vec{m})(\inv{\lambda(\vec{n}, \vec{m})}((\vec{x},\vec{y})))
 =  \\
 \phi(\vec{n})(\vec{x}) \underset{C} \otimes \phi(\vec{m})(\vec{y})
 = \overline{\phi}(\vec{n},\vec{x}) \underset{C} \otimes \overline{\phi}(\vec{m},\vec{y}).
 \end{multline*}
 where the second equality follows from \eqref{boxdot-def}.
Thus the functor $\overline{\phi}$ preserves the symmetric monoidal product
strictly. Finally, we would like to show that this
functor is uniquely defined. Let $G:\PNat(X) \to C$ be another a strict monoidal functor
  such that $\OplaxExp{\Leins(X)}{G} \circ i = \phi$. Then for every object $(\vec{n}, \vec{x})$ in $\PNat(X)$
 \[
 G ((\vec{n}, \vec{x})) = G \circ i(\vec{n})(\vec{x}) = \phi(\vec{n})(\vec{x})
 = \overline{\phi}(\vec{n},\vec{x}).
 \]
 A similar argument for morphisms of $\PNat(X)$ shows that G agrees with
 $\overline{\phi}$ on morphisms also. Thus we have proved that $\overline{\phi}$
 is a universal normalized lax symmetric monoidal section.
\end{proof}
\begin{nota}
As seen above we may \emph{compose} a bicycle with a functor to obtain another bicycle. More precisely, let $F:\PNat(X) \to C$ be a strict symmetric monoidal functor, then we will denote by $F \circ i$ the composite strict symmetric monoidal functor
\[
\Leins \overset{i} \to \SMSec{X}{\PNat(X)} \overset{\SMSec{X}{F}} \to \SMSec{X}{C}.
\]
This \emph{composition} defines a functor which we will denote by
$i^\ast:\StrSMHom{\PNat(X)}{C} \to \SMSec{X}{C}$.
\end{nota}
The above lemma and and argument similar to the proof of theorem \ref{oplax-SM-cones-isom} lead us to the following corollary:
\begin{coro}
\label{bike-SSM-isom-cat}
The functor $i^\ast:\StrSMHom{\PNat(X)}{C} \to \SMSec{X}{C}$ is
an isomorphism of categories which is natural in both $X$ and $C$.
\end{coro}

 The above corollary together with theorem \ref{oplax-SM-cones-isom} and proposition \ref{Bikes-as-sections} provide us with the following chain of isomorphisms of categories:
 \begin{multline}
 \label{isom-univ-bike}
 \Bikes{X}{C} \overset{J} \to \OplaxSec{X}{C} \overset{E} \to \\
  \SMSec{X}{C} \overset{i^\ast} \to \StrSMHom{\Pnor(X)}{C}.
 \end{multline}

Now we start the second step involved in establishing the adjunction $\PNat \dashv \Kbar$.
We want to define an oplax symmetric monoidal functor
$\epsilon:\N \to \PsExp{\Kbar(C)}{C}$. In order to do so we will define, for each $\underline{n} \in Ob(\N)$, a functor $\epsilon(\underline{n}):\Kbar(C)(n^+) \to C$. On objects this functor is defined as follows:
\begin{equation*}
\epsilon(\underline{n})(\Phi) := \Phi((id_{n^+}))
\end{equation*}
and on morphisms it is defined as follows:
\begin{equation*}
\epsilon(\underline{n})(F) := F(\underline{n})((id_{n^+})),
\end{equation*}
where $F: \Phi \to \Psi$ is a morphism in $\Kbar(C)$.
It is easy to see that the above definition preserves composition and identity in $\Kbar(C)(n^+)$. We recall that for each map $f:\underline{n} \to \underline{m}$ in $\N$ we get a functor $\PNat(f):\PNat(m^+) \to \PNat(n^+)$ which maps an object $(f_1, f_2, \dots, f_k) \in Ob(\PNat(m))$ to $(f_1 \circ f, f_2 \circ f, \dots, f_k \circ f) \in Ob(\PNat(n))$. The functor $\Kbar(C)(f):\Kbar(C)(n^+) \to \Kbar(m^+)$
is defined by precomposition \emph{i.e.} for each strict symmetric monoidal functor $\Phi:\PNat(n) \to \C$, $\Kbar(C)(f)(\Phi) := \Phi \circ \PNat(f)$.
 For each morphism $f:\underline{n} \to \underline{m}$ we will define a natural transformation $\epsilon(f): \epsilon(\underline{n})   \Rightarrow \epsilon(\underline{m}) \circ \Kbar(f)$. We recall that the identity map of $\underline{n}$ determines a map $can:(id_{n^+}) \to (f)$ in the category $\PNat(n)$ \emph{i.e.} the following diagram commutes
 \begin{equation*}
 \xymatrix{
 \textit{Supp}(id_{n^+}) = \underline{n} \ar[rd]_{id} \ar[rr]^{id} && \underline{n} = \textit{Supp}(f) \ar[ld]^{id} \\
 & \underline{n} 
 }
 \end{equation*}
For an object $\Phi \in \Kbar(n^+)$ we define
\begin{equation*}
\epsilon(f)(\Phi) := \Phi(can).
\end{equation*}
We observe that domain of $\Phi(can)$ is $\epsilon(\underline{n})(\Phi) = \Phi((id_{n^+}))$ and its codomain is $\epsilon(\underline{m})(\Kbar(C)(f)(\Phi)) =
 \Kbar(C)(f)(\Phi)(id_{m^+}) = \Phi((f))$.
 Let $F:\Phi \to \Psi$ be a morphism in $\Kbar(C)(n^+)$, then we have the following commutative diagram
\begin{equation}
\label{couniv-bike-alpha-nat-trans}
  \xymatrix@C=20mm{
 \Phi((id_{n^+}))  \ar[r]^{\epsilon(f)(\Phi)}
 \ar[d]_{F(\underline{n})((id_{n^+}))} & \Phi(f) \ar[d]^{F(\underline{m})(f)} \\
 \Psi((id_{n^+})) \ar[r]_{\epsilon(f)(\Psi)} & \Psi(f)
 }
\end{equation}
where we observe that the map $F(\underline{n})((id_{n^+}))$ is the same as $\epsilon(\underline{n})(F)$ and the map $F(\underline{m})(f)$ is the same as
$(\epsilon(m) \circ \Kbar(C)(f))(F)$.
Thus we have defined a natural transformations $\epsilon(f)$ for all $f \in Mor(\N)$. For another morphism $g:\underline{m} \to \underline{k}$
in the category $\N$ one can check that $\epsilon(g \circ f) = \epsilon(g) \circ \epsilon(f)$.
\begin{prop}
The functor $\epsilon$ defined above is an oplax symmetric monoidal section of $\PsExp{X}{C}$.
\end{prop}
let $X$ and $Y$ be a $\gCats$ and $C$ be a permutative category.
We will say that an oplax symmetric monoidal section $H:\N \to \PsExp{\PNat(X)}{C}$ is \emph{co-universal} if for any other oplax symmetric monoidal section $M:\N \to \PsExp{\PNat(Y)}{C}$ there exists a unique morphism of $\gCats$ $F:Y \to X$ such that the following diagram commutes
\begin{equation*}
\xymatrix{
\N \ar[r]^{H \ \ \ \ }\ar[rd]_M & \PsExp{\PNat(X)}{C}  \ar[d]^{\PsExp{\PNat(F)}{C}} \\
& \PsExp{\PNat(Y)}{C} 
}
\end{equation*}
 In the above situation we get a bijection
 \begin{equation*}
  Ob(\OLsmSec{Y}{C}) \cong Hom_{\gCAT}(Y, X).
\end{equation*}
 
 \begin{prop}
 The oplax symmetric monoidal functor $\epsilon:\N \to \OplaxExp{\Kbar(C)}{C}$ defined above is co-universal.
 \end{prop}

  \section{On local objects in a model category enriched over categories}
\label{Cat-Local}
\subsection{Introduction}
A model category $E$ is enriched over categories if the category $E$ is enriched(over $\Cat$), tensored and cotensored, and the functor $[- ; -]: E^{op} \times$E$\to \Cat$ is a Quillen functor of two variables, where $\Cat = (\Cat, Eq)$ is the natural model structure for categories. The purpose of this appendix is to introduce the notion of local object with respect to a map in a model category enriched over categories.
\subsection{Preliminaries}
Recall that a Quillen model structure on a category $E$ is determined by its class of cofibrations together with its class of fibrant objects. For examples, the category of simplicial sets $\sSets = [\Delta^{op},Set]$ admits two model structures in which the cofibrations are the monomorphisms: the fibrant objects are the Kan complexes in one, and they are the quasi-categories in the other. We call the former the model structure for Kan complexes and the latter the model structure for quasi-categories. We shall denote them respectively by $(\sSets, Kan)$ and $(\sSets, QCat)$.
In this appendix we consider categories enriched over $\Cat$.  If $E = (E, [- ; -])$ is a category enriched over $\Cat$, then so is the category $\CatFunc{E}{\Cat}$ of enriched functors $E \to \Cat$. An enriched functor $F :$E$\to \Cat$ isomorphic to the enriched functor $[A, -] :E \to \Cat$ is said to be \emph{representable}.  The enriched functor $F$ is said to be represented by  $A$. We say that an enriched (over $\Cat$) category $E = (E,[-,-])$ is \emph{tensored by} $I$, where $I$ is the category with two objects and one non-identity arrow, if the enriched functor
\[
[I, [A, -]] :E \to \Cat
\]
is representable (by an object denoted $I \otimes A$) for every object $A \in E$. 
Dually, we say that an enriched category $E$ is \emph{cotensored by} $I$ if the enriched functor 
\[
[I, [-, X]] : E^{op} \to \Cat
\]
is representable (by an object denoted $X^I$ or $[I, X]$) for every object $X \in E$.
\begin{df}
\label{enrich-QCat}
We shall say that a model category $E$ is \emph{enriched over categories} if the category $E$ is enriched over $\Cat$, tensored and cotensored over $I$ and the functor $[-,-] : E^{op} \times E \to \Cat$ is a Quillen functor of two variables, where $\Cat = (\Cat, Eq)$ \emph{i.e.} $\Cat$ is endowded with the natural model category structure.
\end{df}
\begin{nota}
	We will denote the homotopy mapping spaces or the function complexes of a model category $M$, see \cite{DK3}, \cite{Hovey}, \cite{Hirchhorn}, by $\HMapC{a}{b}{M}$, for each pair of objects $a, b \in M$.
	\end{nota}
\subsection{Function spaces for categories}
If $C$ is a category, we shall denote by $J(C)$ the sub-category of invertible arrows in $C$. The sub-category $J(C)$ is the largest sub-groupoid of $C$. More generally, if $X$ is a quasi-category, we shall denote by $J(X)$ the largest sub- Kan complex of $X$. By construction, we have a pullback square
\begin{equation*}
\xymatrix@C=20mm{
J(X) \ar[r] \ar[d] & X \ar[d]^h \\
NJ(\tau_1(X)) \ar[r] & N\tau_1(X)
}
\end{equation*}
where $\tau_1(X)$ is the fundamental category of $X$ and $h$ is the canonical map. The function space $X^A$ is a quasi-category for any simplicial set $A$. We shall denote by $X^{(A)}$ the full sub-simplicial set of $X^A$ whose vertices are the maps $A \to X$ that factor through the inclusion $J(X) \subseteq X$. The simplicial set $X^{(\Delta[1])}$ is a path-space for $X$. We recall that $\tau_1(X) \cong ho(X)$.

\begin{lem}
\label{J-Nerve-comm}
 If $C$ is a category, then the simplicial object $N(J(C)) \cong J(N(C))$,
 where $N:\Cat \to \sSets$ is the nerve functor. Further this isomorphism is natural in $C$, \emph{i.e.} there is a natural isomorphism between the two composite functors $NJ \cong JN$.
\end{lem}
\begin{proof}
    We recall that the Kan complex $J(N(C))$ is defined by the following pullback square:
	\begin{equation*}
	\xymatrix@C=20mm{
		J(N(C)) \ar[r] \ar[d] & N(C) \ar@{=}[d] \\
	   NJ(C) \cong N(J(\tau_1(N(C)))) \ar[r] & N(\tau_1(N(C))) \cong N(C)
	}
	\end{equation*}
	Since the above commutative diagram is a pullback diagram in which the right vertical arrow is the identity therefore we have the isomorphism $J(N(C)) \cong NJ(C)$. The second statement follows from the functorality of pullbacks.
\end{proof}

 The category $\Cat$ can be enriched over simplicial sets by defining
 \[
 \MapC{C}{D}{\Cat} := N([C, D]).
 \]
 Thus making $\Cat$ a simplicial category.
The adjunction $\tau_1:\sSets \rightleftharpoons \Cat:N$ makes $\Cat$ a $(\sSets, QCat)$-model category, see \cite[Prop. 6.14]{AJ1}. In other words the functor
$N([-;-]):\Cat^{\textit{op}} \times \Cat \to \sSets$ is a Quillen functor of two variables with respect to the natural model category structure on $\Cat$ and the Joyal model category structure on $\sSets$. We recall from \cite[Appendix B.3]{Sharma4} the the mapping space:
\begin{equation}
\label{map-sp-Cat}
\HMapC{C}{D}{\Cat} = J(\MapC{C}{D}{\Cat}) = J(N([C, D]))
\end{equation}
\subsection{Local objects}
Let $\Sigma$ be a set of maps in a model category $E$. An object $X \in E$ is said to be $\Sigma$-local if the map
\begin{equation*}
\HMapC{u}{X}{E} : \HMapC{A'} {X}{E} \to \HMapC{A}{ X}{E}
\end{equation*}
is a homotopy equivalence for every map $u:A \to A'$ in $\Sigma$. Notice that if an object $X$ is weakly equivalent to a $\Sigma$-local object, then $X$ is $\Sigma$-local. If the model category $E$ is simplicial (=enriched over Kan complexes) and $\Sigma$ is a set of maps between cofibrant objects, then a fibrant object $X \in E$ is $\Sigma$-local iff the map
$\MapC{u}{X}{\sSets} :\MapC{A'}{X}{\sSets} \to \MapC{A}{X}{\sSets}$ is a homotopy equivalence for every map $u : A \to A'$ in $\Sigma$,
where $\MapC{-}{-}{\sSets}:E^{\textit{op}} \times E \to \sSets$ is the mapping space functor providing the simplicial enrichment of the category $E$.

\begin{lem}
\label{char-lo-QCat-en}
 Let $E$ be a model category enriched over categories. If $u : A \to B$ is a map between cofibrant objects, then the following conditions on a fibrant object $X \in E$ are equivalent
 \begin{enumerate}
\item the map $[u, X] : [B, X] \to [A, X]$ is an equivalence of categories;
\item the object $X$ is local with respect to the pair of maps $\lbrace u, I \otimes u \rbrace$, where $I \otimes u : I \otimes A \to I \otimes B$.
\end{enumerate}
\end{lem}
\begin{proof}
 (1 $\Rightarrow$ 2) The map $[u, X] : [B, X] \to [A, X]$ is an equivalence of categories if and only if the following two maps
 \[
J([u, X]) : J([B, X]) \to J([A, X])
 \]
 and
 \[
 J([I, [u, X]]): J([I, [B, X]]) \to J([I, [A, X]])
 \]
 are equivalences of categories (or groupoids in this case). Since the model category $E$ is enriched over categories, there is a natural isomorphism $c_{B,X }:[I, [B, X] \cong [I \otimes B, X]$, \emph{i.e.} the following diagram is commutative
 \begin{equation*}
 \xymatrix@C=16mm{
 J([I, [B, X]]) \ar[r]^{J([I, [u, X]])} \ar[d]_{J(c_{B, X})} & J([I, [A, X]]) \ar[d]^{J(c_{A, X})} \\
 J([I \otimes B, X]) \ar[r]_{J([I \otimes u, X])} & J([I \otimes A, X])
}
 \end{equation*}
 for each pair of objects $A, B \in Ob(E)$.
 Now the two out of three property of model categories says that the bottom horizontal arrow in the diagram above $J([I \otimes u, X])$ is an equivalence of categories.
  The nerve functor takes an equivalence of categories to a categorical equivalence of simplicial sets (quasi-categories), see \cite[Prop. 6.14]{AJ1}. However a categorical equivalence between Kan complexes is a weak homotopy equivalence. Thus the two maps $N(J([I \otimes u, X]))$ and $N(J([ u, X]))$ are weak homotopy equivalences of simplicial sets. However from applying Lemma \ref{J-Nerve-comm} we deduce that the simplicial maps
 $J(N([I \otimes u, X])) = \HMapC{I \otimes u}{X}{\sSets}$ and $J(N([ u, X])) = \HMapC{u}{X}{\sSets}$ are weak homotopy equivalences of simplicial sets. In other words the object $X$ is local with respect to the pair of maps $\lbrace u, I \otimes u \rbrace$.
 
(1 $\Leftarrow$ 2) By assumption, the maps
\[
 J(N([u, X])):J(N([B, X])) \to J(N([A, X]))
 \]
  and
  \[
   J(N([I \otimes u, X])):J(N([I \otimes B, X])) \to J(N([I \otimes A, X]))
   \] 
   are weak homotopy equivalences of Kan complexes, therefore they are also categorical equivalences of simplicial sets. By lemma \ref{J-Nerve-comm} and the two out of three property of model categories, the simplicial maps $N(J([u, X]))$ and $N(J([I \otimes u, X]))$ are also categorical equivalences of simplicial sets. Since $\tau_1$ is a left Quillen functor and every object in $(\sSets, QCat)$ is cofibrant therefore it takes categorical equivalences to equivalences of categories, hence the maps $\tau_1N(J([u, X])) = J([u, X])$ and $\tau_1N(J([I \otimes u, X])) = J([I \otimes u, X])$ are equivalences of categories. Since for each small category $C$, the functor $[I, [C, -]]:E \to \Cat$ is representable by the object $I \otimes C$ \emph{i.e.} the functor $[I, [C, -]]$ is isomorphic to the functor $[I \otimes C, -]:E \to \Cat$, therefore the functor $J([I,[u, X]])$ is an equivalence of categories because $J([I \otimes u, X])$ is an equivalence of categories.
   Since both $J([u, X])$ and $J([I,[u, X]])$ are equivalences of categories
   therefore by lemma \ref{char-eq-cat}, the functor $[u, X]$ is an equivalence of categories.
\end{proof}



\section[oplax to SM functors]{From oplax to symmetric monoidal functors}
\label{OL-to-SM}
Throughout this paper we have been using a universal characteristic property of the permutative category $\Leins$ namely any oplax symmetric monoidal functor $F:\N \to M$, where $M$ is a permutative category extends uniquely to a symmetric monoidal functor $\Leins(F):\Leins \to M$. The objective of this section is to provide a proof of this universal property. We begin by understanding the Leinster's category $\Leins$ abstractly.

 The forgetful functor $U:\PCat \to \Cat$ has a left adjoint $S:\Cat \to \PCat$ which associates to a category $C$ the permutative category $S(C)$ freely generated by $C$ \cite{}. This
permutative category is defined as the following coproduct:
\begin{equation*}
S(C) := \underset{n \in \NatNum} \bigsqcup S^n(C),
\end{equation*}
where $S^n(C)$ is the \emph{symmetric n-power} of $C$. We observe that the $n$th symmetric group $\Sigma_n$ acts naturally on $C^n$ with the right action defined by
\begin{equation*}
\overline{x} \cdot \sigma := (x_{\sigma(1)}, x_{\sigma(2)}, \dots, x_{\sigma(n)}),
\end{equation*}
for $\overline{x} := (x_1, x_2, \dots, x_n) \in C^n$ and $\sigma \in \Sigma_n$.
If we apply the Grothendieck construction to this right action, we obtain the symmetric $n$-power
of $C$,
\begin{equation*}
S^n(C) := \Sigma_n \int C^n,
\end{equation*}
Explicitly, an object of $S^n(C)$ is a finite sequence $\overline{x} = (x_1, x_2, \dots, x_n)$ of
objects of $C$, and the hom set between $\overline{x}$ and $\overline{y}$ is defined as follows:
\begin{equation*}
S^n(C)[\overline{x}, \overline{y}] := \underset{\sigma \in \Sigma_n} \bigsqcup C(x_1, y_{\sigma(1)}) \times C(x_2, y_{\sigma(2)}) \times \cdots \times C(x_n, y_{\sigma(n)}).
\end{equation*}
The tensor product on $S(C)$ is defined by concatenation and the symmetry natural isomorphism is given by the shuffle permutation.
The unit of $S(C)$ is the empty sequence. There is an inclusion functor
\begin{equation*}
\iota_C:C \to S(C)
\end{equation*}
which takes an object $c \in Ob(C)$ to the one element sequence $(c) \in S(C)$. This functor exhibits $S(C)$ as the \emph{free permutative category} generated by $C$. More precisely, for every permutative category $\A$, the restriction functor
\begin{equation*}
\iota_C^*:\PCat(S(C), A) \to \Cat(C, A)
\end{equation*}
is a bijection. We observe that the object set of the category $\Leins$ is the following:
\begin{equation*}
Ob(\Leins) = Ob(S(\N)).
\end{equation*}

Let $P$ be a permutative category.
For each pair of functors $F_1, F_2:P \to \Sets$ one can define another functor which is denoted by $F_1 \ast F_2:P \to \Sets$ and is called the \emph{convolution product of $F_1$ with $F_2$} as follows:
\begin{equation*}
F_1 \ast F_2(m) := \int^{k \in \N} \int^{l \in P} F_1(k) \times F_2(l) \times P(k \underset{P} \otimes l, m).
\end{equation*}
The above construction defines a bifunctor
\begin{equation*}
- \ast - :[P; \Sets] \times [P; \Sets] \to [P; \Sets].
\end{equation*}
This bifunctor endows the functor category $[P; \Sets]$ with a monoidal structure whose unit object is the functor $P(0, -):P \to \Sets$. See \cite{Day-local}, \cite{Day2}, \cite{IK}.

The adjunction $(S, U)$ dicussed above has a counit map \emph{i.e.} for each permutative category $P$ there is a symmetric monoidal functor $\epsilon_P:S(P) \to P$. This counit functor provides us with the following composite functor
\begin{equation*}
P \overset{Y} \to [{P}^{op}; \Sets] \overset{\epsilon_{P}^\ast }\to [{S(P)}^{op}; \Sets]
\end{equation*}
where $Y$ is the Yoneda's embedding functor. We denote this composite functor by $\rho_P:P \to [{S(P)}^{op}; \Sets]$. For each $m \in Ob(P)$ we get a functor $\rho_P(m):{S(P)}^{op} \to \Sets$ which is defined as follows:
\begin{equation*}
\rho_P(m)((k_1, k_2, \dots, k_r)) := P(\epsilon_P(\vec{k}); m) = P(k_1 \underset{P} \otimes k_2 \underset{P} \otimes \cdots \underset{P} \otimes k_r; m)
\end{equation*}
where $\vec{k} = (k_1, k_2, \dots, k_r)$ is an object of the permutative category $S(P)$.
For another object $\vec{l} = (l_1, l_2, \dots, l_s)$ in $S(P)$ we define:
\begin{equation}
P^e(-, \vec{l}) := \rho_P(l_1) \ast \rho_P(l_2) \ast \cdots \ast \rho_P(l_s) = P(\epsilon_P(-); l_1) \ast P(\epsilon_\N(-); l_2) \ast \cdots \ast P(\epsilon_P(-); l_s).
\end{equation}
In other words the mapping set $P^e(\vec{k}, \vec{l})$ is the following coend:
\begin{equation*}
\int^{q_1 \in S(P)} \cdots  \int^{q_s \in S(P)} P(\epsilon_P(\vec{q_1}); l_1) \times \cdots \times  P(\epsilon_P(\vec{q_s}); l_s) \times S(P)(\vec{k}; q_1 \underset{S(P)} \otimes q_2 \underset{S(P)} \otimes \cdots \underset{S(P)} \otimes q_s).
\end{equation*}

The above iterated coend has a simple description:
 A map in the mapping set $P^e(\vec{k}, \vec{l})$ can be described as an $(s+1)$-tuple $(f_1, f_2, \dots, f_s; h)$, where $|\vec{l}| = s$, $h:\ud{r} \to \ud{s}$ is a map in $\N$ and $f_i:\underset{j \in \inv{h}({i})} \otimes k_j \to l_i$ is a map in $P$ for all $1 \le i \le s$. 
\begin{lem}
	The collection of mapping sets $\lbrace P^e(\vec{k}, \vec{l}) : \vec{k}, \vec{l} \in Ob(S(P))  \rbrace$ glue together to define a permutative category $P^e$ whose object set is the same as that of $S(P)$.
\end{lem}
\begin{proof}
	We begin the proof by defining composition in the category $P^e$. Let $(f_1, f_2, \dots, f_s; h): \vec{k} \to \vec{l}$, $(g_1, g_2, \dots, g_t; q): \vec{l} \to \vec{m} = (m_1, \dots, m_t)$ be two maps in $P^e$. We define their composite to be a map $(c_1, c_2, \dots, c_t; q \circ h): \vec{k} \to \vec{m}$ where the arrow $v_i$ is defined as follows:
	\[
	\underset{z \in \inv{q \circ h}(i)} \otimes k_z \overset{\textit {can}} \to \underset{j \in \inv{q}(i)} \otimes \underset{w \in \inv{h}(j)} \otimes k_w \overset{ \underset{j \in \inv{q}(i)} \otimes f_j} \to  \underset{j \in \inv{q}(i)} \otimes l_j \to m_i
	\]
	for $1 \le i \le t$. The isomorphism $\textit{can}$ in the above diagram is the unique isomorphism provided by the coherence theorem for symmetric monoidal categories between \emph{words} of the same length. The associativity of this composition follows from the coherence theorem for symmetric monoidal categories and the associativity of composition in the permutative category $P$.
	The permutative structure is given by concatenation.

	\end{proof}
\begin{rem}
	The permutative category $\Leins$ is isomorphic to the permutative category $((\N^{\textit{op}})^e)^{\textit{op}}$.
	\end{rem}


As above, let $\vec{k}$ and $\vec{l}$ be two objects of $P^e$ having lengths $|\vec{k}| = r$ and $|\vec{l}| = s$ respectively. We define a function
\[
\lambda_{\vec{k}, \vec{l}}:P^e(\vec{k}, \vec{l}) \to P(\underset{i = 1}{\overset{r} \otimes} k_i; \underset{j = 1}{\overset{s} \otimes} l_j).
\]
as follows:
\[
\lambda_{\vec{k}, \vec{l}}((f_1, f_2, \dots, f_s; h)) := (\underset{i = 1}{\overset{s} \otimes}{\underset{j \in \inv{h}(i)} \otimes f_j}) \circ \textit{can}
\]

\begin{lem}
The collection of functions $\lbrace \lambda_{\vec{k}, \vec{l}} : \vec{k}, \vec{l} \in Ob(P^e)  \rbrace$ glue together to define a strict symmetric monoidal functor
\begin{equation}
\label{section-inc-N-L}
\lambda_P:P^e \to P
\end{equation}
whose objects function maps an object $\vec{k} = (k_1, k_2, \dots, k_r)$ to $k_1 \underset{P} \otimes k_2 \underset{P} \otimes \cdots \underset{P} \otimes k_r$ in $Ob(P)$.
\end{lem}
\begin{proof}
	We have to check that the function defined above respects compsition in the category $P^e$. Let $(f_1, f_2, \dots, f_s; h):\vec{k} \to \vec{l}$ and $(g_1, g_2, \dots, g_t; q):\vec{l} \to \vec{m} = (m_1, m_2. \dots, m_t)$ be two composable arrows in the category $P^e$. We will show the following equality:
	\begin{multline}
	\lambda_{\vec{l}, \vec{m}}((f_1, f_2, \dots, f_s; h)) \circ \lambda_{\vec{k}, \vec{l}}((g_1, g_2, \dots, g_t; q)) = \\	 
	\lambda_{\vec{k}, \vec{m}}((g_1, g_2, \dots, g_t; q) \circ (f_1, f_2, \dots, f_s; h))
	\end{multline}
	This is equivalent to showing that the following diagram commutes:
	\begin{equation*}
	\xymatrix{
	\underset{i=1} {\overset{t} \otimes } \underset{z \in \inv{q \circ h}(i)} { \otimes} k_z  \ar[d]_{\textit{can}} & \underset{i=1} {\overset{r} \otimes} k_i \ar[d]^{\textit{can}} \ar[l]_{\textit{can}} \\
		\underset{i=1} {\overset{t} \otimes } \underset{j \in \inv{q}(i)} \otimes 
	\underset{w \in \inv{h}(j)} \otimes k_w  \ar[d]_{\underset{i=1} {\overset{t} \otimes } \underset{j \in \inv{q}(i)} \otimes f_j} & \underset{i=1} {\overset{s} \otimes } \underset{j \in \inv{h}(i)} \otimes k_j \ar[d]^{\underset{i=1} {\overset{s} \otimes } \underset{j \in \inv{h}(i)} \otimes f_j}  \ar[l]^{\textit{ \ \  can}} \\
	\underset{i=1} {\overset{t} \otimes } \underset{j \in \inv{q}(i)} \otimes l_j  \ar[d]_{\underset{j \in \inv{q}(i)} \otimes g_i} \ar@{=}[rd]  & \underset{i=1} {\overset{s} \otimes } l_i \ar[d]^{\textit{can}} \ar[l]^{\textit{can}}  \\
	\underset{i=1} {\overset{t} \otimes} m_i  & \underset{i=1} {\overset{t} \otimes } \underset{j \in \inv{q}(i)} \otimes l_j \ar[l]^{\underset{j \in \inv{q}(i)} \otimes g_i}
	 }
 \end{equation*}
 The composite of the top left arrow and the three left vertical arrows is the same as $\lambda_{\vec{l}, \vec{m}}((f_1, f_2, \dots, f_s; h)) \circ \lambda_{\vec{k}, \vec{l}}((g_1, g_2, \dots, g_t; q))$ and the composite of the right three vertical arrows and the bottom horizontal arrows is $\lambda_{\vec{k}, \vec{m}}((g_1, g_2, \dots, g_t; q) \circ (f_1, f_2, \dots, f_s; h))$. The bottom square is obviously commutative because both composite arrows are composites of the same two arrows. The top rectangle is made up of canonical isomorphisms provided by the coherence theorem for symmetric monoidal categories. Now the commutativity of the top rectangle follows from \cite[Corollary 1.6]{JS}
which implies that any diagram made of canonical coherence maps commutes. The commutativity of the middle rectangle follows from the naturality of the canonical isomorphisms provided by the coherence theorem for symmetric monoidal categories.

	\end{proof}
We recall that the bifunctor $- \otimes -:P \times P \to P$ providing the permutative structure of a permutative category $P$ extends uniquely to a functor
\[
- \overset{r} \otimes -:\Prod{1}{r} P \to P,
\]
for all $r \in \Nat$.
Similarly any natural transformation $\eta: (- \otimes -) \circ (F \times F) \Rightarrow F \circ (- \otimes -)$ extends uniquely to a natural transformation $\eta_r$ as shown in the diagram below:
\begin{equation*}
\xymatrix{
\Prod{i}{r} P \ar[d]_{\Prod{i}{r} F} \ar[r]^{- \overset{r} \otimes -} & P \ar[d]^F \ar@{=>}[ld]^{\eta_r} \\
\Prod{i}{r} D \ar[r]_{- \overset{r} \otimes -} & D
}
\end{equation*}
\begin{nota}
	We will refer to the functor $- \overset{r} \otimes -$ as the $r$-fold extension of the permutative structure of $P$ and refer to $\eta_r$ as the $r$-fold extension of $\eta$.
	\end{nota}
Any lax symmetric monoidal functor $(F, \mu, \eta):P \to D$ determines a strict symmetric monoidal functor
\begin{equation}
\label{Lax-to-str}
F^e:P^e \to D^e
\end{equation}
The object function of $F^e$ is defined as follows:
\[
\vec{k} = (k_1, k_2, \dots, k_r) \mapsto (F(k_1), F(k_2), \dots, F(k_r)).
\]
Given another object $\vec{l} = (l_1, \dots, l_s)$ in $P$ we define a map
\[
F_{\vec{k}, \vec{l}}:P^e(\vec{k}, \vec{l}) \to D^e(F(\vec{k}), F(\vec{l}))
\]
as follows:
\[
(f_1, f_2, \dots, f_s; h) \mapsto (g_1, g_2, \dots, g_s; h),
\]
where the map $g_i:\underset{j \in \inv{h}(i)} \otimes F(k_j) \to F(l_i)$ is the following composite:
\[
 \underset{j \in \inv{h}(i)} \otimes F(k_j) \overset{\mu_{r_i}} \to F(\underset{j \in \inv{h}(i)} \otimes k_j) \overset{F(f_i)} \to F(l_i),
\]
where $r_i = \inv{h}(i)$. An application of the coherence theorem for symmetric monoidal categories shows that $F^e$ is a functor and it is easy to see that it preserves the permutative structure.


\begin{thm}
	The symmetric lax monoidal inclusion functor $\iota:P \to P^e$ is universal: for any symmetric permutative category $D$ and a symmetric lax monoidal functor $\phi:P \to D$ there exists a unique strict symmetric monoidal functor $\psi:P^e \to D$ such that $\psi \circ \iota = \phi$:
	\begin{equation*}
	\xymatrix{
	P \ar[r]^\iota \ar[rd]_{\phi} & P^e \ar[d]^\psi \\
	& D
    }
	\end{equation*}
	\end{thm}
\begin{proof}
	We define the functor $\psi:P \to D$ to be the following composite:
	\[
	P^e \overset{\phi^e} \to D^e \overset{\lambda_D} \to D
	\]
  The uniqueness of this functor is an easy consequence of the symmetric monoidal structure on the permutative category $P^e$ and the strict symmetric monoidal nature of the functor $\psi$.

	\end{proof}

An Oplax symmetric monoidal functor $F:P \to D$ uniquely determines a Lax symmetric monoidal functor between the opposite categories namely $F^{op}:P^{op} \to D^{op}$. This duality provides us with the following corollary:
\begin{coro}
	\label{extension-to-strmon}
	Let $P$ be a strict symmetric monoidal category, then there exists another strict symmetric monoidal category $P_e$ which is equipped with an inclusion functor $\iota:P \to P_e$ which is universal: For any strict symmetric monoidal category $D$ and an oplax symmetric monoidal functor $F:P \to D$ there exists a unique strict symmetric monoidal functor $\psi:P_e \to D$ such that $\psi \circ \iota = F$ \emph{i.e.} the following diagram commutes:
	\textbf{\begin{equation*}
		\xymatrix{
			P \ar[r]^\iota \ar[rd]_{F} & P_e \ar[d]^\psi \\
			& D
		}
		\end{equation*}}
	The strict symmetric monoidal category $P_e$ is isomorphic to $((P^{\textit{op}})^e)^{\textit{op}}$.
	\end{coro}

 \bibliographystyle{amsalpha}
 \bibliography{EinGammaCat}

\providecommand{\bysame}{\leavevmode\hbox to3em{\hrulefill}\thinspace}
\providecommand{\MR}{\relax\ifhmode\unskip\space\fi MR }
\providecommand{\MRhref}[2]{%
  \href{http://www.ams.org/mathscinet-getitem?mr=#1}{#2}
}
\providecommand{\href}[2]{#2}
\begin{thebibliography}{AGV72}

\bibitem[AGV72]{SGA4}
M.~Artin, A.~Grothendieck, and J.~L. Verdier, \emph{Theorie des topos et
  cohomologie etale des schemas ({SGA} 4)}, Lecture Notes in Mathematics.
  \textbf{270} (1972).

\bibitem[AR94]{AR94}
J.~Adamek and J.~Rosicky, \emph{Locally presentable and accessible categories},
  London Mathematical Society Lecture Note Series, Cambridge University Press,
  1994.

\bibitem[Ara14]{ara2014}
Dimitri Ara, \emph{Higher quasi-categories vs higher rezk spaces}, Journal of
  K-theory \textbf{14} (2014), no.~3, 701–749.

\bibitem[Bar07]{CB1}
C.~Barwick, \emph{On (enriched) left bousfield localization of model
  categories}, \url{arXiv:0708.2067}, 2007.

\bibitem[BF78]{BF78}
A.~K. Bousfield and E.~M. Friedlander, \emph{Homotopy theory of
  {$\Gamma$}-spaces, spectra and bisimplicial sets.}, Geometric applications of
  homotopy theory II, Lecture Notes in Math. (1978), no.~658.

\bibitem[BM03]{BM1}
C.~Berger and I.~Moerdijk, \emph{Cofibrant operads and universal
  {$E_\infty$}-operads}, Comment. Math. Helv. \textbf{78} (2003).

\bibitem[Day70]{Day2}
B.~Day, \emph{On closed categories of functors, reports of the midwest category
  seminar {IV}}, Lecture notes in Mathematics, vol. 137, Springer-Verlag, 1970.

\bibitem[Day73]{Day-local}
\bysame, \emph{Note on monoidal localisation}, Bull. Austr. Math. Soc.
  \textbf{8} (1973).

\bibitem[DK80a]{DK80}
W.~G. Dwyer and D.~M. Kan, \emph{Calculating simplicial localizations}, Journal
  of Pure and Appl. Algebra \textbf{18} (1980), 17--35.

\bibitem[DK80b]{DK3}
\bysame, \emph{Function complexes in homotopical algebra}, Topology \textbf{19}
  (1980), 427--440.

\bibitem[DK80c]{DK1980}
\bysame, \emph{Simplicial localizations of categories}, Journal of Pure and
  Appl. Algebra \textbf{17} (1980), 267--284.

\bibitem[Dun94]{Dunn}
Gerald Dunn, \emph{En-monoidal categories and their group completions}, J. Pure
  Appl. Algebra \textbf{95} (1994), 27--39.

\bibitem[EM06]{mandell2}
A.~D. Elmendorf and M.~A. Mandell, \emph{Permutative categories,
  multicategories and algebraic {K}-theory}, Alg. and Geom. Top. \textbf{9}
  (2006), no.~4, 163--228.

\bibitem[Gam08]{NG}
N.~Gambino, \emph{Homotopy limits for 2-categories}, Mathematical Proceedings
  of the Cambridge Philosophical Society \textbf{145} (2008), no.~1, 43–63.

\bibitem[GJ08]{GA1}
N.~Gambino and A.~Joyal, \emph{On operads, bimodules and analytic functors},
  Mem. Amer. Math. Soc. \textbf{249} (2008), no.~1184.

\bibitem[GJO17]{GJO}
N.~Gurski, N.~Johnson, and A.~M. Osorno, \emph{Extending homotopy theories
  across adjunctions}, Homology, Homotopy and Applications \textbf{19} (2017),
  no.~2, 89--110.

\bibitem[GS07]{goer-sch}
P.~G. Goerss and K.~Schemmerhorn, \emph{Model categories and simplicial
  methods}, Contemp. Math. \textbf{436} (2007), 3--49.

\bibitem[GZ67]{GZ}
P.~Gabriel and M.~Zisman, \emph{Calculus of fractions and homotopy theory},
  Ergebnisse der Mathematik imd ihrer Grenzgebiete, vol.~35, Springer-Verlag,
  1967.

\bibitem[Hir02]{Hirchhorn}
Phillip~S. Hirchhorn, \emph{Model categories and their localizations},
  Mathematical Surveys and Monographs, vol.~99, Amer. Math. Soc., Providence,
  RI, 2002.

\bibitem[Hov99]{Hovey}
M.~Hovey, \emph{Model categories}, Mathematical Surveys and Monographs,
  vol.~63, Amer. Math. Soc., Providence, RI, 1999.

\bibitem[Joy08]{AJ1}
A.~Joyal, \emph{Theory of quasi-categories and applications},
  \url{http://mat.uab.cat/~kock/crm/hocat/advanced-course/Quadern45-2.pdf},
  2008.

\bibitem[JS93]{JS}
A.~Joyal and R.~Street, \emph{Braided tensor categories}, Adv. in Math.
  \textbf{19} (1993), 20--78.

\bibitem[JT08]{JT1}
A.~Joyal and M.~Tierney, \emph{Notes on simplicial homotopy theory},
  \url{http://mat.uab.cat/~kock/crm/hocat/advanced-course/Quadern47.pdf}, 2008.

\bibitem[KS15]{KS}
D.~Kodjabachev and S.~Sagave, \emph{Strictly commutative models for
  {$E_\infty$} quasi-categories}, Homology Homotopy Appl. \textbf{17} (2015).

\bibitem[Lac07]{Lack}
S.~Lack, \emph{Homotopy-theoretic aspects of 2-monads.}, Journal of Homotopy
  and Related Structures \textbf{2} (2007), no.~2, 229--260 (eng).

\bibitem[Lei00]{Leinster}
T.~Leinster, \emph{Homotopy algebras for operads}, Arxiv \textbf{126} (2000).

\bibitem[Lur09]{JL}
Jacob Lurie, \emph{Higher topos theory}, Annals of Mathematics Studies, vol.
  170, Princeton University Press, Princeton, NJ, 2009.

\bibitem[Lyd99]{Lydakis}
M.~Lydakis, \emph{Smash products and {$\Gamma$}- spaces}, Math. Proc. Camb.
  Soc. \textbf{126} (1999).

\bibitem[Mac71]{MacL}
S.~MacLane, \emph{Categories for the working mathematician}, Springer-Verlag,
  1971.

\bibitem[Man10]{mandell}
M.~A. Mandell, \emph{An inverse {K}-theory functor}, Doc. Math. \textbf{15}
  (2010), 765--791.

\bibitem[May72]{May72}
P.~May, \emph{The geometry of iterated loop spaces}, Lectures Notes in
  Mathematics, vol. 271, Springer-Verlag, 1972.

\bibitem[May78]{May4}
J.~P. May, \emph{The spectra associated to permutative categories}, Topology
  \textbf{17} (1978), no.~3, 225--228.

\bibitem[Rez10]{rezk2010}
Charles Rezk, \emph{A cartesian presentation of weak $n$–categories}, Geom.
  Topol. \textbf{14} (2010), no.~1, 521--571.

\bibitem[Sch99]{schwede}
S.~Schwede, \emph{Stable homotopical algebra and {$\Gamma$}- spaces}, Math.
  Proc. Camb. Soc. \textbf{126} (1999), 329.

\bibitem[Sch08]{Schmitt}
V.~Schmitt, \emph{Tensor product for symmetric monoidal categories},
  arXiv/0812.0150 (2008).

\bibitem[Seg74]{segal}
G.~Segal, \emph{Categories and cohomology theories}, Topology \textbf{13}
  (1974), 293--312.

\bibitem[Sha]{Sharma4}
A.~Sharma, \emph{The homotopy theory of coherently commutative monoidal
  quasi-categories}, \url{arXiv:1908.05668}.

\bibitem[Shaon]{sharma3}
\bysame, \emph{Symmetric multicategories as distributors}, In Preparation.

\bibitem[Smi]{smith}
J.~Smith, \emph{Combinatorial model categories}, unpublished.

\bibitem[SS79]{SK}
N.~Shimada and K.~Shimakawa, \emph{Delooping symmetric monoidal categories},
  Hiroshima Math. J. \textbf{9} (1979), 627--645.

\bibitem[Tho80]{Thomason2}
R.~W. Thomason, \emph{Cat as a closed model category}, Cahiers Topologie Geom.
  Differential \textbf{21} (1980), no.~3, 305--324.

\bibitem[Tho95]{Thomason}
\bysame, \emph{Symmetric monoidal categories model all connective spectra},
  Theory Appl. Categ. \textbf{1} (1995), no.~5, 78--118.

\end{thebibliography}

\end{document}